\documentclass[a4paper, 12pt, reqno,final]{amsart}

\usepackage{amsmath, amssymb}
\usepackage{enumerate}
\usepackage{ccaption} 
\usepackage{array}
\usepackage{multirow}
\usepackage[all]{xy} 
\usepackage[margin=3.7cm]{geometry}

\usepackage{showkeys}
\usepackage{comment}

\newtheorem{Thm}{Theorem}[section]
\newtheorem{Lem}[Thm]{Lemma}
\newtheorem{Prop}[Thm]{Proposition}

\newtheorem{ThmA}{Theorem}

\theoremstyle{remark}
\newtheorem{Rem}{Remark}

\theoremstyle{definition}

\theoremstyle{definition}
\newtheorem{Ex}[Thm]{Example}
\newtheorem{Recipe}[Thm]{Recipe}

\newcommand{\BS}[1]{\boldsymbol{#1}}
\newcommand{\MC}[1]{\mathcal{#1}}
\newcommand{\MF}[1]{\mathfrak{#1}}
\newcommand{\MRM}[1]{\mathrm{#1}}
\newcommand{\OPE}[1]{\operatorname{#1}} 

\newcommand{\PreserveBackslash}[1]{\let\temp=\\#1\let\\=\temp}
\let\PBS=\PreserveBackslash 

\newlength{\LENGTHH}
\newlength{\LENGTHTHETA}
\title[Local orbit type]{Local orbit types of the isotropy representations for
semisimple pseudo-Riemannian symmetric spaces}
\author{Kurando Baba}
\date{\today}

\begin{document}
\maketitle

\begin{abstract}
We list up all the possible local orbit types
of hyperbolic
or elliptic orbits for the isotropy representations of semisimple
pseudo-Riemannian symmetric spaces.
It is key to
give a recipe to determine the local orbit types of hyperbolic
principal orbits by using three kind of 
restricted root systems and Satake diagrams
associated with semisimple pseudo-Riemannian symmetric spaces.
\end{abstract}

\section*{Introduction}

Let $G/H$ be a semisimple pseudo-Riemannian symmetric space.
The isotropy representation of $G/H$ is called an \textit{s-representation}.
In this paper, we investigate
all the possible local orbit types
(i.e., the conjugate classes of isotropy subalgebras)
of hyperbolic or elliptic orbits
for the s-representation of $G/H$
by using three kind of restricted root system associated with $G/H$.
An orbit is said to be \textit{hyperbolic principal}
(resp.\,\textit{elliptic principal})
if it is a hyperbolic orbit (resp.\,an elliptic orbit)
whose local orbit type is the smallest one among
the hyperbolic orbits (resp.\, the elliptic orbits).
We also investigate the local orbit type
of a hyperbolic principal orbit and an elliptic principal orbit
by using the three kind of Satake diagrams associated with $G/H$.
The present work is based on the paper \cite{B}.
In 1992, Heintze and Olmos (\cite{MR1023346}) determined
the isotropy subalgebras of the orbits in the case
where $G/H$ is Riemannian.
Moreover,
in 2007,
Boumuki (\cite{MR2370009}) gave
the isotropy subalgebras of elliptic orbits
in the case where $G/H$ is a semisimple Lie group.

\medskip

We state the main result of this paper.
Let $G/H$ be a semisimple pseudo-Riemannian symmetric space.
Denote by $\MF{g}$ (resp.~$\MF{h}$)
the Lie algebra of $G$ (resp.~$H$).
Let $\sigma$ be an involution of $\MF{g}$ whose
fixed point set coincides with $\MF{h}$.
Set $\MF{q}:=\{X \in \MF{g}\mid\sigma(X)=-X\}$,
which is identified with the tangent space of $G/H$ at $eH$.
Here $e$ is the identity element of $G$.
Let
$\MF{a}$ be a vector-type maximal split abelian subspace of
$\MF{q}$.
Denote by $\Delta$ the restricted root system
of $G/H$ with respect to $\MF{a}$.
Let $\varPsi$ be a simple root system of $\Delta$,
and
$\Delta_{+}$ be the positive root system of $\Delta$
with respect to $\varPsi$.
Set,
for any $\Theta \subset \varPsi$,
$\MF{h}_{\Theta}:=
\MF{z}_{\MF{h}}+\sum_{\lambda \in \Delta_{\Theta}\cap\Delta_{+}}\MF{h}_{\lambda}$,
where
$\MF{z}_{\MF{h}}$ denotes the centralizer of $\MF{a}$ in $\MF{h}$,
$\MF{h}_{\lambda}$ denotes the root subspace of $\MF{h}$
for $\lambda$, and
$\Delta_{\Theta}:=\Delta\cap\sum_{\lambda \in \Theta}\BS{R}\lambda$.
Denote by $[\MF{h}_{\Theta}]$ the conjugate class of $\MF{h}_{\Theta}$,
and by $(m^{+}(\lambda),m^{-}(\lambda))$
the signature of $\lambda \in \Delta$.
Let $\MC{W}(\Delta)$ be the Weyl group of $\Delta$
and $\MC{W}(\Delta^{a})$ be the Weyl group of $\Delta^{a}$,
where $\Delta^{a} := \{\lambda \in \Delta\mid m^{+}(\lambda)>0\}$.
Then we have the following result.

\begin{ThmA}\label{thm.main}
Let $w_{1},w_{2},\ldots,w_{l}$ be a complete system
of representatives for $\MC{W}(\Delta)/\MC{W}(\Delta^{a})$.
The set of all local orbit types of
the hyperbolic orbits for the s-representation of $G/H$
coincides with
\begin{equation}\label{obj.local}
\bigcup^{l}_{i=1}\big\{[\MF{h}_{\Theta}]\mid\Theta \subset w_{i} \cdot \varPsi\big\},  
\end{equation}
where $w_{i} \cdot \varPsi$ denotes the set of all the functions $w_{i}\cdot \lambda (\lambda \in \varPsi)$ defined by $w_{i}\cdot \lambda(A)=\lambda(w^{-1}_{i}(A))$ for all $A \in \MF{a}$.
\end{ThmA}

\noindent
A major difficulty in indefinite case
arises from the fact that
$\MC{W}(\Delta)$ is not necessarily isomorphic to
$N_{H}(\MF{a})/Z_{H}(\MF{a})$,
where
$N_{H}(\MF{a})$ (resp.\,$Z_{H}(\MF{a})$)
denotes the normalizer (resp.\,the centralizer) of $\MF{a}$ in $H$.
Note that $w_{i}$'s 
($w_{i} \in \MC{W}(\Delta)/\MC{W}(\Delta^{a})$)
are not necessarily
preserve the signatures of the roots in $\varPsi$ invariantly,
so that, for $1 \leq i\neq j \leq l$,
$\{[\MF{h}_{\Theta}]\mid \Theta \subset w_{i}\cdot \varPsi\}$
and $\{[\MF{h}_{\Theta}]\mid \Theta \subset w_{j}\cdot \varPsi\}$
are not necessarily equal to each other
(cf.\,Example \ref{ex.isotro1}).
In positive definite case,
$\MC{W}(\Delta)$ is equal to $\OPE{Ad}(H)|_{\MF{a}}$
and therefore $l=1$ holds.
For classical-type semisimple pseudo-Riemannian symmetric spaces,
we give the lists of the indices of $\MC{W}(\Delta)/\MC{W}(\Delta^{a})$
(cf.\,Table \ref{table.index}).
From Theorem \ref{thm.main},
in order to determine the set of local orbit types,
it is sufficient to determine $\{[\MF{h}_{\Theta}]\mid\Theta \subset w_{i}\cdot\varPsi\}$ for each $i \in \{1,2,\ldots,l\}$.
In \cite{B},
we gave a recipe to determine
$\{[\MF{h}_{\Theta}]\mid\Theta \subset \Pi\}$ for a simple
root system $\Pi$ of $\Delta$.
It follows from our recipe
that for each $\Theta \subset \Pi$,
$\MF{h}_{\Theta}$ corresponds
to the subdiagram of the Dynkin diagram associated with $\Pi$,
and is determined by 
the hyperbolic principal isotropy subalgebra and
a semisimple subsymmetric pair of $(\MF{g}, \MF{h})$
associated with $\Delta_{\Theta}$ (see page 315 in \cite{B}).
By applying our recipe to $w_{i}\cdot\varPsi$ for each $i \in \{1,\ldots,l\}$,
we can determine (\ref{obj.local}).
Our recipe is analogous to determine the local orbit
types of the orbits for the isotropy action on
Riemannian symmetric spaces of compact type by Tamaru (\cite{MR1702475}).

\begin{Rem}
Similarly, the set of all local orbit types of the elliptic orbits
for the s-representation is also determined by
using the restricted root system with respect to
a troidal-type maximal split abelian subspace.
\end{Rem}

The isotropy subalgebra of a hyperbolic
principal orbit is isomorphic to $\MF{z}_{\MF{h}}$,
which is called a \textit{hyperbolic principal isotropy subalgebra}
(abbreviated to HPIS).
Suppose that $\theta$ is a Cartan involution of $\MF{g}$
commuting with $\sigma$
and $\MF{a}$ is a subspace of $\MF{p}$,
where $\MF{p}:=\OPE{Ker}(\theta+\OPE{id})$.
Set $\MF{k}:=\OPE{Ker}(\theta-\OPE{id})$.
Let $\MF{a}_{\MF{q}}$ (resp.\,$\MF{a}_{\MF{p}}$)
be a maximal abelian subspace of $\MF{q}$ (resp.\,$\MF{p}$)
containing $\MF{a}$.
By using the Satake diagrams associated with
$G/H$ with respect to $\MF{a}$, $\MF{a}_{\MF{q}}$ and $\MF{a}_{\MF{p}}$,
we investigate the (Lie algebra) structure of $\MF{z}_{\MF{h}}$.
Moreover, we give a recipe to determine $\MF{z}_{\MF{h}}$
(cf.\,Recipe \ref{recipe.hprin} in Section \ref{sec.hprin}).
We obtained the these Satake diagrams for classical-type
semisimple pseudo-Riemannian symmetric spaces in \cite{B2}.
Then we have the following result in terms of Table 1 in \cite{B2}.

\begin{ThmA}
The hyperbolic principal
isotropy subalgebras of the s-represen-tations
associated with all classical-type semisimple pseudo-Riemannian symmetric spaces
are as in Table \ref{table.hprin}.
\end{ThmA}

\begin{table}[htbp]
\footnotesize
\caption{The hyperbolic principal isotropy subalgebras}\label{table.hprin}
\centering
\begin{tabular}{|l|c|c|}
\hline
$(\MF{g},\MF{h})$& HPIS & Remarks\\
\hline
\hline
$(\MF{sl}(n,\BS{C}),\MF{sl}(n,\BS{R}))$
&$\BS{R}^{[(n-1)/2]}+\MF{so}(2)^{[n/2]}$&\\\hline
$(\MF{sl}(n,\BS{R})^{2},\MF{sl}(n,\BS{R}))$
&$\BS{R}^{n-1}$&\\\hline
$(\MF{sl}(n,\BS{C}),\MF{so}(n,\BS{C}))$
&$\{0\}$&\\\hline
$(\MF{sl}(2n,\BS{C}),\MF{su}^{*}(2n))$
&$\BS{R}^{n-1}+\MF{so}(2)^{n}$&\\\hline
$(\MF{su}^{*}(2n)^{2},\MF{su}^{*}(2n))$
&$\BS{R}^{n-1}+\MF{sp}(1)^{n}$&\\\hline
$(\MF{sl}(2n,\BS{C}),\MF{sp}(n,\BS{C}))$
&$\MF{sp}(1,\BS{C})^{n}$&\\\hline
$(\MF{sl}(n,\BS{C}),\MF{su}(p,n-p))$
&$\MF{so}(2)^{n-1}$&\\\hline
\multirow{2}{*}{$(\MF{su}(p,n-p)^{2},\MF{su}(p,n-p))$}&
$\BS{R}^{p}+\MF{so}(2)^{p}+\MF{su}(n-2p)$&$n>2p$\\\cline{2-3}
&$\BS{R}^{p}+\MF{so}(2)^{p-1}$&$n=2p$\\\hline
\multirow{2}{*}{$(\MF{sl}(n,\BS{C}),\MF{sl}(p,\BS{C})+\MF{sl}(n-p,\BS{C})+\BS{C})$}&
$\BS{C}^{p}+\MF{sl}(n-2p,\BS{C})$&$n>2p$\\\cline{2-3}
&$\BS{C}^{p-1}$&$n=2p$\\\hline
$(\MF{so}(2n,\BS{C}),\MF{so}^{*}(2n))$
&$\MF{so}(2)^{n}$&\\\hline
\multirow{2}{*}{$(\MF{so}^{*}(2n)^{2},\MF{so}^{*}(2n))$}&
$\BS{R}^{m}+\MF{su}(2)^{m}+\MF{so}(2)$&$n=2m+1$\\\cline{2-3}
&$\BS{R}^{m}+\MF{su}(2)^{m}$&$n=2m$\\\hline
\multirow{2}{*}{$(\MF{so}(2n,\BS{C}),\MF{sl}(n,\BS{C})+\BS{C})$}&
$\MF{sl}(2,\BS{C})^{m}+\BS{C}$&$n=2m+1$\\\cline{2-3}
&$\MF{sl}(2,\BS{C})^{m}$&$n=2m$\\\hline
\end{tabular}
\end{table}

\begin{table}[htbp]
\footnotesize
\contcaption{(continued)}
\centering
\begin{tabular}{|l|c|c|}
\hline
$(\MF{g},\MF{h})$& HPIS & Remarks\\
\hline
\hline
\multirow{6}{*}{$(\MF{so}(n,\BS{C}),\MF{so}(p,n-p))$}&
\multirow{2}{*}{$\MF{so}(2)^{m}$}&$n=2m+1$\\
&&$p=2q$\\\cline{2-3}
&\multirow{2}{*}{$\MF{so}(2)^{m}+\BS{R}$}&$n=2m$\\
&&$p=2q+1$\\\cline{2-3}
&\multirow{2}{*}{$\MF{so}(2)^{m}$}&$n=2m$\\
&&$p=2q$\\\hline
$(\MF{so}(p,n-p)^{2},\MF{so}(p,n-p))$&
$\BS{R}^{p}+\MF{so}(n-2p)$&\\\hline
$(\MF{so}(n,\BS{C}),\MF{so}(p,\BS{C})+\MF{so}(n-p,\BS{C}))$&
$\MF{so}(n-2p,\BS{C})$&\\\hline
$(\MF{sp}(n,\BS{C}),\MF{sp}(n,\BS{R}))$
&$\MF{so}(2)^{n}$&\\\hline
$(\MF{sp}(n,\BS{R})^{2},\MF{sp}(n,\BS{R}))$
&$\BS{R}^{n}$&\\\hline
$(\MF{sp}(n,\BS{C}),\MF{sl}(n,\BS{C})+\BS{C})$
&$\{0\}$&\\\hline
$(\MF{sp}(n,\BS{C}),\MF{sp}(p,n-p))$
&$\MF{so}(2)^{n}$&\\\hline
$(\MF{sp}(p,n-p)^{2},\MF{sp}(p,n-p))$&
$\BS{R}^{p}+\MF{sp}(1)^{p}+\MF{sp}(n-2p)$&\\\hline
$(\MF{sp}(n,\BS{C}),\MF{sp}(p,\BS{C})+\MF{sp}(n-p,\BS{C}))$&
$\MF{sp}(1,\BS{C})^{p}+\MF{sp}(n-2p,\BS{C})$&\\\hline
$(\MF{sl}(n,\BS{R}),\MF{so}(p,n-p))$
&$\{0\}$&\\\hline
$(\MF{su}(p,n-p),\MF{so}(p,n-p))$&
$\MF{so}(n-2p)$&\\\hline
\multirow{2}{*}{$(\MF{sl}(n,\BS{R}),\MF{sl}(p,\BS{R})+\MF{sl}(n-p,\BS{R})+\BS{R})$}&
$\BS{R}^{p}+\MF{sl}(n-2p,\BS{R})$&$n>2p$\\\cline{2-3}
&$\BS{R}^{p-1}$&$n=2p$\\\hline
$(\MF{su}^{*}(2n),\MF{sp}(p,n-p))$
&$\MF{sp}(1)^{n}$&\\\hline
$(\MF{su}(2p,2(n-p)),\MF{sp}(p,n-p))$&
$\MF{sl}(2,\BS{C})^{p}+\MF{sp}(n-2p)$&\\\hline
\multirow{2}{*}{$(\MF{su}^{*}(2n),\MF{su}^{*}(2p)+\MF{su}^{*}(2(n-p))+\BS{R})$}&
$\BS{R}^{p}+\MF{sp}(1)^{p}+\MF{su}^{*}(2(n-2p))$&$n>2p$\\\cline{2-3}
&$\BS{R}^{p-1}+\MF{sp}(1)^{p}$&$n=2p$\\\hline
$(\MF{sl}(2n,\BS{R}),\MF{sp}(n,\BS{R}))$
&$\MF{sp}(1,\BS{R})^{n}$&\\\hline
$(\MF{su}^{*}(2n),\MF{so}^{*}(2n))$
&$\MF{u}(1)^{n}$&\\\hline
$(\MF{su}(n,n),\MF{so}^{*}(2n))$
&$\{0\}$&\\\hline
$(\MF{sl}(2n,\BS{R}),\MF{sl}(n,\BS{C})+\MF{so}(2))$
&$\BS{R}^{n-1}$&\\\hline
\multirow{2}{*}{$(\MF{su}^{*}(2n),\MF{sl}(n,\BS{C})+\MF{so}(2))$}&
$\BS{R}^{m}+\MF{su}(2)^{m}+\MF{so}(2)$&$n=2m+1$\\\cline{2-3}
&$\BS{R}^{m-1}+\MF{su}(2)^{m}$&$n=2m$\\\hline
\multirow{2}{*}{$(\MF{su}(n,n),\MF{sp}(n,\BS{R}))$}&
$\MF{sp}(1,\BS{C})^{m}+\MF{sp}(1,\BS{R})$&$n=2m+1$\\\cline{2-3}
&$\MF{sp}(1,\BS{C})^{m}$&$n=2m$\\\hline
$(\MF{su}(n,n),\MF{sl}(n,\BS{C})+\BS{R})$
&$\MF{so}(2)^{n-1}$&\\\hline
\multirow{6}{*}{$(\MF{so}^{*}(2n),\MF{su}(p,n-p)+\MF{so}(2))$}&
\multirow{2}{*}{$\MF{su}(2)^{m}+\MF{so}(2)$}&$n=2m+1$\\
&&$p=2q$\\\cline{2-3}
&\multirow{2}{*}{$\MF{su}(2)^{m}+\MF{so}(2)$}&$n=2(m+1)$\\
&&$p=2q+1$\\\cline{2-3}
&\multirow{2}{*}{$\MF{su}(2)^{m}$}&$n=2m$\\
&&$p=2q$\\\hline
$(\MF{so}(2p,2(n-p)),\MF{su}(p,n-p)+\MF{so}(2))$&
$\MF{su}(1,1)^{p}+\MF{u}(n-2p)$&\\\hline
\end{tabular}
\end{table}

\begin{table}[htbp]
\footnotesize
\contcaption{(continued)}
\centering
\begin{tabular}{|l|c|c|}
\hline
$(\MF{g},\MF{h})$& HPIS & Remarks\\
\hline
\hline
$(\MF{so}^{*}(2n),\MF{so}^{*}(2p)+\MF{so}^{*}(2(n-p)))$&
$\MF{so}(2)^{p}+\MF{so}^{*}(2(n-2p))$&\\\hline
$(\MF{so}(n,n),\MF{so}(n,\BS{C}))$
&$\{0\}$&\\\hline
$(\MF{so}^{*}(2n),\MF{so}(n,\BS{C}))$
&$\MF{so}(2)^{[n/2]}$&\\\hline
\multirow{2}{*}{$(\MF{so}(n,n),\MF{sl}(n,\BS{R})+\BS{R})$}&
$\BS{R}+\MF{sl}(2,\BS{R})^{m}$&$n=2m+1$\\\cline{2-3}
&$\MF{sl}(2,\BS{R})^{m}$&$n=2m$\\\hline
$(\MF{so}^{*}(4n),\MF{su}^{*}(2n)+\BS{R})$
&$\MF{sp}(1)^{n}$&\\\hline
$(\MF{sp}(n,\BS{R}),\MF{su}(p,n-p)+\MF{so}(2))$
&$\{0\}$&\\\hline
$(\MF{sp}(p,n-p),\MF{su}(p,n-p)+\MF{so}(2))$&
$\MF{u}(1)^{p}+\MF{u}(n-2p)$&\\\hline
$(\MF{sp}(n,\BS{R}),\MF{sp}(p,\BS{R})+\MF{sp}(n-p,\BS{R}))$&
$\MF{sp}(1,\BS{R})^{p}+\MF{sp}(n-2p,\BS{R})$&\\\hline
$(\MF{sp}(n,\BS{R}),\MF{sl}(n,\BS{R})+\BS{R})$
&$\{0\}$&\\\hline
$(\MF{sp}(n,n),\MF{sp}(n,\BS{C}))$
&$\MF{sp}(1)^{n}$&\\\hline
$(\MF{sp}(2n,\BS{R}),\MF{sp}(n,\BS{C}))$
&$\MF{sp}(1,\BS{R})^{n}$&\\\hline
$(\MF{sp}(n,n),\MF{su}^{*}(2n)+\BS{R})$
&$\MF{u}(1)^{n}$&\\\hline
\end{tabular}
\end{table}

\begin{table}[htbp]
\footnotesize
\centering
\begin{tabular}{|>{\PBS\centering}p{70mm}|>{\PBS\centering}p{30mm}|}
\multicolumn{2}{l}{
$(\MF{g},\MF{h})=(\MF{su}(n,m),\MF{su}(i,j)+\MF{su}(n-i,m-j)+\MF{so}(2))$}\\
\hline
HPIS & Remarks\\
\hline
\hline
$\MF{so}(2)^{n-1}$&$i+j=n=m$\\\hline
$\MF{so}(2)^{n}+\MF{su}(m-n)$&$n<i+j=m$\\\hline
$\MF{so}(2)^{m+n-(i+j)}+\MF{su}(i+j-n,i+j-m)$&$n\leq m<i+j$\\\hline
$\MF{so}(2)^{n}+\MF{su}(m-n)$&$n=i+j<m$\\\hline
$\MF{so}(2)^{n+1}+\MF{su}(i+j-n)+\MF{su}(m-(i+j))$&$n<i+j<m$\\\hline
$\MF{so}(2)^{i+j}+\MF{su}(n-(i+j),m-(i+j))$&$i+j<n\leq m$\\\hline
\end{tabular}
\end{table}

\begin{table}[htbp]
\footnotesize
\centering
\begin{tabular}{|>{\PBS\centering}p{70mm}|>{\PBS\centering}p{30mm}|}
\multicolumn{2}{l}{
$(\MF{g},\MF{h})=(\MF{so}(n,m),\MF{so}(i,j)+\MF{so}(n-i,m-j))$}\\
\hline
HPIS & Remarks\\
\hline
\hline
$\{0\}$&$i+j=n=m$\\\hline
$\MF{so}(m-n)$&$n<i+j=m$\\\hline
$\MF{so}(i+j-n,i+j-m)$&$n\leq m<i+j$\\\hline
$\MF{so}(m-n)$&$n=i+j<m$\\\hline
$\MF{so}(i+j-n)+\MF{so}(m-(i+j))$&$n<i+j<m$\\\hline
$\MF{so}(n-(i+j),m-(i+j))$&$i+j<n\leq m$\\\hline
\end{tabular}
\end{table}

\begin{table}[htbp]
\footnotesize
\centering
\begin{tabular}{|>{\PBS\centering}p{70mm}|>{\PBS\centering}p{30mm}|}
\multicolumn{2}{l}{
$(\MF{g},\MF{h})=(\MF{sp}(n,m),\MF{sp}(i,j)+\MF{sp}(n-i,m-j))$}\\
\hline
HPIS & Remarks\\
\hline
\hline
$\MF{sp}(1)^{n}$&$i+j=n=m$\\\hline
$\MF{sp}(1)^{n}+\MF{sp}(m-n)$&$n<i+j=m$\\\hline
$\MF{sp}(1)^{m+n-(i+j)}+\MF{sp}(i+j-n,i+j-m)$&$n\leq m<i+j$\\\hline
$\MF{sp}(1)^{n}+\MF{sp}(m-n)$&$n=i+j<m$\\\hline
$\MF{sp}(1)^{n}+\MF{sp}(i+j-n)+\MF{sp}(m-(i+j))$&$n<i+j<m$\\\hline
$\MF{sp}(1)^{i+j}+\MF{sp}(n-(i+j),m-(i+j))$&$i+j<n\leq m$\\\hline
\end{tabular}
\end{table}

\clearpage

\noindent
It is a difficult problem to determine the structures
of the HPISs, since these are not necessarily compact.
By using the Satake diagram associated with $G/H$ with respect to $\MF{a}$,
we determine the structure of $\MF{z}_{\MF{g}}^{\BS{C}}$,
where $\MF{z}_{\MF{g}}$ denotes the centralizer
of $\MF{a}$ in $\MF{g}$.
Moreover, we investigate the decomposition
$\MF{z}_{\MF{g}}=\MF{z}_{\MF{g}}\cap\MF{h}+\MF{z}_{\MF{g}}\cap\MF{q}$
and $\MF{z}_{\MF{g}}=\MF{z}_{\MF{g}}\cap\MF{k}+\MF{z}_{\MF{g}}\cap\MF{p}$
by using the Satake diagrams associated with $G/H$ with respect to
$\MF{a}$, $\MF{a}_{\MF{q}}$ and $\MF{a}_{\MF{p}}$.
In positive definite case,
the isotropy subalgebras are compact,
so that,
in order to investigate the structures of the isotropy subalgebras,
it is sufficient to investigate the dimension
of their centers and the Dynkin diagrams of their semisimple parts
by using the Satake diagram with respect to $\MF{a}$.

\medskip

The organization of this paper is as follows.
In Section \ref{Sec.pre}, we give preliminaries
for the restricted root systems with respect to maximal split abelian subspaces
for semisimple pseudo-Riemannian symmetric spaces.
Moreover,
we recall the notion of the Satake diagrams.
In Section \ref{Sec.proof},
we prove Theorem \ref{thm.main}.
In Section \ref{sec.weyl},
we shall give a complete representatives for
$\MC{W}(\Delta)/\MC{W}(\Delta^{a})$ in the case
where $G/H$ is of classical-type
except for $(\Delta,\Delta^{a}) = (A_{n-1},A_{p-1}\times A_{n-p-1})$.
In the case of $(\Delta,\Delta^{a})=(A_{n-1},A_{p-1}\times A_{n-p-1})$,
we give a recursive formula for $\MC{W}(\Delta)/\MC{W}(\Delta^{a})$
(Proposition \ref{lem.fmula1}).
By using the formula (\ref{fmula1}),
we give a complete representative for $\MC{W}(\Delta)/\MC{W}(\Delta^{a})$
in the case of $(\Delta,\Delta^{a})=(A_{4},A_{1}\times A_{2})$
(Example \ref{ex.weyl1}).
In Section \ref{sec.hprin},
we shall give a recipe to determine the (Lie algebra) structures
of HPISs (Recipe \ref{recipe.hprin}).
Moreover,
we give examples for $(\MF{g},\MF{h})=(\MF{su}(2p,2(n-p)),\MF{sp}(p,n-p)),\,(\MF{sl}(n,\BS{C}),\MF{sl}(n,\BS{R}))$ (Example \ref{exam.1} and \ref{exam.2}).
In Section \ref{Sec.det},
we investigate the isotropy subalgebras of
the hyperbolic orbits
for the s-representation of
a semisimple pseudo-Riemannian symmetric space $G/H$.
We give all the possible
local orbit types of hyperbolic orbits
for the s-representation associated with
$(\MF{sl}(4, \BS{R}),\MF{so}(2,2))$
(Example \ref{ex.isotro1}).
We also give all the possible
local orbit types of hyperbolic orbits
for the s-representations
associated with all classical semisimple pseudo-Riemannian symmetric
spaces in the case where
$\Delta=\Delta^{a}$, $(\Delta, \Delta^{a})=((BC)_{r},B_{r})$ or
$(\Delta,\Delta^{a})=(C_{r}, D_{r})$ (Example \ref{ex.local1} and \ref{ex.local2}).
In Section \ref{Sec.ellip},
we discuss the relation between
the local orbit types of the elliptic orbits
and those of the hyperbolic orbits.

\medskip

\paragraph{\textbf{Research plan in the future.}}
We will explicitly give a standard complete system of representatives for $\MC{W}(\Delta)/\MC{W}(\Delta^{a})$
for exceptional-type semisimple pseudo-Riemannian symmetric spaces.
Moreover, we will give HPISs for these spaces.
We will develop the submanifold geometry of orbits for
s-representations by using their isotropy subalgebras.

\section{Preliminaries}\label{Sec.pre}

Let $G$ be a connected semisimple noncompact Lie group,
and $\sigma$ be an involution of $G$.
Let $H$ be a closed subgroup of $G$ with
$(G_{\sigma})_{0} \subset H \subset G_{\sigma}$,
where $G_{\sigma}$ denotes the fixed point group
of $\sigma$ and $(G_{\sigma})_{0}$ denotes
the identity component of $G_{\sigma}$.
The pair $(G, H)$ is called a \textit{semisimple symmetric pair}.
Then the coset space $G/H$ equipped with
the metric induced from the Killing form of
the Lie algebra $\MF{g}$ of $G$
is a semisimple pseudo-Riemannian symmetric space.
The holonomy representation of $G/H$
is equivalent to the s-representation.
Denote by $\OPE{Ad_{\mathit{G}}}$ (resp.~$\OPE{ad_{\MF{g}}}$)
the adjoint representation of $G$ (resp.~$\MF{g}$).
The involution $\sigma$ of $G$ induces an involution of $\MF{g}$,
which is also denoted by the same symbol $\sigma$.
Then the Lie algebra $\MF{h}$ of $H$ coincides with
$\{X \in \MF{g}\mid \sigma(X)=X\}$.
The pair $(\MF{g},\MF{h})$ is called a \textit{semisimple symmetric pair}.
Set $\MF{q} := \{X \in \MF{g}\mid\sigma(X) = -X\}$.
The s-representation is equivalent to
the representation $\OPE{Ad}$ of $H$ on $\MF{q}$
defined by $\OPE{Ad}(h) = \OPE{Ad}_{G}(h)|_{\MF{q}}$
for all $h \in H$.
An element $X \in \MF{q}$ is said to be \textit{semisimple}
if the endomorphism $\OPE{ad}_{\MF{g}}(X)^{\BS{C}}$ is
diagonalizable,
where $\OPE{ad}_{\MF{g}}(X)^{\BS{C}}$ is the complexification
of $\OPE{ad}_{\MF{g}}(X)$.
A semisimple element $X \in \MF{q}$ is said to be \textit{hyperbolic}
(resp.\,\textit{elliptic}) if any eigenvalue of $\OPE{ad}_{\MF{g}}(X)^{\BS{C}}$
is real (resp.~pure imaginary).
An orbit through a hyperbolic (resp.~an elliptic) element
is called a \textit{hyperbolic orbit} (resp.\,an \textit{elliptic orbit}).
Denote by $H_{X}$ the isotropy subgroup
of $H$ at $X \in \MF{q}$, by $\MF{h}_{X}$ the Lie algebra of $H_{X}$.
It is clear that
$H_{\OPE{Ad}(h)X}=hH_{X}h^{-1}$
and
$\MF{h}_{\OPE{Ad}(h)X}=\OPE{Ad}_{G}(h)\MF{h}_{X}$hold for all $h \in H$.
The conjugate classes $\{hH_{X}h^{-1}\mid h\in H\}$
and $\{\OPE{Ad}_{G}(h)\MF{h}_{X}\mid h\in H\}$ are called
an \textit{orbit type} and a \textit{local orbit type},
respectively.

We recall the notion of the restricted root systems with respect to
maximal split abelian subspaces for semisimple pseudo-Riemannian symmetric spaces (cf.\,\cite{MR518716}, \cite{MR810638}).
Let $\MF{a}$ be a maximal split abelian subspace of $\MF{q}$
(i.e., a maximal abelian subspace of $\MF{q}$ which consists of only
hyperbolic elements or only elliptic elements).
We say that $\MF{a}$ is \textit{vector-type} (resp.\,\textit{troidal-type})
if all elements of $\MF{a}$ are hyperbolic (resp.~elliptic).
Set, for any $\lambda \in \MF{a}^{*}$, 
\begin{align*}
\MF{g}_{\lambda} &:= \{X \in \MF{g}\mid\OPE{ad}(A)X=(\sqrt{-1})^{\epsilon}
\lambda(A)X,\forall A \in \MF{a}\},\\
\MF{h}_{\lambda} &:= \{X \in \MF{h}\mid\OPE{ad}(A)^{2}X =
(-1)^{\epsilon}\lambda(A)^{2}X, \forall A \in \MF{a}\},\\
\MF{q}_{\lambda} &:= \{X \in \MF{q}\mid\OPE{ad}(A)^{2}X =
(-1)^{\epsilon}\lambda(A)^{2}X, \forall A \in \MF{a}\},
\end{align*}
where $\epsilon = 0$ (resp.~$1$) when $\MF{a}$
is vector-type (resp.~troidal-type).
Denote by $\Delta := \{\lambda \in \MF{a}^{*}\setminus\{0\}\mid\MF{g}_{\lambda}
\neq \{0\}\}=
\{\lambda \in
\MF{a}^{*}\setminus\{0\}\mid\MF{q}_{\lambda}\neq \{0\}\}$,
which is called the \textit{restricted root system} of $G/H$
(or $(\MF{g},\MF{h})$) with respect to $\MF{a}$.
The dimension of $\MF{a}$ is called the split rank
of $G/H$ (or $(\MF{g},\MF{h})$),
which is denoted by $\OPE{s\mathchar`-rank}(G/H)$
(or $\OPE{s\mathchar`-rank}(\MF{g},\MF{h})$).
Note that, for each $\lambda \in \Delta$,
the restriction of the Killing form of $\MF{g}$ to
$\MF{q}_{\lambda} \times \MF{q}_{\lambda}$
is a nondegenerate inner product of $\MF{q}_{\lambda}$,
whose index denotes $m^{-}(\lambda)$.
Set, for each $\lambda \in \Delta$,
$m(\lambda):=\dim \MF{q}_{\lambda}$
and $m^{+}(\lambda) := m(\lambda) - m^{-}(\lambda)$.
We call $m(\lambda)$ and $(m^{+}(\lambda),m^{-}(\lambda))$
the \textit{multiplicity} of $\lambda$ and the \textit{signature} of $\lambda$,
respectively.
A symmetric pair $(\MF{g}, \MF{h})$
is called \textit{basic} if $m^{+}(\lambda)\geq m^{-}(\lambda)$
for any $\lambda \in \Delta$ such that $\frac{1}{2}\lambda \not\in \Delta$.
Let $\varepsilon$ be a signature of $\Delta$, i.e.,
$\varepsilon$ is a mapping from $\Delta$ to $\{\pm 1\}$
satisfying the two conditions: (i)
$\varepsilon(\lambda+\mu) = \varepsilon(\lambda)\varepsilon(\mu)\,
(\lambda,\,\mu,\,\lambda+\mu \in \Delta)$, and (ii)
$\varepsilon(-\lambda) = \varepsilon(\lambda)
\,(\lambda \in \Delta)$.
Denote by $\sigma_{\varepsilon}$
the involution of $\MF{g}$ defined by
\[
\sigma_{\varepsilon}(X)=
\begin{cases}
\sigma(X) & (X \in \MF{z}_\MF{g}),\\
\varepsilon(\lambda)\sigma(X)
& (X \in \MF{g}_{\lambda},\lambda \in \Delta).
\end{cases}
\]

\noindent
Then $(\MF{g}, \MF{h}_{\varepsilon})$
is a semisimple symmetric pair,
where $\MF{h}_{\varepsilon}=\OPE{Ker}(\sigma_{\varepsilon}-\OPE{id})$.
Denote by $\MF{z}_{\MF{g}}$, $\MF{z}_{\MF{h}}$ and $\MF{z}_{\MF{q}}$
the centralizer of $\MF{a}$ in $\MF{g}$,
$\MF{h}$ and $\MF{q}$, respectively.
Then we have the following decompositions.

\begin{Lem}
Let $\Delta_{+}$ be a positive root system of $\Delta$.
\[
\MF{g}=\MF{z}_{\MF{g}}+\sum_{\lambda \in \Delta}\MF{g}_{\lambda},\quad
\MF{h} = \MF{z}_{\MF{h}} + \sum_{\lambda \in
 \Delta_{+}}\MF{h}_{\lambda},\quad
\MF{q} = \MF{z}_{\MF{q}} + \sum_{\lambda \in
 \Delta_{+}}\MF{q}_{\lambda}.
\]
\end{Lem}

\noindent
Denote by $s_{\lambda}$ the reflection of $\MF{a}$ along
the hyperplane $\lambda^{-1}(0)$,
and
$\MC{W}(\Delta)$ by the Weyl group of $\Delta$
(i.e., the group generated by $s_{\lambda}$'s ($\lambda \in \Delta$)).
Let $\varPsi$ be a simple root system of $\Delta$.
Then $\{A \in \MF{a}\mid\lambda(A)>0, \forall \lambda \in \varPsi\}$
is a Weyl chamber associated with $\Delta$,
and $\MC{W}(\Delta)$ acts simply transitively
on the set of the Weyl chambers associated with $\Delta$.

Next, we recall three Satake diagrams associated with three kind of 
restricted root systems associated with
semisimple pseudo-Riemannian symmetric spaces (cf.\,\cite{B2}).
For simplicity,
in the sequel, suppose that $\MF{a}$ is a vector-type maximal
split abelian subspace of $\MF{q}$.
It is known that $\MF{a}$ is vector-type
if and only if $\MF{a}$ is a maximal abelian subspace of $\MF{p}\cap\MF{q}$,
where $\MF{p}$ is the $(-1)$-eigenspace
of certain Cartan involution commuting with $\sigma$.
Fix a such Cartan involution $\theta$.
Set $\MF{k}:=\{X \in \MF{g}\mid\theta(X)=X\}$.
Let $\MF{a}_{\MF{q}}$ (resp.~$\MF{a}_{\MF{p}}$)
be a maximal abelian subspace of $\MF{q}$ (resp.~$\MF{p}$)
containing $\MF{a}$.
Let $\tilde{\MF{a}}$ be a maximal abelian subalgebra of $\MF{g}$
containing $\MF{a}_{\MF{q}}$ and $\MF{a}_{\MF{p}}$.
Then $\tilde{\MF{a}}^{\BS{C}}$ is a Cartan subalgebra of $\MF{g}^{\BS{C}}$.
Denote by $R$ the root system of $\MF{g}^{\BS{C}}$
with respect to $\tilde{\MF{a}}^{\BS{C}}$.
We define the vector $A_{\alpha} (\alpha \in R)$
of $\tilde{\MF{a}}^{\BS{C}}$ by
$B(A,A_{\alpha}) = \alpha(A)$ for all $A \in
\tilde{\MF{a}}^{\BS{C}}$,
where $B$ denotes the Killing form of $\MF{g}^{\BS{C}}$.
Set $\tilde{\MF{a}}_{\BS{R}} := \OPE{Span}_{\BS{R}}\{A_{\alpha} \mid \alpha
\in R\}.$
Then we have
$\tilde{\MF{a}}_{\BS{R}}= \sqrt{-1}\tilde{\MF{a}}\cap
(\MF{k}\cap\MF{h}) + \MF{a}_{\MF{p}}\cap\MF{h}+
\sqrt{-1}\MF{a}_{\MF{q}}\cap\MF{k} + \MF{a}$.
Fix a basis $\MC{A} := (A_{1},A_{2},\ldots,A_{r})$ of $\tilde{\MF{a}}_{\BS{R}}$
such that $(A_{1},A_{2},\ldots,A_{l})$ is a basis of $\MF{a}$,
$(A_{l+1},A_{l+2}\ldots,A_{m})$ is a basis of
$\sqrt{-1}(\MF{a}_{\MF{q}}\cap\MF{k})$,
$(A_{m+1},A_{m+2},\ldots,A_{m+n-l})$ is a basis of
$\MF{a}_{\MF{p}}\cap\MF{h}$, and
$(A_{m+n-l+1},\ldots,A_{r})$ is a basis of
$\sqrt{-1}\tilde{\MF{a}}\cap(\MF{k}\cap\MF{h})$,
where $l:=\OPE{s\mathchar`-rank}(\MF{g},\MF{h})$,
$m:=\OPE{rank}(\MF{g},\MF{h})$, $n:=\OPE{rank}(\MF{g},\MF{k})$ and
$r:=\OPE{rank}(\MF{g}^{\BS{C}})$.
Take the lexicographic ordering $>$ of
the dual space of $\tilde{\MF{a}}_{\BS{R}}$ with respect to $\MC{A}$.
Denote by $\varPhi$ the simple root system of $R$
with respect to $>$.
From the Dynkin diagram of $R$ associated with $\varPhi$
we construct the \textit{Satake diagram}
$S(G/H,\MF{a})$ (or $S(\MF{g},\MF{h},\MF{a})$)
associated with $G/H$ with respect to $\MF{a}$ as follows.
First,
replace white circles in the Dynkin diagram,
which imply elements in $\varPhi_{0}:=\{\alpha \in \varPhi\mid\alpha|_{\MF{a}}= 0\}$,
to a black circle.
Second,
if simple roots $\alpha, \beta \in \varPhi \setminus \varPhi_{0}$
and $\alpha|_{\MF{a}} =\beta|_{\MF{a}}$,
then join $\alpha$ to $\beta$ with a curved arrow.
Similarly, we construct the Satake diagrams
$S(G/H, \MF{a}_{\MF{q}})$
(or $S(\MF{g},\MF{h}, \MF{a}_{\MF{q}})$)
associated with $G/H$ with respect to $\MF{a}_{\MF{q}}$.
Note that the Satake diagram with respect to $\MF{a}_{\MF{p}}$
coincides with the Satake diagram $S(G/K, \MF{a}_{\MF{p}})$
of the Riemannian symmetric space $G/K$ with respect to
$\MF{a}_{\MF{p}}$,
where $K$ denotes the Lie subgroup of $G$ with Lie algebra $\MF{k}$.
By the definition of $>$,
$S(G/H,\MF{a})$, $S(G/H, \MF{a}_{\MF{q}})$
and $S(G/K,\MF{a}_{\MF{p}})$
are constructed from the Dynkin diagram of the same simple
root system of $\varPhi$.
Then $S(G/H,\MF{a})$, $S(G/H, \MF{a}_{\MF{q}})$
and $S(G/K,\MF{a}_{\MF{p}})$ are said to be \textit{compatible}
with one another.
Set
$\varPhi_{0,\MF{p}} :=\{\alpha \in R \mid \alpha|_{\MF{a}_{\MF{p}}}=0\}$.
Denote by $p_{\theta}$
the Satake involution of $S(G/K,\MF{a}_{\MF{p}})$.
By the classification of the Satake diagrams of
irreducible Riemannian symmetric spaces,
we have the following fact.

\begin{Lem}\label{lem.satake1}
Let $G/K$ be an irreducible Riemannian symmetric space.\break
Then the Satake diagram of $G/K$ has the following two properties.
\begin{enumerate}[$(1)$]
\item If $\beta \in \varPhi \setminus \varPhi_{0,\MF{p}}$ is
disconnected with the roots of $\varPhi_{0,\MF{p}}$,
then $(-\theta)\beta = p_{\theta}\beta$ holds.
\item If $\beta \in \varPhi \setminus \varPhi_{0,\MF{p}}$
is connected with a root of $\varPhi_{0,\MF{p}}$,
then $(-\theta)\beta \equiv p_{\theta}\beta$
$\OPE{mod}\, \{\langle\beta\rangle\}_{\BS{Z}}$,
where $\langle\beta\rangle$ $(\subset \varPhi_{0,\MF{p}})$ denotes
the union of the connected components of $\varPhi_{0,\MF{p}}$
connected with $\beta$.
\end{enumerate}
\end{Lem}

\section{Proof of Theorem \ref{thm.main}}\label{Sec.proof}

Let $(\MF{g}, \MF{h})$ be a semisimple symmetric pair,
and $\sigma$ be a involution of $\MF{g}$
with $\MF{h}=\OPE{Ker}(\sigma-\OPE{id})$.
Suppose that $\MF{a}$ is a vector-type maximal split abelian subspace
of $\MF{q}\,(:= \OPE{Ker}(\sigma + \OPE{id}))$.
Denote by $\Delta$ the restricted root system of $(\MF{g}, \MF{h})$
with respect to $\MF{a}$.

\begin{proof}[Proof of Theorem \ref{thm.main}]

Let $\theta$ be a Cartan involution of $\MF{g}$
commuting with $\sigma$,
and $\MF{g}=\MF{k}+\MF{p}$ be the Cartan decomposition
corresponding to $\theta$.
Any vector-type maximal split abelian subspace is $\OPE{Ad}(H)$-conjugate
to a $\theta$-invariant one.
Without loss of generality, we may assume that
$\MF{a}$ is a maximal abelian subspace of $\MF{p}\cap\MF{q}$.
By Lemma 2 of \cite{MR518716},
if $A \in \MF{a}$ and $\OPE{Ad}(h)A \in \MF{a}\,(h \in H)$,
then there exists a $k \in N_{H\cap K}(\MF{a})$
such that $\OPE{Ad}(h)A = \OPE{Ad}(k)A$,
where $N_{H\cap K}(\MF{a}):=\{k \in H\cap K\mid\OPE{Ad}(k)\MF{a} = \MF{a}\}$.
Therefore $\OPE{Ad}(H)X$ and $\OPE{Ad}(H)Y\,(X,Y \in \MF{a})$
are same orbit if and only if $X$ is conjugate in $Y$ in $N_{H\cap K}(\MF{a})$.
Set $Z_{H\cap K}(\MF{a}):=\{h \in H\cap K\mid\OPE{Ad}(h)A=A, \forall A \in \MF{a}\}$.
Then the quotient group $N_{H\cap K}/Z_{H\cap K}$
is isomorphic to the Weyl group
$\MC{W}(\Delta^{a})$ associated with $\Delta^{a}$.
Let $w_{1},w_{2},\ldots,w_{l}$ be a complete system
of representatives for $\MC{W}(\Delta)/\MC{W}(\Delta^{a})$.
Then we can show that
$\MF{a}=\bigcup^{l}_{i=1}\overline{C(w_{i}\cdot \varPsi)}$ and
that
$
\bigcup_{i=1}^{l}\OPE{Ad}(H)(\overline{C(w_{i}\cdot \varPsi)})
$
coincides with the set of all the hyperbolic elements in $\MF{q}$,
where $C(w_{i}\cdot \varPsi):= \{A \in \MF{a}\mid\lambda(A) >0, \forall \lambda \in w_{i}\cdot \varPsi\}$ and $\overline{C(w_{i}\cdot \varPsi)}$ denotes the
closure of $C(w_{i}\cdot \varPsi)$.
Hence we have
\begin{align*}
\MC{L}_{h}(G/H)
&:=\{[\MF{h}_{A}]\mid\text{ $A \in \MF{q}$ is hyperbolic}\}\\
&=\bigcup_{i=1}^{l}\big\{[\MF{h}_{A}]\,\big|\,A \in \overline{C(w_{i}\cdot \varPsi)}\big\}.
\end{align*}

\noindent
For each $\Theta \subset w_{i}\cdot \varPsi$,
set $C(w_{i}\cdot\varPsi, \Theta) := \{A \in \MF{a}\mid\lambda(A)>0(\forall \lambda \in (w_{i}\cdot \varPsi)\setminus \Theta),\mu(A)=0(\forall \mu \in \Theta)\}
(\subset \overline{C(w_{i}\cdot \varPsi)})$.
Then, for any two subsets $\Theta, \Theta'$ of $w_{i}\cdot \varPsi$,
we have $C(w_{i}\cdot \varPsi, \Theta) \cap C(w_{i}\cdot\varPsi, \Theta')=\emptyset$ if $\Theta \neq \Theta'$.
Therefore $\overline{C(w_{i}\cdot \varPsi)}$ is decomposed 
as
\[
\overline{C(w_{i}\cdot \varPsi)} = \bigcup_{\Theta \subset w_{i}\cdot \varPsi}C(w_{i}\cdot\varPsi,\Theta)\quad(\MRM{disjoint}).\]
Moreover,
for any $A \in C(w_{i}\cdot\varPsi,\Theta)$,
$\Theta$ is a simple root system of $\Delta_{A}(:=\{\lambda \in \Delta\mid\lambda(A)=0\})$.
This implies that, for all $A \in C(w_{i}\cdot\varPsi, \Theta)$,
\[
\MF{h}_{A} = \MF{z}_{\MF{h}}+\sum_{\lambda \in \Delta_{A}\cap\Delta_{+}}\MF{h}_{\lambda}
            = \MF{z}_{\MF{h}}+\sum_{\lambda \in \Delta_{\Theta}\cap\Delta_{+}}\MF{h}_{\lambda}\\
            = \MF{h}_{\Theta},
\]
where $\Delta_{\Theta} = \Delta\cap\sum_{\lambda \in \Theta}\BS{R}\lambda$.
Hence we have
\begin{align*}
\MC{L}_{h}(G/H) &= \bigcup_{i=1}^{l}\bigcup_{\Theta \subset w_{i}\cdot \varPsi}\big\{[\MF{h}_{A}]\,\big|\,A \in C(w_{i}\cdot\varPsi, \Theta)\big\}\\
                &= \bigcup_{i=1}^{l}\big\{[\MF{h}_{\Theta}]\,\big|\,\Theta \subset w_{i}\cdot \varPsi\big\}.
\end{align*}
\end{proof}

\section{Complete systems of representatives for $\MC{W}(\Delta)/\MC{W}(\Delta^{a})$}\label{sec.weyl}

Let $(\MF{g}, \MF{h})$ be a semisimple symmetric pair
and $\sigma$ be an involution of $\MF{g}$
with $\MF{h}=\OPE{Ker}(\sigma-\OPE{id})$.
Suppose that $\theta$ is a Cartan involution of $\MF{g}$
commuting with $\sigma$,
and $\MF{a}$ is a maximal abelian subspace of $\MF{p}\cap\MF{q}$.
Denote by $\Delta$ the restricted root system of $(\MF{g}, \MF{h})$
with respect to $\MF{a}$.
Set $\MF{h}^{a} = \OPE{Ker}(\sigma\circ\theta - \OPE{id})$.
It is clear that $\MF{h}^{a} = \MF{k}\cap\MF{h}+\MF{p}\cap\MF{q}$ holds.
Then $\MF{h}^{a}$ is reductive, and
$(\MF{h}^{a}, \MF{k}\cap\MF{h})$ is a Riemannian
symmetric pair whose restricted root system is
$\Delta^{a}:=\{\lambda \in \Delta \mid \MF{g}_{\lambda} \cap \MF{h}^{a} \neq \{0\}\}$.
Denote by $\MC{W}(\Delta)$ (resp.\ $\MC{W}(\Delta^{a})$)
the Weyl group of $\Delta$ (resp.\ $\Delta^{a}$).
In this section,
for classical-type semisimple symmetric pairs $(\MF{g},\MF{h})$,
we shall give a complete system of representatives for
$\MC{W}(\Delta)/\MC{W}(\Delta^{a})$.
First, we shall list up
the types of $\Delta$ and $\Delta^{a}$, and
the index $|\MC{W}(\Delta)/\MC{W}(\Delta^{a})|$
of $\MC{W}(\Delta)/\MC{W}(\Delta^{a})$
(i.e.\ the cardinality of $\MC{W}(\Delta)/\MC{W}(\Delta^{a})$)
(see, Table \ref{table.index}).
In Table \ref{table.index},
$_{n}C_{k}$ is a binomial coefficient and,
for a Lie algebra $\MF{l}$, we denote by
$\MF{l}^{2}$ the direct sum $\MF{l}+\MF{l}$.

\begin{table}[htbp]
\centering
\caption{The index of $\MC{W}(\Delta)/\MC{W}(\Delta^{a})$}\label{table.index}
\scriptsize
\begin{tabular}{|p{55mm}|>{\PBS\centering}p{15mm}|>{\PBS\centering}p{21mm}|>{\PBS\centering}p{8mm}|>{\PBS\centering}p{18mm}|}
\hline
$(\MF{g},\MF{h})$&Type of $\Delta$& Type of $\Delta^{a}$&
Index&Remarks\\
\hline
\hline
\multirow{2}{*}{$(\MF{sl}(n,\BS{C}),\MF{sl}(n,\BS{R}))$}&
$(BC)_{m}$&$B_{m}$&$1$&$n=2m+1$\\\cline{2-5}
&$C_{m}$&$D_{m}$&$2$&$n=2m$\\\hline
$(\MF{sl}(n,\BS{R})^{2},\MF{sl}(n,\BS{R}))$
&$A_{n-1}$&$A_{n-1}$&$1$&\\\hline
$(\MF{sl}(n,\BS{C}),\MF{so}(n,\BS{C}))$
&$A_{n-1}$&$A_{n-1}$&$1$&\\\hline
$(\MF{sl}(2n,\BS{C}),\MF{su}^{*}(2n))$
&$C_{n}$&$C_{n}$&$1$&\\\hline
$(\MF{su}^{*}(2n)^{2},\MF{su}^{*}(2n))$
&$A_{n-1}$&$A_{n-1}$&$1$&\\\hline
$(\MF{sl}(2n,\BS{C}),\MF{sp}(n,\BS{C}))$
&$A_{n-1}$&$A_{n-1}$&$1$&\\\hline
$(\MF{sl}(n,\BS{C}),\MF{su}(p,n-p))$
&$A_{n-1}$&$A_{p-1}\times A_{n-p-1}$&$_{n}C_{p}$&\\\hline
\multirow{2}{*}{$(\MF{su}(p,n-p)^{2},\MF{su}(p,n-p))$}&
$(BC)_{p}$&$(BC)_{p}$&$1$&$n>2p$\\\cline{2-5}
&$C_{p}$&$C_{p}$&$1$&$n=2p$\\\hline
\multirow{2}{*}{$(\MF{sl}(n,\BS{C}),\MF{sl}(p,\BS{C})+\MF{sl}(n-p,\BS{C})+\BS{C})$}&
$(BC)_{p}$&$(BC)_{p}$&$1$&$n>2p$\\\cline{2-5}
&$C_{p}$&$C_{p}$&$1$&$n=2p$\\\hline
$(\MF{so}(2n,\BS{C}),\MF{so}^{*}(2n))$
&$D_{n}$&$A_{n-1}$&$2^{n-1}$&\\\hline
\multirow{2}{*}{$(\MF{so}^{*}(2n)^{2},\MF{so}^{*}(2n))$}&
$(BC)_{m}$&$(BC)_{m}$&$1$&$n=2m+1$\\\cline{2-5}
&$C_{m}$&$C_{m}$&$1$&$n=2m$\\\hline
\multirow{2}{*}{$(\MF{so}(2n,\BS{C}),\MF{sl}(n,\BS{C})+\BS{C})$}&
$(BC)_{m}$&$(BC)_{m}$&$1$&$n=2m+1$\\\cline{2-5}
&$C_{m}$&$C_{m}$&$1$&$n=2m$\\\hline
\multirow{6}{*}{$(\MF{so}(n,\BS{C}),\MF{so}(p,n-p))$}&
\multirow{2}{*}{$B_{m}$}&\multirow{2}{*}{$D_{q}\times B_{m-q}$}&\multirow{2}{*}{$2_{m}C_{q}$}&$n=2m+1$\\
&&&&$p=2q$\\\cline{2-5}
&\multirow{2}{*}{$B_{m}$}&\multirow{2}{*}{$B_{q}\times B_{m-q}$}&\multirow{2}{*}{$_{m}C_{q}$}&$n=2(m+1)$\\
&&&&$p=2q+1$\\\cline{2-5}
&\multirow{2}{*}{$D_{m}$}&\multirow{2}{*}{$D_{q}\times D_{m-q}$}&\multirow{2}{*}{$2_{m}C_{q}$}&$n=2m$\\
&&&&$p=2q$\\\hline
\end{tabular}
\end{table}

\begin{table}[htbp]
\centering
\contcaption{(continued)}
\scriptsize
\begin{tabular}{|p{55mm}|>{\PBS\centering}p{15mm}|>{\PBS\centering}p{21mm}|>{\PBS\centering}p{8mm}|>{\PBS\centering}p{18mm}|}
\hline
$(\MF{g},\MF{h})$&Type of $\Delta$& Type of $\Delta^{a}$&
Index&Remarks\\
\hline
\hline
\multirow{2}{*}{$(\MF{so}(p,n-p)^{2},\MF{so}(p,n-p))$}&
$B_{p}$&$B_{p}$&$1$&$n>2p$\\\cline{2-5}
&$D_{p}$&$D_{p}$&$1$&$n=2p$\\\hline
\multirow{2}{*}{$(\MF{so}(n,\BS{C}),\MF{so}(p,\BS{C})+\MF{so}(n-p,\BS{C}))$}&
$B_{p}$&$B_{p}$&$1$&$n>2p$\\\cline{2-5}
&$D_{p}$&$D_{p}$&$1$&$n=2p$\\\hline
$(\MF{sp}(n,\BS{C}),\MF{sp}(n,\BS{R}))$
&$C_{n}$&$A_{n-1}$&$2^{n}$&\\\hline
$(\MF{sp}(n,\BS{R})^{2},\MF{sp}(n,\BS{R}))$
&$C_{n}$&$C_{n}$&$1$&\\\hline
$(\MF{sp}(n,\BS{C}),\MF{sl}(n,\BS{C})+\BS{C})$
&$C_{n}$&$C_{n}$&$1$&\\\hline
$(\MF{sp}(n,\BS{C}),\MF{sp}(p,n-p))$
&$C_{n}$&$C_{p}\times C_{n-p}$&$_{n}C_{p}$&\\\hline
\multirow{2}{*}{$(\MF{sp}(p,n-p)^{2},\MF{sp}(p,n-p))$}&
$(BC)_{p}$&$(BC)_{p}$&$1$&$n>2p$\\\cline{2-5}
&$C_{p}$&$C_{p}$&$1$&$n=2p$\\\hline
\multirow{2}{*}{$(\MF{sp}(n,\BS{C}),\MF{sp}(p,\BS{C})+\MF{sp}(n-p,\BS{C}))$}&
$(BC)_{p}$&$(BC)_{p}$&$1$&$n>2p$\\\cline{2-5}
&$C_{p}$&$C_{p}$&$1$&$n=2p$\\\hline
$(\MF{sl}(n,\BS{R}),\MF{so}(p,n-p))$
&$A_{n-1}$&$A_{p-1}\times A_{n-p-1}$&$_{n}C_{p}$&\\\hline
\multirow{2}{*}{$(\MF{su}(p,n-p),\MF{so}(p,n-p))$}&
$(BC)_{p}$&$B_{p}$&$1$&$n>2p$\\\cline{2-5}
&$C_{p}$&$D_{p}$&$2$&$n=2p$\\\hline
\multirow{2}{*}{$(\MF{sl}(n,\BS{R}),\MF{sl}(p,\BS{R})+\MF{sl}(n-p,\BS{R})+\BS{R})$}&
$(BC)_{p}$&$B_{p}$&$1$&$n>2p$\\\cline{2-5}
&$C_{p}$&$D_{p}$&$2$&$n=2p$\\\hline
$(\MF{su}^{*}(2n),\MF{sp}(p,n-p))$
&$A_{n-1}$&$A_{p-1}\times A_{n-p-1}$&$_{n}C_{p}$&\\\hline
\multirow{2}{*}{$(\MF{su}(2p,2(n-p)),\MF{sp}(p,n-p))$}&
$(BC)_{p}$&$(BC)_{p}$&$1$&$n>2p$\\\cline{2-5}
&$C_{p}$&$C_{p}$&$1$&$n=2p$\\\hline
\multirow{2}{*}{$(\MF{su}^{*}(2n),\MF{su}^{*}(2p)+\MF{su}^{*}(2(n-p))+\BS{R})$}&
$(BC)_{p}$&$(BC)_{p}$&$1$&$n>2p$\\\cline{2-5}
&$C_{p}$&$C_{p}$&$1$&$n=2p$\\\hline
$(\MF{sl}(2n,\BS{R}),\MF{sp}(n,\BS{R}))$
&$A_{n-1}$&$A_{n-1}$&$1$&\\\hline
$(\MF{su}^{*}(2n),\MF{so}^{*}(2n))$
&$A_{n-1}$&$A_{n-1}$&$1$&\\\hline
$(\MF{su}(n,n),\MF{so}^{*}(2n))$
&$C_{n}$&$C_{n}$&$1$&\\\hline
$(\MF{sl}(2n,\BS{R}),\MF{sl}(n,\BS{C})+\MF{so}(2))$
&$C_{n}$&$C_{n}$&$1$&\\\hline
\multirow{2}{*}{$(\MF{su}^{*}(2n),\MF{sl}(n,\BS{C})+\MF{so}(2))$}&
$(BC)_{m}$&$(BC)_{m}$&$1$&$n=2m+1$\\\cline{2-5}
&$C_{m}$&$C_{m}$&$1$&$n=2m$\\\hline
\multirow{2}{*}{$(\MF{su}(n,n),\MF{sp}(n,\BS{R}))$}&
$(BC)_{m}$&$(BC)_{m}$&$1$&$n=2m+1$\\\cline{2-5}
&$C_{m}$&$C_{m}$&$1$&$n=2m$\\\hline
$(\MF{su}(n,n),\MF{sl}(n,\BS{C})+\BS{R})$
&$C_{n}$&$A_{n-1}$&$2^{n}$&\\\hline
\multirow{6}{*}{$(\MF{so}^{*}(2n),\MF{su}(p,n-p)+\MF{so}(2))$}&
\multirow{2}{*}{$(BC)_{m}$}&\multirow{2}{*}{$C_{q}\times (BC)_{m-q}$}&\multirow{2}{*}{$_{m}C_{q}$}&$n=2m+1$\\
&&&&$p=2q$\\\cline{2-5}
&\multirow{2}{*}{$(BC)_{m}$}&\multirow{2}{*}{$C_{q}\times (BC)_{m-q}$}&\multirow{2}{*}{$_{m}C_{q}$}&$n=2(m+1)$\\
&&&&$p=2q+1$\\\cline{2-5}
&\multirow{2}{*}{$C_{m}$}&\multirow{2}{*}{$C_{q}\times C_{m-q}$}&\multirow{2}{*}{$_{m}C_{q}$}&$n=2m$\\
&&&&$p=2q$\\\hline
\multirow{2}{*}{$(\MF{so}(2p,2(n-p)),\MF{su}(p,n-p)+\MF{so}(2))$}&
$(BC)_{p}$&$(BC)_{p}$&$1$&$n>2p$\\\cline{2-5}
&$C_{p}$&$C_{p}$&$1$&$n=2p$\\\hline
\multirow{2}{*}{$(\MF{so}^{*}(2n),\MF{so}^{*}(2p)+\MF{so}^{*}(2(n-p)))$}&
$(BC)_{p}$&$(BC)_{p}$&$1$&$n>2p$\\\cline{2-5}
&$C_{p}$&$C_{p}$&$1$&$n=2p$\\\hline
\end{tabular}
\end{table}

\begin{table}[htbp]
\centering
\contcaption{(continued)}
\scriptsize
\begin{tabular}{|p{55mm}|>{\PBS\centering}p{15mm}|>{\PBS\centering}p{21mm}|>{\PBS\centering}p{8mm}|>{\PBS\centering}p{18mm}|}
\hline
$(\MF{g},\MF{h})$&Type of $\Delta$& Type of $\Delta^{a}$&
Index&Remarks\\
\hline
\hline
$(\MF{so}(n,n),\MF{so}(n,\BS{C}))$
&$D_{n}$&$A_{n-1}$&$2^{n-1}$&\\\hline
\multirow{2}{*}{$(\MF{so}^{*}(2n),\MF{so}(n,\BS{C}))$}&
$(BC)_{m}$&$B_{m}$&$1$&$n=2m+1$\\\cline{2-5}
&$C_{m}$&$D_{m}$&$2$&$n=2m$\\\hline
\multirow{2}{*}{$(\MF{so}(n,n),\MF{sl}(n,\BS{R})+\BS{R})$}&
$(BC)_{m}$&$B_{m}$&$1$&$n=2m+1$\\\cline{2-5}
&$C_{m}$&$D_{m}$&$2$&$n=2m$\\\hline
$(\MF{so}^{*}(4n),\MF{su}^{*}(2n)+\BS{R})$
&$C_{n}$&$A_{n-1}$&$2^{n}$&\\\hline
$(\MF{sp}(n,\BS{R}),\MF{su}(p,n-p)+\MF{so}(2))$
&$C_{n}$&$C_{p}\times C_{n-p}$&$_{n}C_{p}$&\\\hline
\multirow{2}{*}{$(\MF{sp}(p,n-p),\MF{su}(p,n-p)+\MF{so}(2))$}&
$(BC)_{p}$&$(BC)_{p}$&$1$&$n>2p$\\\cline{2-5}
&$C_{p}$&$C_{p}$&$1$&$n=2p$\\\hline
\multirow{2}{*}{$(\MF{sp}(n,\BS{R}),\MF{sp}(p,\BS{R})+\MF{sp}(n-p,\BS{R}))$}&
$(BC)_{p}$&$(BC)_{p}$&$1$&$n>2p$\\\cline{2-5}
&$C_{p}$&$C_{p}$&$1$&$n=2p$\\\hline
$(\MF{sp}(n,\BS{R}),\MF{sl}(n,\BS{R})+\BS{R})$
&$C_{n}$&$A_{n-1}$&$2^{n}$&\\\hline
$(\MF{sp}(n,n),\MF{sp}(n,\BS{C}))$
&$C_{n}$&$A_{n-1}$&$2^{n}$&\\\hline
$(\MF{sp}(2n,\BS{R}),\MF{sp}(n,\BS{C}))$
&$C_{n}$&$C_{n}$&$1$&\\\hline
$(\MF{sp}(n,n),\MF{su}^{*}(2n)+\BS{R})$
&$C_{n}$&$C_{n}$&$1$&\\\hline
\end{tabular}
\end{table}

\begin{table}[htbp]
\centering
\scriptsize
\begin{tabular}{|>{\PBS\centering}p{25mm}|>{\PBS\centering}p{35mm}|>{\PBS\centering}p{25mm}|>{\PBS\centering}p{30mm}|}
\multicolumn{4}{l}{
$(\MF{g},\MF{h})=(\MF{su}(n,m),\MF{su}(i,j)+\MF{su}(n-i,m-j)+\MF{so}(2))$}\\
\hline
Type of $\Delta$& Type of $\Delta^{a}$&
Index&Remarks\\
\hline
\hline
$C_{n}$&$C_{i}\times C_{n-i}$&$_{n}C_{i}$&$i+j=n=m$\\\hline
$(BC)_{n}$&$C_{i}\times (BC)_{n-i}$&$_{n}C_{i}$&$n<i+j=m$\\\hline
$(BC)_{m+n-(i+j)}$&$(BC)_{m-j}\times (BC)_{n-i}$&$_{m+n-(i+j)}C_{n-i}$&$n\leq m<i+j$\\\hline
$(BC)_{n}$&$(BC)_{i}\times C_{n-i}$&$_{n}C_{i}$&$n=i+j<m$\\\hline
$(BC)_{n}$&$(BC)_{i}\times (BC)_{n-i}$&$_{n}C_{i}$&$n<i+j<m$\\\hline
$(BC)_{i+j}$&$(BC)_{i}\times (BC)_{j}$&$_{i+j}C_{i}$&$i+j<n\leq m$\\\hline
\end{tabular}
\end{table}

\begin{table}[htbp]
\centering
\scriptsize
\begin{tabular}{|>{\PBS\centering}p{25mm}|>{\PBS\centering}p{35mm}|>{\PBS\centering}p{25mm}|>{\PBS\centering}p{30mm}|}
\multicolumn{4}{l}{
$(\MF{g},\MF{h})=(\MF{so}(n,m),\MF{so}(i,j)+\MF{so}(n-i,m-j))$}\\
\hline
Type of $\Delta$& Type of $\Delta^{a}$&
Index&Remarks\\
\hline
\hline
$D_{n}$&$D_{i}\times D_{n-i}$&$2_{n}C_{i}$&$i+j=n=m$\\\hline
$B_{n}$&$D_{i}\times B_{n-i}$&$2_{n}C_{i}$&$n<i+j=m$\\\hline
$B_{m+n-(i+j)}$&$B_{m-j}\times B_{n-i}$&$_{m+n-(i+j)}C_{n-i}$&$n\leq m<i+j$\\\hline
$B_{n}$&$B_{i}\times D_{n-i}$&$2_{n}C_{i}$&$n=i+j<m$\\\hline
$B_{n}$&$B_{i}\times B_{n-i}$&$_{n}C_{i}$&$n<i+j<m$\\\hline
$B_{i+j}$&$B_{i}\times B_{j}$&$_{i+j}C_{i}$&$i+j<n\leq m$\\\hline
\end{tabular}
\end{table}

\begin{table}[htbp]
\centering
\scriptsize
\begin{tabular}{|>{\PBS\centering}p{25mm}|>{\PBS\centering}p{35mm}|>{\PBS\centering}p{25mm}|>{\PBS\centering}p{30mm}|}
\multicolumn{4}{l}{
$(\MF{g},\MF{h})=(\MF{sp}(n,m),\MF{sp}(i,j)+\MF{sp}(n-i,m-j))$}\\
\hline
Type of $\Delta$& Type of $\Delta^{a}$&
Index&Remarks\\
\hline
\hline
$C_{n}$&$C_{i}\times C_{n-i}$&$_{n}C_{i}$&$i+j=n=m$\\\hline
$(BC)_{n}$&$C_{i}\times (BC)_{n-i}$&$_{n}C_{i}$&$n<i+j=m$\\\hline
$(BC)_{m+n-(i+j)}$&$(BC)_{m-j}\times (BC)_{n-i}$&$_{m+n-(i+j)}C_{n-i}$&$n\leq m<i+j$\\\hline
$(BC)_{n}$&$(BC)_{i}\times C_{n-i}$&$_{n}C_{i}$&$n=i+j<m$\\\hline
$(BC)_{n}$&$(BC)_{i}\times (BC)_{n-i}$&$_{n}C_{i}$&$n<i+j<m$\\\hline
$(BC)_{i+j}$&$(BC)_{i}\times (BC)_{j}$&$_{i+j}C_{i}$&$i+j<n\leq m$\\\hline
\end{tabular}
\end{table}

If $\Delta=\Delta^{a}$ holds,
then $\{\OPE{id}\}$ ($\OPE{id}$: the identity transformation of $\MF{a}$)
gives a complete system of representative for
$\MC{W}(\Delta)/\MC{W}(\Delta^{a})$.
Assume that $(\MF{g}, \MF{h})$ is an irreducible symmetric pair
with $\Delta^{a} \subsetneq \Delta$.
Without loss of generality, we assume that
$(\MF{g},\MF{h})=(\MF{g}',\MF{h}'_{\varepsilon(\lambda)})$
for suitable basic symmetric pair $(\MF{g}',\MF{h}')$
and $\lambda \in \varPsi'$,
where $\varPsi'$ is a simple root system of
the restricted root system of $(\MF{g}',\MF{h}')$
and $\varepsilon(\lambda)$ denotes the signature of $\Delta$
defined by $\varepsilon(\mu)=1$ if $\mu \in \varPsi'\setminus \{\lambda\}$,
$\varepsilon(\lambda)=-1$
(cf.\,Section 6 in \cite{MR810638}).
Note that by using this assumption,
we can express the inclusion $\Delta^{a}\subset\Delta$
explicitly (see, Table V in \cite{MR810638} and Table 1, 2 in \cite{OS2}).

In the sequel,
we shall follow notations of irreducible root systems in \cite{MR1920389}:
\begin{align*}
A_{n} &= \{e_{i}-e_{j} \mid 1 \leq i \neq j \leq n+1 \},\\
B_{n} &= \{e_{i} \pm e_{j} \mid 1 \leq i < j \leq n \}\cup \{e_{i} \mid 1 \leq i \leq n\},\\
C_{n} &= \{e_{i} \pm e_{j} \mid 1 \leq i < j \leq n \}\cup \{2e_{i} \mid 1 \leq i \leq n\},\\
D_{n} &= \{e_{i} \pm e_{j} \mid 1 \leq i < j \leq n \},\\
(BC)_{n} &=\{e_{i} \pm e_{j} \mid 1 \leq i < j \leq n \}\cup \{e_{i}, 2e_{i} \mid 1 \leq i \leq n\}.
\end{align*}
\noindent
Denote by $s_{\lambda}$ $(\lambda \in \Delta)$
the reflection on
$\MF{a}$ along $\lambda^{-1}(0)$.

\subsection{Type $(\Delta,\Delta^{a})=(C_{n},A_{n-1})$.}\label{subsec.CA}

Assume that
\[A_{n-1}=\{e_{i}-e_{j}\mid 1 \leq i \neq j \leq n\} \subset C_{n}.\]
$\varPsi:=\{e_{i}-e_{i+1}\mid 1 \leq i \leq n-1\}\cup\{2e_{n}\}$
is a simple root system of $\Delta$.
For $\Delta=C_{n}$,
$\MC{W}(\Delta)$ is generated by all permutations of $e_{1},\ldots,e_{n}$,
and all sign changes of the coefficients for $e_{1},\ldots,e_{n}$.
Set
\[
t_{i}=\begin{cases}
s_{e_{i}-e_{i+1}}\cdots s_{e_{n-1}-e_{n}}s_{2e_{n}}s_{e_{n-1}-e_{n}}\cdots s_{e_{i}-e_{i+1}} & (1 \leq i \leq n-1),\\
s_{2e_{n}} & (i=n).
\end{cases}
\]
Since
$\MC{W}(\Delta^{a})$ is generated all permutations of $e_{1},\ldots,e_{n}$,
we have $t_{i} \not\in \MC{W}(\Delta^{a})$.
Therefore all sign changes, that is,
$
\left\{\prod_{i=1}^{n}t_{i}^{l_{i}} \mid l_{i}=0,1\right\}
$
gives a complete system of representatives
for $\MC{W}(\Delta)/\MC{W}(\Delta^{a})$.
Indeed,
\[
\left(\prod_{i=1}^{n}t_{i}^{l_{i}}\right)^{-1}\prod_{i=1}^{n}t_{i}^{k_{i}}
= \prod_{i=1}^{n}t_{i}^{l_{i}+k_{i}} \not\in \MC{W}(\Delta^{a})
\]
is equivalent to $l_{i_{0}}\neq k_{i_{0}}$ for some $1 \leq i_{0} \leq n$.

\subsection{Type $(\Delta,\Delta^{a})=(C_{n},D_{n})$.}\label{subsec.CD}
Assume that
\[
D_{n}=\{\pm e_{i}\pm e_{j}\mid 1 \leq i < j \leq n \} \subset C_{n}.
\]
$\varPsi:=\{e_{i}-e_{i+1}\mid 1 \leq i \leq n-1\}\cup\{2e_{n}\}$
is a simple root system of $\Delta$.
Since $\MC{W}(\Delta^{a})$ is generated by
all permutations of $e_{1},\ldots,e_{n}$ and
all even sign changes of $e_{1}, \ldots,e_{n}$,
we have $s_{2e_{n}} \not\in \MC{W}(\Delta^{a})$.
Hence $\left\{\OPE{id}, s_{2e_{n}}\right\}$
gives a complete system of representatives
for $\MC{W}(\Delta)/\MC{W}(\Delta^{a})$.

\subsection{Type $(\Delta,\Delta^{a})=(D_{n},A_{n-1})$.}\label{subsec.DA}

Assume that
\[
A_{n-1}=\{e_{i}-e_{j}\mid 1 \leq i \neq j \leq n\}\subset D_{n}.
\]
$\varPsi:=\{e_{i}-e_{i+1}\mid 1 \leq i \leq n-1\}\cup\{e_{n-1}+e_{n}\}$
is a simple root system of $\Delta$.
By the assumption,
$\MC{W}(\Delta^{a})$ is generated by all permutations of $e_{1},\ldots,e_{n}$.
By a similar argument for the case of type $(C_{n},A_{n-1})$,
all even sign changes, that is,
$\left\{
\prod_{i=1}^{n-1}(s_{e_{i}-e_{i+1}}s_{e_{i}+e_{i+1}})^{l_{i}}\mid l_{i}=0,1\right\}
$
give a complete system of representatives for 
$\MC{W}(\Delta)/\MC{W}(\Delta^{a})$.

\begin{Rem}
Assume that $\Delta=D_{n}$ and
$$\Delta^{a} =\{e_{i}-e_{j}\mid 1 \leq i \neq j \leq n-1\}\cup\{\pm(e_{i}+e_{n})\mid 1 \leq i \leq n-1\} (\cong A_{n-1}).$$
Then we can prove that
$\left\{
\prod_{i=1}^{n-1}(s_{e_{i}-e_{i+1}}s_{e_{i}+e_{i+1}})^{l_{i}} \mid l_{i}=0,1
\right\}
$
gives a complete system of representatives for
$\MC{W}(\Delta)/\MC{W}(\Delta^{a})$.
\end{Rem}

\subsection{Type $(\Delta,\Delta^{a})=(A_{n-1},A_{p-1}\times A_{n-p-1})$.}\label{subsec.AAA}
Assume that
\begin{align*}
A&_{p-1}\times A_{n-p-1}\\
&=\{e_{i} - e_{j}\mid 1 \leq i \neq j \leq p \}\cup
\{e_{i} - e_{j}\mid p+1 \leq i \neq j \leq n \} \subset A_{n-1}.
\end{align*}
$\varPsi:=\{e_{i}-e_{i+1}\mid 1 \leq i \leq n-1\}$
is a simple root system of $\Delta$ and
$\varPsi\cap\Delta^{a} = \{e_{i}-e_{i+1}\mid i \neq p\}$
is a simple root system of $\Delta^{a}$.
For $p=0, n$, we have $\Delta=\Delta^{a}$, so that $\OPE{id}$
gives a complete system of representative.
For $p=1,\ldots,n-1$,
we give the following formula for $\MC{W}(\Delta)/\MC{W}(\Delta^{a})$.
\begin{Prop}\label{lem.fmula1}
\begin{align}\label{fmula1}
&\MC{W}(\Delta)/\MC{W}(\Delta^{a}) \\\notag&= \MC{W}(\Gamma)/\MC{W}(\Gamma^{a})
\cup \left\{w s_{e_{n-1}-e_{n}}\cdots s_{e_{p}-e_{p+1}}\,\biggm|\,
w \in \MC{W}(\Gamma)/\MC{W}(\tilde{\Gamma})\right\},
\end{align}
where $\Gamma:=\{e_{i}-e_{j}\mid 1 \leq i\neq j \leq n-1\}(\subset \Delta)$,
$\Gamma^{a}:=\Gamma\cap\Delta^{a}$ and
$\tilde{\Gamma}:=\{e_{i}-e_{j} \mid 1 \leq i \neq j \leq p-1\} \cup \{e_{i}-e_{j}\mid p \leq i \neq j \leq n-1\}$.
\end{Prop}
\begin{proof}
By using the formula ${}_{n}C_{p}={}_{n-1}C_{p-1}+{}_{n-1}C_{p}$,
the cardinalities of the both sides in $(\ref{fmula1})$ are the same.
Let $w_{i}\,(1 \leq i \leq {}_{n-1}C_{p-1})$
(resp.\,$z_{j}\,(1\leq j \leq {}_{n-1}C_{p})$) be a complete system of
representatives for
$\MC{W}(\Gamma)/\MC{W}(\Gamma^{a})$ (resp.\,$\MC{W}(\Gamma)/\MC{W}(\tilde{\Gamma})$).
By the definition of $\Gamma^{a}$,
$\MC{W}(\Gamma^{a})$ is generated by all permutations of
$e_{1}, \ldots, e_{p}$ and all permutations of $e_{p+1}, \ldots,e_{n-1}$.
For $w_{i},w_{j}\,(i\neq j)$,
we have $w_{i}^{-1}w_{j}(e_{p})=e_{k}$
for some $p+1 \leq k \leq n-1$.
This implies that
$w_{i}^{-1}w_{j} \not\in \MC{W}(\Delta^{a})$ holds.
Indeed, any element in $\MC{W}(\Delta^{a})$
must preserve $\{e_{1},\ldots,e_{p}\}$ invariantly.
For $w_{i}, z_{j}$,
we have
\[w_{i}^{-1}(z_{j}s_{e_{n-1}-e_{n}}\cdots s_{e_{p}-e_{p+1}})(e_{p})
=w_{i}^{-1}z_{j}(e_{n})=e_{n}.\]
This implies that
$w_{i}^{-1}(z_{j}s_{e_{n-1}-e_{n}}\cdots s_{e_{p}-e_{p+1}}) \not\in \MC{W}(\Delta^{a})$ holds.
For $z_{i},z_{j}\,(i\neq j)$,
we have $z_{i}^{-1}z_{j}(e_{k_{1}})=e_{k_{2}}$
for some $1 \leq k_{1} \leq p-1$ and $p \leq k_{2} \leq n-1$.
Then we have 
\begin{align*}
(z_{i}&s_{e_{n-1}-e_{n}}\cdots s_{e_{p}-e_{p+1}})^{-1}(z_{j}s_{e_{n-1}-e_{n}}\cdots s_{e_{p}-e_{p+1}})(e_{k_{1}})\\
&=(s_{e_{p}-e_{p+1}} \cdots s_{e_{n-1}-e_{n}})(z_{i}^{-1}z_{j})(e_{k_{1}})
=(s_{e_{p}-e_{p+1}} \cdots s_{e_{n-1}-e_{n}})(e_{k_{2}})\\
& \in \{e_{p+1},\ldots,e_{n}\}.
\end{align*}
Therefore
$(z_{i}s_{e_{n-1}-e_{n}}\cdots s_{e_{p}-e_{p+1}})^{-1}(z_{j}s_{e_{n-1}-e_{n}}\cdots s_{e_{p}-e_{p+1}}) \not\in \MC{W}(\Delta^{a})$.
Hence all $w_{i}$'s and $z_{j}$'s are not equivalent to each other
in $\MC{W}(\Delta^{a})$.
This proves Proposition \ref{lem.fmula1}.
\end{proof}

\begin{Rem}
By using Proposition \ref{lem.fmula1},
we can give a complete system of representatives for
$\MC{W}(\Delta)/\MC{W}(\Delta^{a})$, recursively.
\end{Rem}

\begin{Ex}[Type $(\Delta,\Delta^{a})=(A_{4},A_{1}\times A_{2})$]\label{ex.weyl1}
We shall give a complete system of representatives
for $\MC{W}(\Delta)/\MC{W}(\Delta^{a})$
by using Proposition \ref{lem.fmula1}.
Set
$\Gamma^{k}:=\{e_{i}-e_{j}\mid 1 \leq i \neq j \leq k\} (\subset \Delta)$
for $1 \leq k \leq 5$, and
\[
\Gamma^{l,p}=\{e_{i}-e_{j}\mid 1 \leq i \neq j \leq p\}\cup\{e_{i}-e_{j}\mid p+1 \leq i \neq j \leq l\} (\subset \Gamma^{l})
\]
for $0 \leq p \leq l \leq 5$.
Set $s_{i}=s_{e_{i}-e_{i+1}}$ for $1 \leq i \leq 4$.
Then we have
\begin{align*}
&\MC{W}(\Delta)/\MC{W}(\Delta^{a})\\
&=\MC{W}(\Gamma^{5})/\MC{W}(\Gamma^{5,2})\\
&=\MC{W}(\Gamma^{4})/\MC{W}(\Gamma^{4,2})\cup\{ws_{4}s_{3}s_{2}\mid w \in \MC{W}(\Gamma^{4})/\MC{W}(\Gamma^{4,1})\}\\
&=\MC{W}(\Gamma^{3})/\MC{W}(\Gamma^{3,2})
\cup\{ws_{3}s_{2}\mid w \in \MC{W}(\Gamma^{3})/\MC{W}(\Gamma^{3,1})\}\\
&\phantom{\Gamma)\cup\{ws_{3}s_{2}\mid w \in \MC{W}(\Gamma^{3})/\MC{W}(\Gamma^{3,1})\}
}
\cup\{ws_{4}s_{3}s_{2}\mid w \in \MC{W}(\Gamma^{4})/\MC{W}(\Gamma^{4,1})\}\\
&=\MC{W}(\Gamma^{2})/\MC{W}(\Gamma^{2,2})
\cup\{ws_{2}\mid w \in \MC{W}(\Gamma^{2})/\MC{W}(\Gamma^{2,1})\}\\
&\phantom{\MC{W}}
\cup\{ws_{3}s_{2}\mid w \in \MC{W}(\Gamma^{3})/\MC{W}(\Gamma^{3,1})\}
\cup\{ws_{4}s_{3}s_{2}\mid w \in \MC{W}(\Gamma^{4})/\MC{W}(\Gamma^{4,1})\}
\end{align*}
Moreover, we have $\MC{W}(\Gamma^{2})/\MC{W}(\Gamma^{2,2})=\{\OPE{id}\}$,
\begin{align*}
\MC{W}(\Gamma^{2})/\MC{W}(\Gamma^{2,1})
&= \MC{W}(\Gamma^{1})/\MC{W}(\Gamma^{1,1})\cup\{ws_{1}\mid w \in \MC{W}(\Gamma^{1})/\MC{W}(\Gamma^{1,0})\}\\
&= \{\OPE{id}, s_{1}\},\\
\MC{W}(\Gamma^{3})/\MC{W}(\Gamma^{3,1})
&= \MC{W}(\Gamma^{2})/\MC{W}(\Gamma^{2,1})\cup\{ws_{2}s_{1}\mid w \in \MC{W}(\Gamma^{2})/\MC{W}(\Gamma^{2,0})\}\\
&= \{\OPE{id}, s_{1}, s_{2}s_{1}\},\\
\MC{W}(\Gamma^{4})/\MC{W}(\Gamma^{4,1})
&= \MC{W}(\Gamma^{3})/\MC{W}(\Gamma^{3,1})\cup\{ws_{3}s_{2}s_{1}\mid w \in \MC{W}(\Gamma^{3})/\MC{W}(\Gamma^{3,0})\}\\
&= \{\OPE{id}, s_{1}, s_{2}s_{1}, s_{3}s_{2}s_{1}\}.\\
\end{align*}
Therefore we conclude that
\[
\MC{W}(\Delta)/\MC{W}(\Delta^{a})=
\left\{
\begin{split}
&\OPE{id},
s_{2}, s_{1}s_{2},
s_{3}s_{2}, s_{1}s_{3}s_{2}, s_{2}s_{1}s_{3}s_{2},
s_{4}s_{3}s_{2},\\
&s_{1}s_{4}s_{3}s_{2},
s_{2}s_{1}s_{4}s_{3}s_{2},
s_{3}s_{2}s_{1}s_{4}s_{3}s_{2}
\end{split}
\right\}.
\]
\end{Ex}

\subsection{Type $(\Delta,\Delta^{a})=(B_{n},B_{p}\times B_{n-p})$,
$(C_{n},C_{p}\times C_{n-p})$, $((BC)_{n},C_{p}\times (BC)_{n-p})$ or
$((BC)_{n},(BC)_{p}\times (BC)_{n-p})$.}\label{subsec.BBBetc}
\hspace{-1mm}We consider the case where
$(\Delta,\Delta^{a})$ is of type $(B_{n},B_{p}\times B_{n-p})$.
Assume that
\begin{align*}
B_{p}&\times B_{n-p}\\
&= \{\pm e_{i}\pm e_{j}\mid 1 \leq i < j \leq p\}\\
&\phantom{e_{i}\pm e_{j}}\cup \{\pm e_{i}\pm e_{j}\mid p+1 \leq i < j \leq n\}
\cup \{\pm e_{i} \mid 1 \leq i \leq n\}\subset B_{n}.
\end{align*}
Since $\MC{W}(\Delta)$ is equal to the Weyl group for $C_{n}$
(cf.\,Subsection \ref{subsec.CA}),
and $\MC{W}(\Delta^{a})$
is generated by all permutations of $e_{1},\ldots,e_{p}$,
all permutations of $e_{p+1},\ldots e_{n}$ and all sign changes
of the coefficients of $e_{1},\ldots,e_{n}$,
we have
$
\MC{W}(\Delta)/\MC{W}(\Delta^{a})= \MC{W}(\Lambda)/\MC{W}(\Lambda^{a}),
$
where $\Lambda :=\{e_{i} - e_{j}\mid 1 \leq i \neq j \leq n\} (\subset \Delta)$
and $\Lambda^{a} = \Lambda \cap \Delta^{a}$.
Moreover, by applying Proposition \ref{lem.fmula1}
to $\MC{W}(\Lambda)/\MC{W}(\Lambda^{a})$,
we can give a complete system of representatives for $\MC{W}(\Delta)/\MC{W}(\Delta^{a})$.
The arguments in the other cases are similar.

\subsection{Type $(\Delta,\Delta^{a})=(D_{n},D_{p}\times D_{n-p})$
or $(B_{n},D_{p}\times B_{n-p})$.}\label{subsec.DDDetc}

We consider the case where
$(\Delta, \Delta^{a})$ is of type $(D_{n},D_{p}\times D_{n-p})$.
Assume that
\begin{align*}
&D_{p}\times D_{n-p}\\
&=\{\pm e_{i} \pm e_{j} \mid 1 \leq i < j \leq p\}\cup
\{\pm e_{i} \pm e_{j} \mid p+1 \leq i < j \leq n\}
\subset D_{n}.
\end{align*}
Then $\MC{W}(\Delta^{a})$
is generated by all permutations of $e_{1}, \ldots, e_{p}$,
all permutations of $e_{1}, \ldots, e_{p}$,
all even sign changes of $e_{1}, \ldots, e_{p}$
and all even sign changes of $e_{p+1}, \ldots, e_{n}$.
In particular, 
$s_{e_{p}-e_{p+1}}s_{e_{p}+e_{p+1}} \not\in \MC{W}(\Delta^{a})$
holds.
Hence
\[
\MC{W}(\Delta)/\MC{W}(\Delta^{a})=\{(s_{e_{p}-e_{p+1}}s_{e_{p}+e_{p+1}})^{l}
w \mid l=0,1, w\in \MC{W}(\Lambda)/\MC{W}(\Lambda^{a})\},
\]
where $\Lambda :=\{e_{i} - e_{j}\mid 1 \leq i \neq j \leq n\} (\subset \Delta)$
and $\Lambda^{a}=\Lambda \cap \Delta^{a}$.
By applying Proposition \ref{lem.fmula1} to $\MC{W}(\Lambda)/\MC{W}(\Lambda^{a})$,
we can give a complete system of representatives for
$\MC{W}(\Delta)/\MC{W}(\Delta^{a})$.
By a similar argument,
we can give a complete system of representatives in the case of type
$(B_{n},D_{p}\times B_{n-p})$.

\section{Determination of hyperbolic principal isotropy subalgebras}\label{sec.hprin}

In this section,
we shall determine the
(Lie algebra) structure of the HPIS for s-representation
of an irreducible semisimple pseudo-Riemannian
symmetric space $G/H$ by using the Satake diagrams associated with $G/H$.
Denote by $\MF{g}$ (resp.\,$\MF{h}$) the Lie algebra of $G$ (resp.\,$H$).
Let $\sigma$ be an involution of $\MF{g}$ with
$\MF{h}=\OPE{Ker}(\sigma -\OPE{id})$.
Suppose that $\theta$ is a Cartan involution of $\MF{g}$
commuting with $\sigma$.
Set $\MF{k}:=\OPE{Ker}(\theta -\OPE{id})$
and $\MF{p}:=\OPE{Ker}(\theta + \OPE{id})$.
Let $\MF{a}$ be a maximal abelian subspace of $\MF{p}\cap\MF{q}$
and $\MF{a}_{\MF{q}}$ (resp.\,$\MF{a}_{\MF{p}}$)
be a maximal abelian subspace of $\MF{q}$ (resp.\,$\MF{p}$)
containing $\MF{a}$.
Let $\tilde{\MF{a}}$ be a maximal abelian subalgebra of $\MF{g}$
containing $\MF{a}_{\MF{q}}+\MF{a}_{\MF{p}}$.
Note that any hyperbolic principal isotropy subalgebra
is equal to the centralizer $\MF{z}_{\MF{h}}$ of $\MF{a}$ in $\MF{h}$.
Denote by $\MF{z}_{\MF{g}}$ the centralizer of $\MF{a}$
in $\MF{g}$.
Since $\MF{z}_{\MF{g}}$ is invariant under $\sigma$ and $\theta$,
we have the decomposition
\[
\MF{z}_{\MF{g}}
=\MF{z}_{\MF{g}}\cap(\MF{k}\cap\MF{h})+\MF{z}_{\MF{g}}\cap(\MF{p}\cap\MF{h})
+\MF{z}_{\MF{g}}\cap(\MF{k}\cap\MF{q})+\MF{z}_{\MF{g}}\cap(\MF{p}\cap\MF{q}).
\]
It is clear that
$\MF{z}_{\MF{g}}\cap(\MF{p}\cap\MF{q})=\MF{a}$.
We also have the decomposition
$\MF{z}_{\MF{g}}=\MF{z}_{\MF{g}}^{c}+\MF{z}_{\MF{g}}^{s}$,
where $\MF{z}_{\MF{g}}^{c}$ (resp.\,$\MF{z}_{\MF{g}}^{s}$)
denotes the center (resp.\,the semisimple part) of $\MF{z}_{\MF{g}}$.
It is clear that $\MF{z}^{c}_{\MF{g}}$ is contained in $\tilde{\MF{a}}$.
Then their decomposition are compatible, i.e.,
\begin{align*}
\MF{z}_{\MF{g}}^{c}&=
\MF{z}_{\MF{g}}^{c}\cap(\MF{k}\cap\MF{h})
+\MF{z}_{\MF{g}}^{c}\cap(\MF{p}\cap\MF{h})
+\MF{z}_{\MF{g}}^{c}\cap(\MF{k}\cap\MF{q})
+\MF{z}_{\MF{g}}^{c}\cap(\MF{p}\cap\MF{q}),\\
\MF{z}_{\MF{g}}^{s}&=
\MF{z}_{\MF{g}}^{s}\cap(\MF{k}\cap\MF{h})
+\MF{z}_{\MF{g}}^{s}\cap(\MF{p}\cap\MF{h})
+\MF{z}_{\MF{g}}^{s}\cap(\MF{k}\cap\MF{q})
+\MF{z}_{\MF{g}}^{s}\cap(\MF{p}\cap\MF{q}).
 \end{align*}
Set $\tilde{\MF{a}}^{s}=\tilde{\MF{a}}\cap\MF{z}_{\MF{g}}^{s}$
and $\tilde{\MF{a}}^{c}=\tilde{\MF{a}}\cap\MF{z}_{\MF{g}}^{c}$.
We can obtain the dimension of $\tilde{\MF{a}}^{c}\cap(\MF{k}\cap\MF{h})$
(resp.\,$\tilde{\MF{a}}^{c}\cap(\MF{p}\cap\MF{h}),\,
\tilde{\MF{a}}^{c}\cap(\MF{k}\cap\MF{q})$) by calculating
$\dim \tilde{\MF{a}}\cap(\MF{k}\cap\MF{h})
- \dim \tilde{\MF{a}}^{s}\cap(\MF{k}\cap\MF{h})$
(resp.\, $\dim \tilde{\MF{a}}\cap(\MF{p}\cap\MF{h})
- \dim \tilde{\MF{a}}^{s}\cap(\MF{p}\cap\MF{h}),\,
\dim\tilde{\MF{a}}\cap(\MF{k}\cap\MF{q})
- \dim \tilde{\MF{a}}^{s}\cap(\MF{k}\cap\MF{q})$).
(cf.\,Recipe \ref{recipe.hprin})

\subsection{The structure of $\MF{z}_{\MF{g}}^{\BS{C}}$}

Denote by $R$ the root system of $\MF{g}^{\BS{C}}$
with respect to $\tilde{\MF{a}}^{\BS{C}}$,
and by $\MF{g}^{\BS{C}}_{\alpha}$ the root space of $\MF{g}^{\BS{C}}$
associated with $\alpha \in R$.
Then we have the decomposition
\[
(\MF{z}_{\MF{g}}^{s})^{\BS{C}}
= \OPE{Span}_{\BS{C}}\{A_{\alpha}\mid \alpha \in R_{0}\}
+ \sum_{\alpha \in R_{0}}\MF{g}_{\alpha}^{\BS{C}},
\]
and $\dim_{\BS{C}}(\MF{z}_{\MF{g}}^{c})^{\BS{C}}=
\dim_{\BS{C}} \tilde{\MF{a}}^{\BS{C}}
-\OPE{rank}R_{0}$,
where $A_{\alpha} \in \tilde{\MF{a}}^{\BS{C}}$ ($\alpha \in R$)
defined by $\alpha(A)=B(A,A_{\alpha})$ for all $A \in \tilde{\MF{a}}^{\BS{C}}$
($B$ is the Killing form of $\MF{g}^{\BS{C}}$), and $R_{0}=\{\alpha \in R\mid \alpha|_{\MF{a}} = 0\}$.
Let $\varPhi$ be a simple root system of $R$.
Since $\varPhi_{0}(:=\varPhi \cap R_{0})$ is a simple root system of
$R_{0}$,
the black circles in the Satake diagram associated with $G/H$
with respect to $\MF{a}$ determines the Dynkin diagram of
$(\MF{z}_{\MF{g}}^{s})^{\BS{C}}$.
In particular,
the rank of $R_{0}$
is equal to the number of the black circles.

\subsection{The structures of the semisimple part of $\MF{z}_{\MF{g}}$
and $\MF{z}_{\MF{h}}$}

In this subsection,
we shall determine the structures
$\MF{z}_{\MF{g}}^{s}$
and $\MF{z}_{\MF{h}}$.
Suppose that the Satake diagrams
associated with
$G/H$ with respect to $\MF{a}$, $\MF{a}_{\MF{q}}$ and $\MF{a}_{\MF{p}}$
are compatible with one another.
Denote by $p_{\sigma}$ (resp.\,$p_{\theta}$)
the Satake involution of the Satake diagram associated with $G/H$
with respect to $\MF{a}_{\MF{q}}$ (resp.\,$\MF{a}_{\MF{p}}$).
Let $R_{0}=R^{1}_{0}\cup R_{0}^{2}\cdots\cup R_{0}^{k}$
be the irreducible decomposition of $R_{0}$.
Denote by $\MF{z}(\Gamma)$ ($\Gamma$ is a closed subsystem of $R_{0}$)
the subalgebra of $(\MF{z}^{s}_{\MF{g}})^{\BS{C}}$
generated by $\{\MF{g}^{\BS{C}}_{\alpha}\mid \alpha \in \Gamma\}$.
Set, for each $i \in \{1,\ldots,k\}$,
$\varPhi^{i}_{0}:=\varPhi_{0}\cap R_{0}^{i}$,
$\varPhi_{0,\MF{q}}^{i}:=\{\alpha  \in\varPhi_{0}^{i}\mid
\alpha|_{\MF{a}_{\MF{q}}}= 0\}$ and
$\varPhi_{0,\MF{p}}^{i}:=\{\alpha  \in\varPhi_{0}^{i}\mid
\alpha|_{\MF{a}_{\MF{p}}}= 0\}$.
Then, for each $i \in \{1, \ldots, k\}$,
$\varPhi^{i}_{0}$ satisfies one of the following.
\begin{enumerate}[C{a}se 1:]
\item $\varPhi_{0}^{i}=\varPhi_{0,\MF{q}}^{i}=\varPhi_{0,\MF{p}}^{i}$.
\item $\varPhi_{0}^{i}\setminus \varPhi_{0,\MF{p}}^{i}(\neq \emptyset)$
is $p_{\theta}$-invariant and $\varPhi_{0}^{i}=\varPhi_{0,\MF{q}}^{i}$.
\item $\varPhi_{0}^{i}=\varPhi_{0,\MF{p}}^{i}$ and
$\varPhi_{0}^{i}\setminus \varPhi_{0,\MF{q}}^{i} (\neq \emptyset)$ is
$p_{\sigma}$-invariant.
\item $\varPhi_{0}^{i}\setminus \varPhi_{0,\MF{p}}^{i}(\neq \emptyset)$
is $p_{\theta}$-invariant and $\varPhi_{0}^{i}\setminus \varPhi_{0,\MF{q}}^{i} (\neq \emptyset)$ is $p_{\sigma}$-invariant.
\item $\varPhi_{0}^{i}\setminus \varPhi_{0,\MF{p}}^{i}(\neq \emptyset)$
is not $p_{\theta}$-invariant.
\item $\varPhi_{0}^{i}\setminus \varPhi_{0,\MF{p}}^{i}(\neq \emptyset)$
is not $p_{\sigma}$-invariant.
\end{enumerate}

\noindent
Set $\MF{g}^{d}:=\MF{k}\cap\MF{h}+\sqrt{-1}\MF{p}\cap\MF{h}+\sqrt{-1}\MF{k}\cap\MF{q}+\MF{p}\cap\MF{q}(\subset \MF{g}^{\BS{C}})$.
Denote by $\tau$ (resp.\,$\tau^{d}$)
the conjugation of $\MF{g}^{\BS{C}}$
with respect to $\MF{g}$ (resp.\,$\MF{g}^{d}$).
For each $\alpha \in R$, the linear functions
$\tau \cdot \alpha$, $\tau^{d}\cdot \alpha$, $\theta \cdot \alpha$
and $\sigma \cdot \alpha$ defined by,
for all $A \in \tilde{\MF{a}}^{\BS{C}}$,
\begin{align*}
(\tau \cdot \alpha)(A)&=\overline{\alpha(\tau A)},\quad
(\tau^{d}\cdot \alpha)(A)=\overline{\alpha(\tau^{d}A)},\\
(\theta \cdot \alpha)(A) &= \alpha(\theta A),\quad
(\sigma \cdot \alpha)(A) = \alpha(\sigma A),
\end{align*}
are again elements of $R$.

\begin{Lem}\label{lem.inv1}
If $\varPhi_{0}^{i} \setminus \varPhi_{0,\MF{p}}^{i}(\neq \emptyset)$
is $p_{\theta}$-invariant, then $R_{0}^{i}$
is $\theta$-invariant.
\end{Lem}

\begin{proof}
Since $\theta \cdot \alpha = \alpha \in \varPhi_{0,\MF{p}}^{i}$
for all $\alpha \in \varPhi^{i}_{0,\MF{p}}$,
$\varPhi_{0,\MF{p}}^{i}$ is $\theta$-invariant.
Let $\alpha$ be a root in $\varPhi_{0}^{i}$
such that $\alpha|_{\MF{a}_{\MF{p}}} \not= 0$.
It follows from Lemma \ref{lem.satake1}
that if $\alpha \in \varPhi^{i}_{0}\setminus\varPhi^{i}_{0,\MF{p}}$ is disconnected with all roots of $\varPhi_{0,\MF{p}}^{i}$,
then we have $\theta \cdot \alpha = - p_{\theta}\alpha \in R_{0}^{i}$.
In the case where $\alpha \in \varPhi^{i}_{0}\setminus\varPhi^{i}_{0,\MF{p}}$
is connected with a root of $\varPhi_{0,\MF{p}}$,
$\theta \cdot \alpha$ has the form
$-p_{\theta}\alpha + \sum_{\beta \in \langle\alpha\rangle}\BS{Z}\beta$,
and $\langle\alpha\rangle$ is contained in $\varPhi_{0}^{i}$,
where $\langle\alpha\rangle$ is the union of the connected components
of $\varPhi_{0,\MF{p}}$ connected with $\alpha$.
Hence we have $\theta \cdot \alpha \in R_{0}^{i}$.
Since $\varPhi_{0}^{i}$ is a simple root system of $R_{0}^{i}$,
$R_{0}^{i}$ is $\theta$-invariant.
\end{proof}

\noindent
By a similar argument as Lemma \ref{lem.inv1} we have the following result.

\begin{Lem}\label{lem.inv2}
If $\varPhi_{0}^{i} \setminus \varPhi_{0,\MF{q}}^{i}(\neq \emptyset)$
is $p_{\sigma}$-invariant, then $R_{0}^{i}$
is $\sigma$-invariant.
\end{Lem}

\begin{Lem}\label{lem.ink}
For any $\alpha \in R_{0}$ with $\alpha|_{\MF{a}_{\MF{p}}}= 0$,
$\MF{g}^{\BS{C}}_{\alpha}\subset \MF{k}^{\BS{C}}$ holds.
\end{Lem}

\begin{proof}
For any $\alpha \in R_{0}$ with $\alpha|_{\MF{a}_{\MF{p}}} = 0$ ,
we have $\theta \cdot \alpha = \alpha$,
so that $\MF{g}_{\alpha}^{\BS{C}}$ is $\theta$-invariant.
Then, for any $X \in \MF{g}^{\BS{C}}_{\alpha}$,
we have $[A, X - \theta X] = 0$ for all $A \in \MF{a}_{\MF{p}}$.
By the maximality of $\MF{a}_{\MF{p}}$ in $\MF{p}$,
$X - \theta X \in \MF{a}_{\MF{p}}^{\BS{C}}$ holds.
Since $X -\theta X \in \MF{a}_{\MF{p}}^{\BS{C}}\cap\MF{g}_{\alpha}^{\BS{C}}
=\{0\}$,
we have $\theta X =X$.
Hence $\MF{g}_{\alpha}^{\BS{C}} \subset \MF{k}^{\BS{C}}$
holds.
\end{proof}

\noindent
By a similar argument as Lemma \ref{lem.ink} we have
the following fact.

\begin{Lem}\label{lem.inh}
For any $\alpha \in R_{0}$ with $\alpha|_{\MF{a}_{\MF{q}}}= 0$,
$\MF{g}^{\BS{C}}_{\alpha}\subset \MF{h}^{\BS{C}}$
holds.
\end{Lem}

\noindent
Note that,
for each $\alpha \in R$,
$\tau \cdot \alpha = -\theta \cdot \alpha$
and $\tau^{d} \cdot \alpha =- \sigma \cdot \alpha$ hold.
Hence it is shown that,
for any closed subsystem $\Gamma$,
$\Gamma$ is $\tau$-invariant $($resp.\,$\tau^{d}$-invariant$)$
if and only if $\Gamma$ is $\theta$-invariant $($resp.\,$\sigma$-invariant$)$.
We also have the following fact.

\begin{Lem}\label{lem.real}
Let $\alpha$ be a root in $R_{0}$.
If $\alpha$ satisfies $\alpha|_{\MF{a}_{\MF{p}}}= 0$ $($resp.\,$\alpha|_{\MF{a}_{\MF{q}}}= 0$ $)$,
then $\MF{g}_{\alpha}^{\BS{C}}+\MF{g}_{-\alpha}^{\BS{C}}$
is $\tau$-invariant $($resp.\,$\tau^{d}$-invariant$)$.
\end{Lem}

\noindent
In the sequel, we shall determine
the structures of the irreducible factors
of $\MF{z}^{s}_{\MF{g}}$ and 
associated with $\varPhi^{i}_{0}$
and these $\MF{h}$-parts.

\medskip

\paragraph{Case 1} It is clear that $R_{0}^{i}$
is invariant under $\sigma$ and $\theta$.
It follows from Lemma \ref{lem.ink} and Lemma \ref{lem.inh} that
$\MF{z}(R_{0}^{i})$ is a subalgebra of $\MF{k}^{\BS{C}}\cap\MF{h}^{\BS{C}}$.
Set $\MF{z}_{\MF{g}}^{i}:=\MF{z}(R_{0}^{i})\cap \MF{g}$,
which is a real form of $\MF{z}(R_{0}^{i})$ by using Lemma \ref{lem.real}.
Moreover, we have
$\MF{z}^{i}_{\MF{g}} =\MF{z}^{i}_{\MF{g}}\cap \MF{h} \subset \MF{k}\cap\MF{h}$.
In particular, $\MF{z}^{i}_{\MF{g}}$
is a compact real form of $\MF{z}(R_{0}^{i})$,
which is uniquely determined (up to isomorphism) by the Dynkin
diagram of $\varPhi_{0}^{i}$.

\medskip

\paragraph{Case 2} It follows from Lemma \ref{lem.inv1}
that $R_{0}^{i}$ is $\theta$-invariant.
This implies that $\MF{z}(R_{0}^{i})$
is $\tau$-invariant.
By using Lemma \ref{lem.inh}, we have
$\MF{z}_{\MF{g}}^{i}:=\MF{z}(R_{0}^{i})\cap\MF{g} \subset \MF{h}$.
Moreover, $\theta|_{\MF{z}^{i}_{\MF{g}}}$
is a Cartan involution of $\MF{z}^{i}_{\MF{g}}$.
Then $(\MF{z}^{i}_{\MF{g}},\MF{z}^{i}_{\MF{g}}\cap\MF{k})$
is an irreducible Riemannian symmetric pair (of noncompact-type).
Moreover, its Satake diagram is given by
the Dynkin diagram of $\varPhi_{0}^{i}$
and $p_{\theta}|_{\varPhi^{i}_{0}\setminus \varPhi^{i}_{0,\MF{p}}}$.

\medskip

\paragraph{Case 3}
By using Lemma \ref{lem.ink},
we have $\MF{z}(R_{0}^{i}) \subset \MF{k}^{\BS{C}}$.
Then
$\MF{z}^{i}_{\MF{g}}:=\MF{z}(R_{0}^{i})
\cap\MF{g} \subset \MF{k}$
and $\MF{z}^{i}_{\MF{g}}\cap\MF{h} \subset \MF{k}\cap\MF{h}$ hold.
Note that $\sigma|_{\MF{z}^{i}_{\MF{g}}}$ is not trivial.
Then $(\MF{z}^{i}_{\MF{g}}, \MF{z}^{i}_{\MF{g}}\cap\MF{h})$
is an irreducible Riemannian symmetric pair (of compact-type).
Moreover, its Satake diagram is given by 
the Dynkin diagram of $\varPhi_{0}^{i}$
and $p_{\sigma}|_{\varPhi^{i}_{0}\setminus \varPhi^{i}_{0,\MF{q}}}$.

\medskip

\paragraph{Case 4}
In this case, $\MF{z}^{i}_{\MF{g}}:=\MF{z}(R_{0}^{i})\cap\MF{g}$
is a noncompact subalgebra of $\MF{g}$,
and $(\MF{z}^{i}_{\MF{g}},\MF{z}^{i}_{\MF{g}}\cap\MF{h})$
is an irreducible semisimple symmetric pair.
Since
$\MF{z}_{\MF{g}}\cap(\MF{p}\cap\MF{q})=\MF{a} \subset \MF{z}_{\MF{g}}^{c}$
holds,
we have $\MF{z}^{i}_{\MF{g}}\cap(\MF{p}\cap\MF{q})\subset\MF{z}^{s}_{\MF{g}}\cap(\MF{p}\cap\MF{q}) = \{0\}$.
This contradicts the fact that any noncompact irreducible
semisimple symmetric pair has the split rank greater than or
equal to one.
Hence Case 4 cannot occur.

\medskip

\paragraph{Case 5}
Denote by $\tilde{R}^{i}_{0}$
is the smallest $\theta$-invariant closed subsystem of $R_{0}$
containing $R^{i}_{0}$.
Then $\tilde{R}^{i}_{0}$ is not connected,
and $\MF{z}(\tilde{R}^{i}_{0})$
is $\tau$-invariant.
Set $\tilde{\MF{z}}^{i}_{\MF{g}}:=\MF{z}(\tilde{R}_{0})\cap\MF{g}$.
Since $\theta|_{\tilde{\MF{z}}^{i}_{\MF{g}}}$
is not trivial,
$(\tilde{\MF{z}}^{i}_{\MF{g}}, \tilde{\MF{z}}^{i}_{\MF{g}}\cap\MF{k})$
is an irreducible Riemannian symmetric pair.
By the classification of irreducible Riemannian symmetric pairs,
$(\tilde{\MF{z}}^{i}_{\MF{g}}, \tilde{\MF{z}}^{i}_{\MF{g}}\cap\MF{k})$
is isomorphic to a Riemannian symmetric pair
$(\MF{l}^{\BS{C}},\MF{\MF{l}})$ (of compact real form type),
where $\MF{l}$ is a simple compact Lie algebra.
Therefore we have
$\varPhi^{i}_{0,\MF{p}} = \emptyset$.
Moreover, the Dynkin diagram of $\MF{l}$ is equal to $\varPhi^{i}_{0}$.
It follows from Lemma 2.8 in \cite{MR810638}
that $\varPhi^{i}_{0} = \varPhi^{i}_{0,\MF{q}}$ holds.
Hence $\tilde{\MF{z}}^{i}_{\MF{g}}$ is contained in $\MF{h}$.

\medskip

\paragraph{Case 6}
Denote by $\hat{R}^{i}_{0}$
the smallest $\sigma$-invariant closed subsystem of $R_{0}$
containing $R^{i}_{0}$, which is not connected.
Then $\MF{z}(\hat{R}^{i}_{0})$
is $\tau^{d}$-invariant.
Set $\hat{\MF{z}}^{i}_{\MF{g}^{d}}:=\MF{z}(\hat{R}^{i}_{0})\cap\MF{g}^{d}$.
Since $\sigma|_{\hat{\MF{z}}^{i}_{\MF{g}^{d}}}$
is not trivial,
$(\hat{\MF{z}}^{i}_{\MF{g}^{d}}, \hat{\MF{z}}^{i}_{\MF{g}^{d}}\cap\MF{k}^{d})$
is isomorphic to a Riemannian symmetric pair $(\MF{m}^{\BS{C}},\MF{m})$ (of
compact real form type),
where $\MF{m}$ is a simple compact Lie algebra.
Note that the Dynkin diagram of $\MF{m}$ is equal to $\varPhi^{i}_{0}$.
Then we have $\varPhi^{i}_{0,\MF{q}}=\emptyset$.
It follows from Lemma 2.8 in \cite{MR810638}
that $\varPhi^{i}_{0} = \varPhi^{i}_{0,\MF{p}}$ holds.
Hence $\MF{z}(\hat{R}^{i}_{0})$ is $\tau$-invariant,
and $\hat{\MF{z}}^{i}_{\MF{g}}:=\MF{z}(\hat{R}^{i}_{0})\cap\MF{g}$
is subalgebra of $\MF{k}$.
Then $(\hat{\MF{z}}^{i}_{\MF{g}},
\hat{\MF{z}}^{i}_{\MF{g}}\cap\MF{h})$
is isomorphic to $(\MF{m}+\MF{m},\MF{m})$.

\medskip

\noindent
Here, we give a recipe to determine
the hyperbolic principal isotropy subalgebra
as follows.

\begin{Recipe}[hyperbolic principal isotropy subalgebras]\label{recipe.hprin}
 
Let $(\MF{g},\MF{h})$ be an irreducible semisimple symmetric pair.

\begin{enumerate}[Step 1.]

\item We calculate all the irreducible components $\varPhi^{i}_{0}$
of $\varPhi_{0}$ by using the Satake diagram $S(\MF{g},\MF{h},\MF{a})$.

\item For each $i$, we investigate whether $\varPhi^{i}_{0}$
corresponds to either Case 1--3, 5 or 6 by using the Satake diagrams
$S(\MF{g},\MF{h},\MF{a}_{\MF{q}})$ and $S(\MF{g},\MF{k},\MF{a}_{\MF{p}})$.

\item For each $i$, we determine
the following subalgebra of $\MF{z}^{s}_{\MF{g}}$
associated with $\varPhi^{i}_{0}$.

\begin{enumerate}[C{a}se 1:]
\item We determine $\MF{z}^{i}_{\MF{g}} (\subset \MF{h})$
by investigating the Dynkin diagram of $\varPhi^{i}_{0}$.

\item We determine $(\MF{z}^{i}_{\MF{g}},\MF{z}^{i}_{\MF{g}}\cap\MF{k})$
by investigating the Satake diagram obtained from
the Dynkin diagram of $\varPhi^{i}_{0}$ and
$p_{\theta}|_{\varPhi^{i}_{0}\setminus\varPhi^{i}_{0,\MF{p}}}$.
Note that $\MF{z}^{i}_{\MF{g}}$ is contained in $\MF{h}$.

\item We determine $(\MF{z}^{i}_{\MF{g}},\MF{z}^{i}_{\MF{g}}\cap\MF{h})$
by investigating the Satake diagram obtained from
the Dynkin diagram of $\varPhi^{i}_{0}$ and
$p_{\sigma}|_{\varPhi^{i}_{0}\setminus\varPhi^{i}_{0,\MF{q}}}$.

\setcounter{enumii}{4}

\item We calculate
$\varPhi^{i}_{0}\cup p_{\theta}\varPhi^{i}_{0}(=:\tilde{\varPhi}^{i}_{0})$
by using the Satake diagrams $S(\MF{g},\MF{h},\MF{a})$
and $S(\MF{g},\MF{k},\MF{a}_{\MF{p}})$.
We determine
$(\tilde{\MF{z}}^{i}_{\MF{g}},\tilde{\MF{z}}^{i}_{\MF{g}}\cap\MF{k})$
by investigating the Satake diagram obtained from the Dynkin
diagram of $\tilde{\varPhi}^{i}_{0}$ and
$p_{\theta}|_{\tilde{\varPhi}^{i}_{0}}$.
Note that $\tilde{\MF{z}}^{i}_{\MF{g}}$ is contained in $\MF{h}$.

\item We calculate $\varPhi^{i}_{0}\cup p_{\sigma}\varPhi^{i}_{0}(=:\hat{\varPhi}^{i}_{0})$ by using the Satake diagrams $S(\MF{g},\MF{h},\MF{a})$
and $S(\MF{g},\MF{h},\MF{a}_{\MF{q}})$.
We determine $(\hat{\MF{z}}^{i}_{\MF{g}},\hat{\MF{z}}_{\MF{g}}\cap\MF{h})$
by investigating the Satake diagram obtained from the Dynkin
diagram of $\hat{\varPhi}^{i}_{0}$ and $p_{\sigma}|_{\hat{\varPhi}^{i}_{0}}$.
\end{enumerate}
\item We calculate the following dimensions.
\begin{align*}
\dim \tilde{\MF{a}}\cap(\MF{k}\cap\MF{h})
&= \OPE{rank}\MF{g}^{\BS{C}} - \OPE{rank}(\MF{g},\MF{h})
- \OPE{rank}(\MF{g},\MF{k}) +\OPE{s-rank}(\MF{g},\MF{h}),\\
\dim \tilde{\MF{a}}^{c}\cap(\MF{k}\cap\MF{h})
&= \dim \tilde{\MF{a}}\cap(\MF{k}\cap\MF{h})
-\dim \tilde{\MF{a}}^{s}\cap(\MF{k}\cap\MF{h}),\\
\dim \tilde{\MF{a}}^{c}\cap(\MF{p}\cap\MF{h})
&= \dim \tilde{\MF{a}}\cap(\MF{p}\cap\MF{h})
-\dim \tilde{\MF{a}}^{s}\cap(\MF{p}\cap\MF{h}).
\end{align*}
\item From the data in Steps 1--4, we determine $\MF{z}_{\MF{h}}$.

\end{enumerate}
\end{Recipe}

\begin{Ex}[$(\MF{g},\MF{h})=(\MF{su}(2p, 2(n-p)),\MF{sp}(p,n-p))$]\label{exam.1}
First,
we calculate $\MF{z}_{\MF{h}}$ in the case of $n>2p$.
We give the Satake diagrams associated with $(\MF{g},\MF{h})$
with respect to $\MF{a}, \MF{a}_{\MF{p}}$ and $\MF{a}_{\MF{q}}$,
respectively (see Table \ref{table.satake1}).

\begin{table}[htbp]
\centering
\caption{$(\MF{g},\MF{h})=(\MF{su}(2p, 2(n-p)),\MF{sp}(p,n-p)) (n>2p)$}\label{table.satake1}
\begin{tabular}{|c|c|}
\hline
the Satake diagram $S(\MF{g},\MF{h},\MF{a})$
&
the Satake diagram $S(\MF{g},\MF{h},\MF{a}_{\MF{p}})$
\\
\hline
\begin{xy}
\tiny
\ar@{-}(0,5) *+!D{\alpha_{1}}*{\bullet};(7,5) *+!D{\alpha_{2}}*{\circ}="R2"
\ar@{-}"R2";(14,5)*+!D{\alpha_{3}}*{\bullet}="R3"
\ar@{-}"R3";(17,5)
\ar@{.}(18,5);(19,5)
\ar@{-}(20,5);(23,5)*+!D{\alpha_{2p}}*{\circ}="R2p"
\ar@{-}"R2p";(30,5)*+!L{\alpha_{2p+1}}*{\bullet}="R2p+1"
\ar@{-}(0,-5) *+!U{\alpha_{2n-1}}*{\bullet};(7,-5) *++!U{\alpha_{2n-2}}*{\circ}="R2n-2"
\ar@{-}"R2n-2";(14,-5)*+!U{\alpha_{2n-3}}*{\bullet}="R2n-3"
\ar@{-}"R2n-3";(17,-5)
\ar@{.}(18,-5);(19,-5)
\ar@{-}(20,-5);(23,-5)*++!U{\alpha_{2n-2p}}*{\circ}="R2n-2p"
\ar@{-}"R2n-2p";(30,-5)*+!L{\alpha_{2n-2p-1}}*{\bullet}="R2n-2p-1"
\ar@{-}"R2p+1";(30,2)*+!L{\alpha_{2p+2}}*{\bullet}="R2p+2"
\ar@{-}"R2p+2";(30,0)
\ar@{.}(30,-1);(30,-2)
\ar@{-}(30,-3);"R2n-2p-1"
\ar@/_/ @{<->}"R2";"R2n-2"
\ar@/_/ @{<->}"R2p";"R2n-2p"
\end{xy}

&

\begin{xy}
\tiny
\ar@{-}(0,5) *+!D{\alpha_{1}}*{\circ}="R1";(7,5) *+!D{\alpha_{2}}*{\circ}="R2"
\ar@{-}"R2";(14,5)*+!D{\alpha_{3}}*{\circ}="R3"
\ar@{-}"R3";(17,5)
\ar@{.}(18,5);(19,5)
\ar@{-}(20,5);(23,5)*+!D{\alpha_{2p}}*{\circ}="R2p"
\ar@{-}"R2p";(30,5)*+!L{\alpha_{2p+1}}*{\bullet}="R2p+1"
\ar@{-}(0,-5) *+!U{\alpha_{2n-1}}*{\circ}="R2n-1";(7,-5) *++!U{\alpha_{2n-2}}*{\circ}="R2n-2"
\ar@{-}"R2n-2";(14,-5)*+!U{\alpha_{2n-3}}*{\circ}="R2n-3"
\ar@{-}"R2n-3";(17,-5)
\ar@{.}(18,-5);(19,-5)
\ar@{-}(20,-5);(23,-5)*++!U{\alpha_{2n-2p}}*{\circ}="R2n-2p"
\ar@{-}"R2n-2p";(30,-5)*+!L{\alpha_{2n-2p-1}}*{\bullet}="R2n-2p-1"
\ar@{-}"R2p+1";(30,2)*+!L{\alpha_{2p+2}}*{\bullet}="R2p+2"
\ar@{-}"R2p+2";(30,0)
\ar@{.}(30,-1);(30,-2)
\ar@{-}(30,-3);"R2n-2p-1"
\ar@/_/ @{<->}"R1";"R2n-1"
\ar@/_/ @{<->}"R2";"R2n-2"
\ar@/_/ @{<->}"R3";"R2n-3"
\ar@/_/ @{<->}"R2p";"R2n-2p"
\end{xy}
\\
\hline
\multicolumn{2}{|c|}{the Satake diagram
 $S(\MF{g},\MF{h},\MF{a}_{\MF{q}})$}\\
\hline
\multicolumn{2}{|c|}{
\begin{xy}
\tiny
\ar@{-}(0,5) *+!D{\alpha_{1}}*{\bullet};(7,5) *+!D{\alpha_{2}}*{\circ}="R2"
\ar@{-}"R2";(14,5)*+!D{\alpha_{3}}*{\bullet}="R3"
\ar@{-}"R3";(17,5)
\ar@{.}(18,5);(19,5)
\ar@{-}(20,5);(23,5)*+!D{\alpha_{2p}}*{\circ}="R2p"
\ar@{-}"R2p";(30,5)*+!L{\alpha_{2p+1}}*{\bullet}="R2p+1"
\ar@{-}(0,-5) *+!U{\alpha_{2n-1}}*{\bullet};(7,-5) *++!U{\alpha_{2n-2}}*{\circ}="R2n-2"
\ar@{-}"R2n-2";(14,-5)*+!U{\alpha_{2n-3}}*{\bullet}="R2n-3"
\ar@{-}"R2n-3";(17,-5)
\ar@{.}(18,-5);(19,-5)
\ar@{-}(20,-5);(23,-5)*++!U{\alpha_{2n-2p}}*{\circ}="R2n-2p"
\ar@{-}"R2n-2p";(30,-5)*+!L{\alpha_{2n-2p-1}}*{\bullet}="R2n-2p-1"
\ar@{-}"R2p+1";(30,2)*+!L{\alpha_{2p+2}}*{\circ}="R2p+2"
\ar@{-}"R2p+2";(30,0)
\ar@{.}(30,-1);(30,-2)
\ar@{-}(30,-3);"R2n-2p-1"
\end{xy}
}\\
\hline
\end{tabular}
\end{table}

\begin{enumerate}[Step 1.]
\item Set $\varPhi^{i}_{0}:=\{\alpha_{2i-1}\},\,\varPhi^{p+i}_{0}:=\{\alpha_{2n-(2i-1)}\}$ for $1\leq i \leq p$,
and $\varPhi^{2p+1}_{0} :=\{\alpha_{2p+1},\ldots,\alpha_{2n-(2p+1)}\}$.
From $S(\MF{g},\MF{h},\MF{a})$,
$\varPhi^{j}_{0}$ is an irreducible component of $\varPhi_{0}$
for $1 \leq j \leq 2p+1$.

\item It follows from $S(\MF{g}, \MF{h}, \MF{a}_{\MF{p}})$
and $S(\MF{g}, \MF{h}, \MF{a}_{\MF{q}})$ that
$\varPhi^{i}_{0}$ and $\varPhi^{p+i}_{0}$ correspond to Case 5,
and
$\varPhi^{2p+1}_{0}$ corresponds to Case 3.
We obtain
$\tilde{\varPhi}^{i}_{0}=\varPhi^{i}_{0}\cup p_{\theta}\varPhi^{i}_{0}=\varPhi^{i}_{0}\cup\varPhi^{p+i}_{0}$ for $1 \leq i \leq p$.

\item For $1 \leq i \leq p$,
 $(\tilde{\MF{z}}^{i}_{\MF{g}}, \tilde{\MF{z}}^{i}_{\MF{g}}\cap\MF{k})$
is isomorphic to $(\MF{sl}(2,\BS{C}),\MF{su}(2))$
because their Satake diagram are ``$\circ \leftrightarrow \circ$''.
We also have $(\tilde{\MF{z}}^{2p+1}_{\MF{g}}, \tilde{\MF{z}}^{2p+1}_{\MF{g}}\cap\MF{h})\cong(\MF{su}(2(n-2p)),\MF{sp}(n-2p))$.

\item We obtain $\dim \tilde{\MF{a}}\cap(\MF{k}\cap\MF{h})=0,\,
\dim \tilde{\MF{a}}\cap(\MF{p}\cap\MF{h})=0$.
Here, we have
$\OPE{rank}\MF{g}^{\BS{C}} = 2n-1$,
$\OPE{s-rank}(\MF{g},\MF{h})=p$,
$\OPE{rank}(\MF{g},\MF{h})=n-1$,
$\OPE{rank}(\MF{g},\MF{k})=2p$,
$\dim \tilde{\MF{a}}^{s}\cap(\MF{k}\cap\MF{h})=n-p$, and
$\dim \tilde{\MF{a}}^{s}\cap(\MF{p}\cap\MF{h})=p$.
\item 
It follows from Step 1--4 that
$\MF{z}_{\MF{h}}$
is isomorphic to
$\MF{sl}(2,\BS{C})^{p}+\MF{sp}(n-2p) (n > 2p)$.
In the case of $n=2p$,
we have $\MF{z}_{\MF{h}}\cong\MF{sl}(2,\BS{C})^{p}$,
by the similar calculation.
Hence we have
$\MF{z}_{\MF{h}}\cong\MF{sl}(2, \BS{C})^{p}+\MF{sp}(n-2p) (n
\geq 2p)$.
Note that we have
$\MF{sl}(2,\BS{C})\cong\MF{so}(3,1)\subset \MF{so}(4,1)\cong \MF{sp}(1,1)$
by using the list of special isomorphisms (see,\,Section 4 of Chapter X in \cite{MR1834454}).
\end{enumerate}
\end{Ex}

\begin{Ex}[$(\MF{g}, \MF{h})=(\MF{sl}(n,\BS{C}),\MF{sl}(n,\BS{R}))$]\label{exam.2}

First,
we determine $\MF{z}_{\MF{h}}$ in the case of $n=2m$.
We give the Satake diagrams associated with $(\MF{g},\MF{h})$
with respect to $\MF{a}, \MF{a}_{\MF{p}}$ and $\MF{a}_{\MF{q}}$,
respectively (see Table \ref{table.satake2}).

\begin{table}[htbp]
\centering
\caption{$(\MF{g},\MF{h})=(\MF{sl}(2m,\BS{C}),\MF{sl}(2m,\BS{R}))$}\label{table.satake2}
\begin{tabular}{|c|c|}
\hline
\multirow{2}{*}{the Satake diagram $S(\MF{g},\MF{h},\MF{a})$}
&
the Satake diagrams $S(\MF{g},\MF{h},\MF{a}_{\MF{p}})$
\\
&the Satake diagram $S(\MF{g},\MF{h},\MF{a}_{\MF{q}})$\\
\hline
\begin{xy}
\tiny
\ar@{-}(0,10.5)*+!D{\alpha_{1}}*{\circ}="R1";(7,10.5)*+!D{\alpha_{2}}*{\circ}="R2"\ar@{-}"R2";(10,10.5)
\ar@{.}(11,10.5);(12,10.5)
\ar@{-}(13,10.5);(16,10.5)*+!D{\alpha_{m-1}}*{\circ}="Rm-1"
\ar@{-}"Rm-1";(20,7)*+!L{~\alpha_{m}}*{\circ}="Rm"
\ar@{-}(0,3.5)*{\circ}="R2m-1";(7,3.5)*{\circ}="R2m-2"\ar@{-}"R2m-2";(10,3.5)
\ar@{.}(11,3.5);(12,3.5)
\ar@{-}(13,3.5);(16,3.5)*{\circ}="Rm+1"
\ar@{-}"Rm";"Rm+1"

\ar@{-}(0,-10.5)*+!U{}*{\circ}="RR1";(7,-10.5)*{\circ}="RR2"\ar@{-}"RR2";(10,-10.5)
\ar@{.}(11,-10.5);(12,-10.5)
\ar@{-}(13,-10.5);(16,-10.5)*{\circ}="RRm-1"
\ar@{-}"RRm-1";(20,-7)*{\circ}="RRm"
\ar@{-}(0,-3.5)*{\circ}="RR2m-1";(7,-3.5)*{\circ}="RR2m-2"\ar@{-}"RR2m-2";(10,-3.5)
\ar@{.}(11,-3.5);(12,-3.5)
\ar@{-}(13,-3.5);(16,-3.5)*{\circ}="RRm+1"
\ar@{-}"RRm";"RRm+1"

\ar@/_/ @{<->}"R1";"R2m-1"
\ar@/_/ @{<->}"R2";"R2m-2"
\ar@/_/ @{<->}"Rm-1";"Rm+1"

\ar@/^/ @{<->}"RR1";"RR2m-1"
\ar@/^/ @{<->}"RR2";"RR2m-2"
\ar@/^/ @{<->}"RRm-1";"RRm+1"

\ar@/_/ @{<->}"R2m-1";"RR2m-1"
\ar@/_/ @{<->}"R2m-2";"RR2m-2"
\ar@/_/ @{<->}"Rm+1";"RRm+1"

\ar@/^/ @{<->}"Rm";"RRm"

\end{xy}

&

\begin{xy}
\tiny
\ar@{-}(0,10.5)*+!D{\alpha_{1}}*{\circ}="R1";(7,10.5)*+!D{\alpha_{2}}*{\circ}="R2"\ar@{-}"R2";(10,10.5)
\ar@{.}(11,10.5);(12,10.5)
\ar@{-}(13,10.5);(16,10.5)*+!D{\alpha_{m-1}}*{\circ}="Rm-1"
\ar@{-}"Rm-1";(20,7)*+!L{~\alpha_{m}}*{\circ}="Rm"
\ar@{-}(0,3.5)*{\circ}="R2m-1";(7,3.5)*{\circ}="R2m-2"\ar@{-}"R2m-2";(10,3.5)
\ar@{.}(11,3.5);(12,3.5)
\ar@{-}(13,3.5);(16,3.5)*{\circ}="Rm+1"
\ar@{-}"Rm";"Rm+1"

\ar@{-}(0,-10.5)*{\circ}="RR1";(7,-10.5)*{\circ}="RR2"\ar@{-}"RR2";(10,-10.5)
\ar@{.}(11,-10.5);(12,-10.5)
\ar@{-}(13,-10.5);(16,-10.5)*{\circ}="RRm-1"
\ar@{-}"RRm-1";(20,-7)*{\circ}="RRm"
\ar@{-}(0,-3.5)*{\circ}="RR2m-1";(7,-3.5)*{\circ}="RR2m-2"\ar@{-}"RR2m-2";(10,-3.5)
\ar@{.}(11,-3.5);(12,-3.5)
\ar@{-}(13,-3.5);(16,-3.5)*{\circ}="RRm+1"
\ar@{-}"RRm";"RRm+1"

\ar@/_/ @{<->}"R1";"RR2m-1"
\ar@/_/ @{<->}"R2";"RR2m-2"
\ar@/_/ @{<->}"Rm-1";"RRm+1"

\ar@/^/ @{<->}"RR1";"R2m-1"
\ar@/^/ @{<->}"RR2";"R2m-2"
\ar@/^/ @{<->}"RRm-1";"Rm+1"

\ar@/^/ @{<->}"Rm";"RRm"

\end{xy}
\\
\hline
\end{tabular}
\end{table}

\begin{enumerate}[Step 1.]
\item From $S(\MF{g},\MF{h},\MF{a})$
we have $\varPhi_{0}=\emptyset$.

\item It is clear that $\varPhi_{0}$ corresponds to Case 1.

\item It is clear that $\MF{z}^{s}_{\MF{g}}\cap\MF{h}=\{0\}$.

\item We have $\dim \tilde{\MF{a}}^{c}\cap(\MF{k}\cap\MF{h})=m$ and
$\dim \tilde{\MF{a}}^{c}\cap(\MF{p}\cap\MF{h})=m-1$
by using $\OPE{rank}\MF{g}^{\BS{C}} = 2(2m-1)$,
$\OPE{s-rank}(\MF{g},\MF{h})=m$,
$\OPE{rank}(\MF{g},\MF{h})=2m-1$, and
$\OPE{rank}(\MF{g},\MF{k})=2m-1$.
\item
It follows from Step 1--4
that $\MF{z}_{\MF{h}}$
is isomorphic to
$\BS{R}^{m-1}+\MF{so}(2)^{m}\,(n=2m)$.
In the case of $n=2m+1$,
we have $\MF{z}_{\MF{h}}\cong\BS{R}^{m}+\MF{so}(2)^{m}$,
by a similar calculation as above.
Hence we have
$\MF{z}_{\MF{h}}\cong\BS{R}^{[(n-1)/2]}+\MF{so}(2)^{[n/2]}$,
where $[(n-1)/2]$ (resp.~$[n/2]$) is
the greatest integer less than or equal to $(n-1)/2$
(resp.\,$n/2$).
\end{enumerate}
\end{Ex}

\section{Determination of local orbit types}\label{Sec.det}

Let $(\MF{g}, \MF{h})$
be a semisimple symmetric pair and $\sigma$ be an involution of $\MF{g}$
with $\MF{h}=\OPE{Ker}(\sigma - \OPE{id})$.
Suppose that $\theta$
is a Cartan involution of $\MF{g}$
commuting with $\sigma$,
and $\MF{a}$ is maximal abelian subspace of $\MF{p}\cap\MF{q}$.
Denote by $\Delta$ the restricted root system of $(\MF{g},\MF{h})$
with respect to $\MF{a}$.
Set $\Delta^{a}:=\{\lambda \in \Delta\mid \MF{g}_{\lambda} \cap \MF{h}^{a}
\neq \{0\}\}$.
Denote by $\MC{W}(\Delta)$ (resp.~$\MC{W}(\Delta^{a})$)
the Weyl group of $\Delta$
(resp.~$\Delta^{a}:=\{\lambda \in \Delta\mid m^{+}(\lambda) >0\}$).
Let $w_{1}, \ldots, w_{l}$
be a complete system of representatives
for $\MC{W}(\Delta)/\MC{W}(\Delta^{a})$.
From Theorem \ref{thm.main},
we can determine the set of all local orbit types of the
hyperbolic orbits by investigating
$\big\{[\MF{h}_{\Theta}]\mid\Theta \subset w_{i} \cdot \varPsi\big\}$
for all $i \in \{1,\ldots,l\}$.
By using the recipe in \cite{B},
we can determine $\big\{[\MF{h}_{\Theta}]\mid\Theta \subset w_{i} \cdot \varPsi\big\}$ for each $i \in \{1,\ldots,l\}$.
In fact,
we can determine $\MF{h}_{\Theta}$ for each
$\Theta \subset w_{i} \cdot \varPsi$,
by using
the hyperbolic principal isotropy subalgebra and
a subsymmetric pair of $(\MF{g}, \MF{h})$
associated with
$\Delta_{\Theta}(:=\Delta\cap\sum_{\lambda \in \Theta}\BS{R}\lambda)$
(see, page 315 in \cite{B} for more detail).

\begin{Ex}[$(\MF{g},\MF{h})=(\MF{sl}(4,\BS{R}),\MF{so}(2,2))$]\label{ex.isotro1}
The Dynkin diagram of the restricted root system of $(\MF{g},\MF{h})$
is the following.

\begin{center}
\small
\begin{tabular}{cc}
\begin{xy}
\ar@{-}(0,5) *+!D{\lambda_{1}}*{\circ};(7,5) *+!D{\lambda_{2}}*{\circ}="R2"
\ar@{-}"R2";(14,5)*+!D{\lambda_{3}}*{\circ}
\end{xy}&
$\begin{pmatrix}
m^{+}(\lambda_{i}) & m^{+}(2\lambda_{i}) \\
m^{-}(\lambda_{i}) & m^{-}(2\lambda_{i})
\end{pmatrix}=
\begin{cases}
 \begin{pmatrix}
1 & 0\\
0 & 0 
\end{pmatrix} & (i=1,3)\\
\begin{pmatrix}
0 & 0\\
1 & 0 
\end{pmatrix} & (i=2)
\end{cases}$
\end{tabular}
\end{center}

\noindent
Set $\varPsi:=\{\lambda_{1},\lambda_{2},\lambda_{3}\}
=\{e_{1}-e_{2},e_{2}-e_{3},e_{3}-e_{4}\}$.
It follows from Table \ref{table.hprin}
that the HPIS is equal to $\{0\}$.
Moreover, from Example \ref{ex.weyl1}
we have
\[
\MC{W}(\Delta)/\MC{W}(\Delta^{a})= \{\OPE{id},s_{2},s_{1}s_{2},s_{3}s_{2},s_{1}s_{3}s_{2},s_{2}s_{1}s_{3}s_{2}\}.
\]
By using Theorem \ref{thm.main} all the possible isotropy subalgebras of the
hyperbolic orbits for the s-representation associated with $(\MF{g},\MF{h})$
are equal to $\MF{h}_{\Theta}$'s for all $\Theta \subset w\cdot\varPsi, w \in \MC{W}(\Delta)/\MC{W}(\Delta^{a})$.
For each $w \in \MC{W}(\Delta)/\MC{W}(\Delta^{a})$,
we can determine $\{[\MF{h}_{\Theta}]\mid \Theta \subset w\cdot\varPsi\}$
by using the recipe in \cite{B}.
In Table \ref{table.ex1},
we shall give the lists of
all the possible isotropy subalgebras of the hyperbolic
orbits for the s-representation associated with $(\MF{g}, \MF{h})$.
\end{Ex}

\clearpage

\begin{table}[h]
\footnotesize
\centering
\caption{All the possible isotropy subalgebras for $(\MF{sl}(4,\BS{R}),\MF{so}(2,2))$}\label{table.ex1}
\begin{tabular}{|c|c|}
\hline
$\Theta(\subset \varPsi)$ & $\MF{h}_{\Theta}$ \\
\hline
\hline
$\varPsi$ & $\MF{so}(2,2)$\\
$\{e_{2}-e_{3},e_{3}-e_{4}\}$ & $\MF{so}(1,2)$\\
$\{e_{1}-e_{2},e_{3}-e_{4}\}$ & $\MF{so}(2,0)+\MF{so}(0,2)$\\
$\{e_{1}-e_{2},e_{2}-e_{3}\}$ & $\MF{so}(2,1)$\\
$\{e_{3}-e_{4}\}$ & $\MF{so}(0,2)$\\
$\{e_{2}-e_{3}\}$ & $\MF{so}(1,1)$\\
$\{e_{1}-e_{2}\}$ & $\MF{so}(2,0)$\\
$\emptyset$ & $\{0\}$\\
\hline
\end{tabular}
\begin{tabular}{|c|c|}
\hline
$\Theta(\subset s_{2}\varPsi)$ & $\MF{h}_{\Theta}$ \\
\hline
\hline
$s_{2}\varPsi$ & $\MF{so}(2,2)$\\
$\{e_{3}-e_{2},e_{2}-e_{4}\}$ & $\MF{so}(1,2)$\\
$\{e_{1}-e_{3},e_{2}-e_{4}\}$ & $\MF{so}(1,1)+\MF{so}(1,1)$\\
$\{e_{1}-e_{3},e_{3}-e_{2}\}$ & $\MF{so}(2,1)$\\
$\{e_{2}-e_{4}\}$ & $\MF{so}(1,1)$\\
$\{e_{3}-e_{2}\}$ & $\MF{so}(1,1)$\\
$\{e_{1}-e_{3}\}$ & $\MF{so}(1,1)$\\
$\emptyset$ & $\{0\}$\\
\hline
\end{tabular}

\medskip

\begin{tabular}{|c|c|}
\hline
$\Theta(\subset s_{1}s_{2}\varPsi)$ & $\MF{h}_{\Theta}$ \\
\hline
\hline
$s_{1}s_{2}\varPsi$ & $\MF{so}(2,2)$\\
$\{e_{3}-e_{1},e_{1}-e_{4}\}$ & $\MF{so}(1,2)$\\
$\{e_{2}-e_{3},e_{1}-e_{4}\}$ & $\MF{so}(1,1)+\MF{so}(1,1)$\\
$\{e_{2}-e_{3},e_{3}-e_{1}\}$ & $\MF{so}(2,1)$\\
$\{e_{1}-e_{4}\}$ & $\MF{so}(1,1)$\\
$\{e_{3}-e_{1}\}$ & $\MF{so}(1,1)$\\
$\{e_{2}-e_{3}\}$ & $\MF{so}(1,1)$\\
$\emptyset$ & $\{0\}$\\
\hline
\end{tabular}
\begin{tabular}{|c|c|}
\hline
$\Theta(\subset s_{3}s_{2}\varPsi)$ & $\MF{h}_{\Theta}$ \\
\hline
\hline
$s_{3}s_{2}\varPsi$ & $\MF{so}(2,2)$\\
$\{e_{4}-e_{2},e_{2}-e_{3}\}$ & $\MF{so}(1,2)$\\
$\{e_{1}-e_{4},e_{2}-e_{3}\}$ & $\MF{so}(1,1)+\MF{so}(1,1)$\\
$\{e_{1}-e_{4},e_{4}-e_{2}\}$ & $\MF{so}(2,1)$\\
$\{e_{2}-e_{3}\}$ & $\MF{so}(1,1)$\\
$\{e_{4}-e_{2}\}$ & $\MF{so}(1,1)$\\
$\{e_{1}-e_{4}\}$ & $\MF{so}(1,1)$\\
$\emptyset$ & $\{0\}$\\
\hline
\end{tabular}

\medskip

\begin{tabular}{|c|c|}
\hline
$\Theta(\subset s_{1}s_{3}s_{2}\varPsi)$ & $\MF{h}_{\Theta}$ \\
\hline
\hline
$s_{1}s_{3}s_{2}\varPsi$ & $\MF{so}(2,2)$\\
$\{e_{4}-e_{1},e_{1}-e_{3}\}$ & $\MF{so}(1,2)$\\
$\{e_{2}-e_{4},e_{1}-e_{3}\}$ & $\MF{so}(1,1)+\MF{so}(1,1)$\\
$\{e_{2}-e_{4},e_{4}-e_{1}\}$ & $\MF{so}(2,1)$\\
$\{e_{1}-e_{3}\}$ & $\MF{so}(1,1)$\\
$\{e_{4}-e_{1}\}$ & $\MF{so}(1,1)$\\
$\{e_{2}-e_{4}\}$ & $\MF{so}(1,1)$\\
$\emptyset$ & $\{0\}$\\
\hline
\end{tabular}
\begin{tabular}{|c|c|}
\hline
$\Theta(\subset s_{2}s_{1}s_{3}s_{2}\varPsi)$ & $\MF{h}_{\Theta}$ \\
\hline
\hline
$s_{2}s_{1}s_{3}s_{2}\varPsi$ & $\MF{so}(2,2)$\\
$\{e_{4}-e_{1},e_{1}-e_{2}\}$ & $\MF{so}(2,1)$\\
$\{e_{3}-e_{4},e_{1}-e_{2}\}$ & $\MF{so}(2,0)+\MF{so}(0,2)$\\
$\{e_{3}-e_{4},e_{4}-e_{1}\}$ & $\MF{so}(1,2)$\\
$\{e_{1}-e_{2}\}$ & $\MF{so}(2,0)$\\
$\{e_{4}-e_{1}\}$ & $\MF{so}(1,1)$\\
$\{e_{3}-e_{4}\}$ & $\MF{so}(0,2)$\\
$\emptyset$ & $\{0\}$\\
\hline
\end{tabular}
\end{table}

\begin{Ex}\label{ex.local1}
Let $G/H$ be a semisimple pseudo-Riemannian symmetric space.
Suppose that the restricted root system $\Delta$
with respect to a vector-type maximal split abelian subspace
satisfies $\Delta = \Delta^{a}$ or
$(\Delta, \Delta^{a})=((BC)_{r},$ $B_{r})$. 
Then we have $\MC{W}(\Delta)/\MC{W}(\Delta^{a})=\{\OPE{id}\}$.
It follows from Theorem \ref{thm.main}
that $\MC{L}_{h}(G/H)=\{[\MF{h}_{\Theta}]\mid \Theta \subset \varPsi\}$
holds,
where $\varPsi$ is a standard simple root system of $\Delta$.
In Table \ref{table.localtable.1},
we list up the set of all the possible local orbit types
of the hyperbolic orbits for the s-representations
associated with
all classical-type semisimple pseudo-Riemannian symmetric spaces
satisfying $\Delta=\Delta^{a}$ or $(\Delta,\Delta^{a})=((BC)_{r},B_{r})$.
\end{Ex}

\begin{table}[htbp]
\footnotesize
\caption{Local orbit types}\label{table.localtable.1} 
\begin{flushleft}
(I) $\Delta=\Delta^{a}=A_{r},\,\varPsi=\{e_{i}-e_{i+1}\mid 1 \leq i \leq r\}$ 
\end{flushleft}

\begin{center}
\begin{tabular}{|>{\PBS\centering}p{\LENGTHTHETA}|>{\PBS\centering}p{\LENGTHH}|}
\multicolumn{2}{l}{
$(\MF{g},\MF{h})=(\MF{sl}(n,\BS{R})+\MF{sl}(n,\BS{R}),\MF{sl}(n,\BS{R}))$
}\\
\hline
$\Theta (\subset \varPsi)$ & 
$\varPsi\setminus\{e_{i_{1}}-e_{i_{1}+1},\ldots,e_{i_{k}}-e_{i_{k}+1}\}$\\
\hline
$\MF{h}_{\Theta}$ &
$\BS{R}^{k}+\displaystyle{\sum^{k+1}_{l=1}}\MF{sl}(i_{l}-i_{l-1},\BS{R})$\\
\hline
Remarks &
$0=i_{0}<i_{1}<\cdots<i_{k}<i_{k+1}= n$\\
\hline
\end{tabular}
\end{center}
\begin{center}
\begin{tabular}{|>{\PBS\centering}p{\LENGTHTHETA}|>{\PBS\centering}p{\LENGTHH}|}
\multicolumn{2}{l}{
$(\MF{g},\MF{h})=(\MF{sl}(n,\BS{C}),\MF{so}(n,\BS{C}))$
}\\
\hline
$\Theta(\subset \varPsi)$ & $\varPsi\setminus\{e_{i_{1}}-e_{i_{1}+1},\ldots,e_{i_{k}}-e_{i_{k}+1}\}$\\
\hline
$\MF{h}_{\Theta}$ & $\displaystyle{\sum^{k+1}_{l=1}}\MF{so}(i_{l}-i_{l-1},\BS{C})$\\
\hline
Remarks & $0=i_{0}<i_{1}<\cdots<i_{k} <  i_{k+1} = n$\\
\hline
\end{tabular}
\end{center}
\begin{center}
\begin{tabular}{|>{\PBS\centering}p{\LENGTHTHETA}|>{\PBS\centering}p{\LENGTHH}|}
\multicolumn{2}{l}{
$(\MF{g},\MF{h})=(\MF{su}^{*}(2n)+\MF{su}^{*}(2n),\MF{su}^{*}(2n))$
}\\
\hline
$\Theta(\subset \varPsi)$ & $\varPsi\setminus\{e_{i_{1}}-e_{i_{1}+1},\ldots,e_{i_{k}}-e_{i_{k}+1}\}$\\
\hline
$\MF{h}_{\Theta}$ & $\BS{R}^{k}+\displaystyle{\sum^{k+1}_{l=1}}\MF{su}^{*}(2(i_{l}-i_{l-1}))$\\
\hline
Remarks & $0=i_{0}<i_{1}<\cdots<i_{k} < i_{k+1} =  n$\\
\hline
\end{tabular}
\end{center}
\begin{center}
\begin{tabular}{|>{\PBS\centering}p{\LENGTHTHETA}|>{\PBS\centering}p{\LENGTHH}|}
\multicolumn{2}{l}{
$(\MF{g},\MF{h})=(\MF{sl}(2n,\BS{C}),\MF{sp}(n,\BS{C}))$
}\\
\hline
$\Theta(\subset \varPsi)$ & $\varPsi\setminus\{e_{i_{1}}-e_{i_{1}+1},\ldots,e_{i_{k}}-e_{i_{k}+1}\}$\\
\hline
$\MF{h}_{\Theta}$ &
$\displaystyle{\sum^{k+1}_{l=1}}\MF{sp}(i_{l}-i_{l-1},\BS{C})$\\
\hline
Remarks & $0=i_{0}<i_{1}<\cdots < i_{k}<i_{k+1}= n$\\
\hline
\end{tabular}
\end{center}
\begin{center}
\begin{tabular}{|>{\PBS\centering}p{\LENGTHTHETA}|>{\PBS\centering}p{\LENGTHH}|}
\multicolumn{2}{l}{
$(\MF{g},\MF{h})=(\MF{sl}(2n,\BS{R}),\MF{sp}(n,\BS{R}))$
}\\
\hline
$\Theta(\subset \varPsi)$ & $\varPsi\setminus\{e_{i_{1}}-e_{i_{1}+1},\ldots,e_{i_{k}}-e_{i_{k}+1}\}$\\
\hline
$\MF{h}_{\Theta}$ &
$\displaystyle{\sum^{k+1}_{l=1}}\MF{sp}(i_{l}-i_{l-1},\BS{R})$
\\
\hline
Remarks & $0=i_{0}<i_{1}<\cdots<i_{k}<i_{k+1}= n$
\\
\hline
\end{tabular}
\end{center}
\begin{center}
\begin{tabular}{|>{\PBS\centering}p{\LENGTHTHETA}|>{\PBS\centering}p{\LENGTHH}|}
\multicolumn{2}{l}{
$(\MF{g},\MF{h})=(\MF{su}^{*}(2n),\MF{so}^{*}(2n))$
}\\
\hline
$\Theta(\subset \varPsi)$ & $\varPsi\setminus\{e_{i_{1}}-e_{i_{1}+1},\ldots,e_{i_{k}}-e_{i_{k}+1}\}$
\\
\hline
$\MF{h}_{\Theta}$& 
$\displaystyle{\sum^{k+1}_{l=1}}\MF{so}^{*}(2(i_{l}-i_{l-1}))$
\\
\hline
Remarks & 
$0=i_{0}<i_{1}<\cdots<i_{k}<i_{k+1}=n$
\\
\hline
\end{tabular}
\end{center}
\medskip
\begin{flushleft}
(II) $\Delta=\Delta^{a}=B_{r},\,\varPsi=\{e_{i}-e_{i+1}\mid 1 \leq i \leq r-1\}\cup\{e_{r}\}$
\end{flushleft}
\begin{center}
\begin{tabular}{|>{\PBS\centering}p{\LENGTHTHETA}|>{\PBS\centering}p{\LENGTHH}|}
\multicolumn{2}{l}{
$(\MF{g},\MF{h})=(\MF{so}(p,n-p)+\MF{so}(p,n-p),\MF{so}(p,n-p))\,(n>2p)$
}\\
\hline
$\Theta(\subset \varPsi)$ & $\varPsi\setminus\{e_{i_{1}}-e_{i_{1}+1},\ldots, e_{i_{k-1}}-e_{i_{k-1}+1},e_{p}\}$
\\
\hline
$\MF{h}_{\Theta}$ & 
$\BS{R}^{k}+\displaystyle{\sum^{k}_{l=1}}\MF{sl}(i_{l}-i_{l-1},\BS{R})+\MF{so}(n-2p)$
\\
\hline
Remarks &
$0=i_{0}<i_{1}<\cdots<i_{k}= p$
\\
\hline
\hline
$\Theta(\subset \varPsi)$&
$\varPsi\setminus\{e_{i_{1}}-e_{i_{1}+1},\ldots,e_{i_{k}}-e_{i_{k}+1}\}$
\\
\hline
$\MF{h}_{\Theta}$& 
$\BS{R}^{k}+\displaystyle{\sum^{k}_{l=1}}\MF{sl}(i_{l}-i_{l-1},\BS{R})+\MF{so}(p-i_{k},n-p-i_{k})$
\\
\hline
Remarks&
$0=i_{0}<i_{1}<\cdots<i_{k}< p$
\\
\hline
\end{tabular}
\end{center}
\end{table}

\begin{table}[htbp]
\footnotesize
\contcaption{(continued)}
\begin{center}
\begin{tabular}{|>{\PBS\centering}p{\LENGTHTHETA}|>{\PBS\centering}p{\LENGTHH}|}
\multicolumn{2}{l}{
$(\MF{g},\MF{h})=(\MF{so}(n,\BS{C}),\MF{so}(p,\BS{C})+\MF{so}(n-p,\BS{C}))\,(n>2p)$
}\\
\hline
$\Theta(\subset \varPsi)$ & $\varPsi\setminus\{e_{i_{1}}-e_{i_{1}+1},\ldots,
e_{i_{k-1}}-e_{i_{k-1}+1},e_{p}
\}$
\\
\hline
$\MF{h}_{\Theta}$
& 
$\displaystyle{\sum^{k}_{l=1}}\MF{so}(i_{l}-i_{l-1},\BS{C})+\MF{so}(n-2p)$
\\
\hline
Remarks
&
$0=i_{0}<i_{1}<\cdots<i_{k}= p$
\\
\hline
\hline
$\Theta(\subset \varPsi)$ & 
$\varPsi\setminus\{e_{i_{1}}-e_{i_{1}+1},\ldots,
e_{i_{k-1}}-e_{i_{k-1}+1},e_{p}
\}$
\\
\hline
$\MF{h}_{\Theta}$
& 
$\displaystyle{\sum^{k}_{l=1}}\MF{so}(i_{l}-i_{l-1},\BS{C})+\MF{so}(p-i_{k},\BS{C})+\MF{so}(n-p-i_{k},\BS{C})$
\\
\hline
Remarks
&
$0=i_{0}<i_{1}<\cdots<i_{k}< p$
\\
\hline
\end{tabular}
\end{center}
\medskip
\begin{flushleft}
(III) $\Delta=\Delta^{a}=C_{r},\,\varPsi=\{e_{i}-e_{i+1}\mid 1 \leq i \leq r-1\}\cup\{2e_{r}\}$ 
\end{flushleft}
\begin{center}\begin{tabular}{|>{\PBS\centering}p{\LENGTHTHETA}|>{\PBS\centering}p{\LENGTHH}|}
\multicolumn{2}{l}{
$(\MF{g},\MF{h})=(\MF{sl}(2n,\BS{C}),\MF{su}^{*}(2n))$
}\\
\hline
$\Theta(\subset \varPsi)$ &
$\varPsi\setminus\{e_{i_{1}}-e_{i_{1}+1},\ldots,
e_{i_{k-1}}-e_{i_{k-1}+1},2e_{n}
\}$\\
\hline
$\MF{h}_{\Theta}$
&
$\BS{R}^{k-1}+\MF{so}(2)^{k}+\displaystyle{\sum^{k}_{l=1}}\MF{sl}(i_{l}-i_{l-1},\BS{C})$
\\
\hline
Remarks
&
$0=i_{0}<i_{1}<\dots<i_{k}= n$\\
\hline
\hline
$\Theta(\subset \varPsi)$ &
$\varPsi\setminus\{e_{i_{1}}-e_{i_{1}+1},\ldots,e_{i_{k}}-e_{i_{k}+1}\}$
\\
\hline
$\MF{h}_{\Theta}$
&
$\BS{R}^{k}+\MF{so}(2)^{k}+\displaystyle{\sum^{k}_{l=1}}\MF{sl}(i_{l}-i_{l-1},\BS{C})+\MF{su}^{*}(2(n-i_{k}))$
\\
\hline
Remarks
&
$0=i_{0}<i_{1}<\dots<i_{k}<n$\\
\hline
\end{tabular}
\end{center}
\begin{center}
\begin{tabular}{|>{\PBS\centering}p{\LENGTHTHETA}|>{\PBS\centering}p{\LENGTHH}|}
\multicolumn{2}{l}{
$(\MF{g},\MF{h})=(\MF{su}(n,n)+\MF{su}(n,n),\MF{su}(n,n))$
}\\
\hline
$\Theta(\subset \varPsi)$ & 
$\varPsi\setminus\{e_{i_{1}}-e_{i_{1}+1},\ldots,
e_{i_{k-1}}-e_{i_{k-1}+1},2e_{n}
\}$\\
\hline
$\MF{h}_{\Theta}$&
$\BS{R}^{k}+\MF{so}(2)^{k-1}+\displaystyle{\sum^{k}_{l=1}}\MF{sl}(i_{l}-i_{l-1},\BS{C})$\\
\hline
Remarks
&
$0=i_{0}<i_{1}<\dots<i_{k}= n$\\
\hline
\hline
$\Theta(\subset \varPsi)$&
$\varPsi\setminus\{e_{i_{1}}-e_{i_{1}+1},\ldots,e_{i_{k}}-e_{i_{k}+1}\}$
\\
\hline
$\MF{h}_{\Theta}$
&
$\BS{R}^{k}+\MF{so}(2)^{k}+\displaystyle{\sum^{k}_{l=1}}\MF{sl}(i_{l}-i_{l-1},\BS{C})+\MF{su}(n-i_{k},n-i_{k})$\\
\hline
Remarks
&
$0=i_{0}<i_{1}<\dots<i_{k}<n$\\
\hline
\end{tabular}
\end{center}
\begin{center}
\begin{tabular}{|>{\PBS\centering}p{\LENGTHTHETA}|>{\PBS\centering}p{\LENGTHH}|}
\multicolumn{2}{l}{
$(\MF{g},\MF{h})=(\MF{sl}(2n,\BS{C}),\MF{sl}(n,\BS{C})+\MF{sl}(n,\BS{C})+\BS{C})$
}\\
\hline
$\Theta(\subset \varPsi)$ &
$\varPsi\setminus\{e_{i_{1}}-e_{i_{1}+1},\ldots,
e_{i_{k-1}}-e_{i_{k-1}+1},2e_{n}
\}$
\\\hline
$\MF{h}_{\Theta}$
&
$\BS{C}^{k-1}+\displaystyle{\sum^{k}_{l=1}}\MF{sl}(i_{l}-i_{l-1},\BS{C})$\\
\hline
Remarks
&
$0=i_{0}<i_{1}<\dots<i_{k}= n$\\
\hline
\hline
$\Theta(\subset \varPsi)$ &
$\varPsi\setminus\{e_{i_{1}}-e_{i_{1}+1},\ldots,e_{i_{k}}-e_{i_{k}+1}\}$
\\\hline
$\MF{h}_{\Theta}$&
$\BS{C}^{k}+\displaystyle{\sum^{k}_{l=1}}\MF{sl}(i_{l}-i_{l-1},\BS{C})+\MF{sl}(n-i_{k},\BS{C})^{2}+\BS{C}$\\\hline
Remarks
&
$0=i_{0}<i_{1}<\dots<i_{k}<n$\\
\hline
\end{tabular}
\end{center}
\end{table}

\begin{table}[htbp]
\footnotesize
\contcaption{(continued)}
\begin{center}

\begin{tabular}{|>{\PBS\centering}p{\LENGTHTHETA}|>{\PBS\centering}p{\LENGTHH}|}
\multicolumn{2}{l}{
$(\MF{g},\MF{h})=(\MF{so}^{*}(4n)+\MF{so}^{*}(4n),\MF{so}^{*}(4n))$
}\\
\hline
$\Theta(\subset \varPsi)$ &
$\varPsi\setminus\{e_{i_{1}}-e_{i_{1}+1},\ldots,
e_{i_{k-1}}-e_{i_{k-1}+1},2e_{n}\}$
\\\hline
$\MF{h}_{\Theta}$
& 
$\BS{R}^{k}+\displaystyle{\sum^{k}_{l=1}}\MF{su}^{*}(2(i_{l}-i_{l-1}))$
\\\hline
Remarks
&
$0=i_{0}<i_{1}<\cdots<i_{k}= n$
\\
\hline
\hline
$\Theta(\subset \varPsi)$ &
$\varPsi\setminus\{e_{i_{1}}-e_{i_{1}+1},\ldots,e_{i_{k}}-e_{i_{k}+1}\}$
\\\hline
$\MF{h}_{\Theta}$
& 
$\BS{R}^{k}+\displaystyle{\sum^{k}_{l=1}}\MF{su}^{*}(2(i_{l}-i_{l-1}))+\MF{so}^{*}(4(n-i_{k}))$
\\\hline
Remarks
&
$0=i_{0}<i_{1}<\cdots<i_{k}< n$
\\
\hline
\end{tabular}
\end{center}
\begin{center}

\begin{tabular}{|>{\PBS\centering}p{\LENGTHTHETA}|>{\PBS\centering}p{\LENGTHH}|}
\multicolumn{2}{l}{
$(\MF{g},\MF{h})=(\MF{so}(4n,\BS{C}),\MF{sl}(2n,\BS{C})+\BS{C})$
}\\
\hline
$\Theta(\subset \varPsi)$ &
$\varPsi\setminus\{e_{i_{1}}-e_{i_{1}+1},\ldots,
e_{i_{k-1}}-e_{i_{k-1}+1},2e_{n}\}$
\\\hline
$\MF{h}_{\Theta}$
& 
$\displaystyle{\sum^{k}_{l=1}}\MF{sp}(i_{l}-i_{l-1},\BS{C})$
\\\hline
Remarks
&
$0=i_{0}<i_{1}<\cdots<i_{k}=n$
\\
\hline
\hline
$\Theta(\subset \varPsi)$ &
$\varPsi\setminus\{e_{i_{1}}-e_{i_{1}+1},\ldots,e_{i_{k}}-e_{i_{k}+1}\}$
\\\hline
$\MF{h}_{\Theta}$
&
$\displaystyle{\sum^{k}_{l=1}}\MF{sp}(i_{l}-i_{l-1},\BS{C})+\MF{sl}(2(n-i_{k}),\BS{C})+\BS{C}$
\\\hline
Remarks
&
$0=i_{0}<i_{1}<\cdots<i_{k}<n$
\\
\hline
\end{tabular}
\end{center}
\begin{center}

\begin{tabular}{|>{\PBS\centering}p{\LENGTHTHETA}|>{\PBS\centering}p{\LENGTHH}|}
\multicolumn{2}{l}{
$(\MF{g},\MF{h})=(\MF{sp}(n,\BS{R})+\MF{sp}(n,\BS{R}),\MF{sp}(n,\BS{R}))$
}\\
\hline
$\Theta(\subset \varPsi)$ &
$\varPsi\setminus\{e_{i_{1}}-e_{i_{1}+1},\ldots,
e_{i_{k-1}}-e_{i_{k-1}+1},2e_{n}\}$
\\\hline
$\MF{h}_{\Theta}$ & 
$\BS{R}^{k}+\displaystyle{\sum^{k}_{l=1}}\MF{sl}(i_{l}-i_{l-1},\BS{R})$
\\\hline
Remarks
&
$0=i_{0}<i_{1}<\cdots<i_{k}= n$
\\
\hline
\hline
$\Theta(\subset \varPsi)$ &
$\varPsi\setminus\{e_{i_{1}}-e_{i_{1}+1},\ldots,e_{i_{k}}-e_{i_{k}+1}\}$
\\\hline
$\MF{h}_{\Theta}$& 
$\BS{R}^{k}+\displaystyle{\sum^{k}_{l=1}}\MF{sl}(i_{l}-i_{l-1},\BS{R})+\MF{sp}(n-i_{k},\BS{R})$
\\\hline
Remarks
&
$0=i_{0}<i_{1}<\cdots<i_{k}< n$
\\
\hline
\end{tabular}
\end{center}
\begin{center}
\begin{tabular}{|>{\PBS\centering}p{\LENGTHTHETA}|>{\PBS\centering}p{\LENGTHH}|}
\multicolumn{2}{l}{
$(\MF{g},\MF{h})=(\MF{sp}(n,\BS{C}),\MF{sl}(n,\BS{C})+\BS{C})$
}\\
\hline
$\Theta(\subset \varPsi)$ &
$\varPsi\setminus\{e_{i_{1}}-e_{i_{1}+1},\ldots,
e_{i_{k-1}}-e_{i_{k-1}+1},2e_{n}\}$
\\\hline
$\MF{h}_{\Theta}$
& 
$\displaystyle{\sum^{k}_{l=1}}\MF{so}(i_{l}-i_{l-1},\BS{C})$
\\\hline
Remarks
&
$0=i_{0}<i_{1}<\cdots<i_{k}=n$
\\
\hline
\hline
$\Theta(\subset \varPsi)$ &
$\varPsi\setminus\{e_{i_{1}}-e_{i_{1}+1},\ldots,e_{i_{k}}-e_{i_{k}+1}\}$
\\\hline
$\MF{h}_{\Theta}$&
$\displaystyle{\sum^{k}_{l=1}}\MF{so}(i_{l}-i_{l-1},\BS{C})+\MF{sl}(n-i_{k},\BS{C})+\BS{C}$
\\\hline
Remarks
&
$0=i_{0}<i_{1}<\cdots<i_{k}<n$
\\
\hline
\end{tabular}
\end{center}
\end{table}

\begin{table}[htbp]
\footnotesize
\contcaption{(continued)}
\begin{center}

\begin{tabular}{|>{\PBS\centering}p{\LENGTHTHETA}|>{\PBS\centering}p{\LENGTHH}|}
\multicolumn{2}{l}{
$(\MF{g},\MF{h})=(\MF{sp}(n,n)+\MF{sp}(n,n),\MF{sp}(n,n))$
}\\
\hline
$\Theta(\subset \varPsi)$ &
$\varPsi\setminus\{e_{i_{1}}-e_{i_{1}+1},\ldots,
e_{i_{k-1}}-e_{i_{k-1}+1},2e_{n}
\}$
\\\hline
$\MF{h}_{\Theta}$
& 
$\displaystyle{\BS{R}^{k}+\sum^{k}_{l=1}\MF{su}^{*}(2(i_{l}-i_{l-1}))}$
\\\hline
Remarks
&
$0=i_{0}<i_{1}<\cdots<i_{k}=n$
\\
\hline
\hline
$\Theta(\subset \varPsi)$ &
$\varPsi\setminus\{e_{i_{1}}-e_{i_{1}+1},\ldots,e_{i_{k}}-e_{i_{k}+1}\}$
\\\hline
$\MF{h}_{\Theta}$
& 
$\displaystyle{\BS{R}^{k}+\sum^{k}_{l=1}\MF{su}^{*}(2(i_{l}-i_{l-1}))+\MF{sp}(n-i_{k},n-i_{k})}$
\\\hline
Remarks
&
$0=i_{0}<i_{1}<\cdots<i_{k}<n$
\\
\hline
\end{tabular}
\end{center}
\begin{center}

\begin{tabular}{|>{\PBS\centering}p{\LENGTHTHETA}|>{\PBS\centering}p{\LENGTHH}|}
\multicolumn{2}{l}{
$(\MF{g},\MF{h})=(\MF{sp}(2n,\BS{C}),\MF{sp}(n,\BS{C})+\MF{sp}(n,\BS{C}))$
}\\
\hline
$\Theta(\subset \varPsi)$ &
$\varPsi\setminus\{e_{i_{1}}-e_{i_{1}+1},\ldots,
e_{i_{k-1}}-e_{i_{k-1}+1},2e_{n}\}$
\\\hline
$\MF{h}_{\Theta}$
& 
$\displaystyle{\sum^{k}_{l=1}\MF{sp}(i_{l}-i_{l-1},\BS{C})}$
\\\hline
Remarks
&
$0=i_{0}<i_{1}<\cdots<i_{k}=n$
\\
\hline
\hline
$\Theta(\subset \varPsi)$ &
$\varPsi\setminus\{e_{i_{1}}-e_{i_{1}+1},\ldots,e_{i_{k}}-e_{i_{k}+1}\}$
\\\hline
$\MF{h}_{\Theta}$ 
& 
$\displaystyle{\sum^{k}_{l=1}\MF{sp}(i_{l}-i_{l-1},\BS{C})}+\MF{sp}(n-i_{k},\BS{C})^{2}$
\\\hline
Remarks
&
$0=i_{0}<i_{1}<\cdots<i_{k}<n$
\\
\hline
\end{tabular}
\end{center}
\begin{center}
\begin{tabular}{|>{\PBS\centering}p{\LENGTHTHETA}|>{\PBS\centering}p{\LENGTHH}|}
\multicolumn{2}{l}{
$(\MF{g},\MF{h})=(\MF{su}(2n,2n),\MF{sp}(n,n))$
}\\
\hline
$\Theta(\subset \varPsi)$ &
$\varPsi\setminus\{e_{i_{1}}-e_{i_{1}+1},\ldots,
e_{i_{k-1}}-e_{i_{k-1}+1},2e_{n}
\}$
\\\hline
$\MF{h}_{\Theta}$
& 
$\displaystyle{\sum^{k}_{l=1}\MF{sp}(i_{l}-i_{l-1},\BS{C})}$
\\\hline
Remarks
&
$0=i_{0}<i_{1}<\cdots<i_{k}=n$
\\
\hline
\hline
$\Theta(\subset \varPsi)$ &
$\varPsi\setminus\{e_{i_{1}}-e_{i_{1}+1},\ldots,e_{i_{k}}-e_{i_{k}+1}\}$
\\\hline
$\MF{h}_{\Theta}$
& 
$\displaystyle{\sum^{k}_{l=1}\MF{sp}(i_{l}-i_{l-1},\BS{C})+\MF{sp}(n-i_{k},n-i_{k})}$
\\\hline
Remarks
&
$0=i_{0}<i_{1}<\cdots<i_{k}<n$
\\
\hline
\end{tabular}
\end{center}
\begin{center}
 \begin{tabular}{|>{\PBS\centering}p{\LENGTHTHETA}|>{\PBS\centering}p{\LENGTHH}|}
\multicolumn{2}{l}{
$(\MF{g},\MF{h})=(\MF{su}^{*}(4n),\MF{su}^{*}(2n)+\MF{su}^{*}(2n)+\BS{R})$
}\\
\hline
$\Theta(\subset \varPsi)$ &
$\varPsi\setminus\{e_{i_{1}}-e_{i_{1}+1},\ldots,
e_{i_{k-1}}-e_{i_{k-1}+1},2e_{n}
\}$
\\\hline
$\MF{h}_{\Theta}$
& 
$\BS{R}^{k-1}+\displaystyle{\sum^{k}_{l=1}}\MF{su}^{*}(2(i_{l}-i_{l-1}))$
\\\hline
Remarks
&
$0=i_{0}<i_{1}<\cdots<i_{k}=n$
\\
\hline
\hline
$\Theta(\subset \varPsi)$
&
$\varPsi\setminus\{e_{i_{1}}-e_{i_{1}+1},\ldots,e_{i_{k}}-e_{i_{k}+1}\}$
\\\hline
$\MF{h}_{\Theta}$
&
$\BS{R}^{k}+\displaystyle{\sum^{k}_{l=1}}\MF{su}^{*}(2(i_{l}-i_{l-1}))+\MF{su}^{*}(2(n-i_{k}))^{2}+\BS{R}$
\\\hline
Remarks
&
$0=i_{0}<i_{1}<\cdots<i_{k}<n$
\\
\hline
\end{tabular}
\end{center}
\end{table}

\begin{table}[htbp]
\footnotesize
\contcaption{(continued)}
\begin{center}
\begin{tabular}{|>{\PBS\centering}p{\LENGTHTHETA}|>{\PBS\centering}p{\LENGTHH}|}
\multicolumn{2}{l}{
$(\MF{g},\MF{h})=(\MF{su}(n,n),\MF{so}^{*}(2n))$
}\\
\hline
$\Theta(\subset \varPsi)$ &
$\varPsi\setminus\{e_{i_{1}}-e_{i_{1}+1},\ldots,
e_{i_{k-1}}-e_{i_{k-1}+1},2e_{n}
\}$
\\\hline
$\MF{h}_{\Theta}$
& 
$\displaystyle{\sum^{k}_{l=1}}\MF{so}(i_{l}-i_{l-1},\BS{C})$
\\\hline
Remarks
&
$0=i_{0}<i_{1}<\cdots<i_{k} =  n$
\\
\hline
\hline
$\Theta(\subset \varPsi)$&
$\varPsi\setminus\{e_{i_{1}}-e_{i_{1}+1},\ldots,e_{i_{k}}-e_{i_{k}+1}\}$
\\\hline
$\MF{h}_{\Theta}(\subset \varPsi)$
& 
$\displaystyle{\sum^{k}_{l=1}}\MF{so}(i_{l}-i_{l-1},\BS{C})+\MF{so}^{*}(2(n-i_{k}))$
\\\hline
Remarks
&
$0=i_{0}<i_{1}<\cdots<i_{k} < n$
\\
\hline
\end{tabular}
\end{center}
\begin{center}
\begin{tabular}{|>{\PBS\centering}p{\LENGTHTHETA}|>{\PBS\centering}p{\LENGTHH}|}
\multicolumn{2}{l}{
$(\MF{g},\MF{h})=(\MF{sl}(2n,\BS{R}),\MF{sl}(n,\BS{C})+\MF{so}(2))$
}\\
\hline
$\Theta(\subset \varPsi)$ &
$\varPsi\setminus\{e_{i_{1}}-e_{i_{1}+1},\ldots,
e_{i_{k-1}}-e_{i_{k-1}+1},2e_{n}
\}$
\\\hline
$\MF{h}_{\Theta}$
& 
$\BS{R}^{k-1}+\displaystyle{\sum^{k}_{l=1}}\MF{sl}(i_{l}-i_{l-1},\BS{R})$
\\\hline
Remarks
&
$0=i_{0}<i_{1}<\cdots<i_{k}=n$
\\
\hline
\hline
$\Theta(\subset \varPsi)$&
$\varPsi\setminus\{e_{i_{1}}-e_{i_{1}+1},\ldots,e_{i_{k}}-e_{i_{k}+1}\}$
\\\hline
$\MF{h}_{\Theta}$
&
$\BS{R}^{k}+\displaystyle{\sum^{k}_{l=1}}\MF{sl}(i_{l}-i_{l-1},\BS{R})+\MF{sl}(n-i_{k},\BS{C})+\MF{so}(2)$
\\\hline
Remarks
&
$0=i_{0}<i_{1}<\cdots<i_{k}<n$
\\
\hline
\end{tabular}
\end{center}
\begin{center}
\begin{tabular}{|>{\PBS\centering}p{\LENGTHTHETA}|>{\PBS\centering}p{\LENGTHH}|}
\multicolumn{2}{l}{
$(\MF{g},\MF{h})=(\MF{su}^{*}(4n),\MF{sl}(2n,\BS{C})+\MF{so}(2))$
}\\
\hline
$\Theta(\subset \varPsi)$ &
$\varPsi\setminus\{e_{i_{1}}-e_{i_{1}+1},\ldots,
e_{i_{k-1}}-e_{i_{k-1}+1},2e_{n}\}$
\\\hline
$\MF{h}_{\Theta}$
& 
$\BS{R}^{k-1}+\displaystyle{\sum^{k}_{l=1}}\MF{su}^{*}(2(i_{l}-i_{l-1}))$
\\\hline
Remarks
&
$0=i_{0}<i_{1}<\cdots<i_{k}=n$
\\
\hline
\hline
$\Theta(\subset \varPsi)$ &
$\varPsi\setminus\{e_{i_{1}}-e_{i_{1}+1},\ldots,e_{i_{k}}-e_{i_{k}+1}\}$
\\\hline
$\MF{h}_{\Theta}$
&
$\BS{R}^{k}+\displaystyle{\sum^{k}_{l=1}}\MF{su}^{*}(2(i_{l}-i_{l-1}))+\MF{sl}(2(n-i_{k}),\BS{C})+\MF{so}(2)$
\\\hline
Remarks
&
$0=i_{0}<i_{1}<\cdots<i_{k}<n$
\\
\hline
\end{tabular}
\end{center}
\begin{center}
\begin{tabular}{|>{\PBS\centering}p{\LENGTHTHETA}|>{\PBS\centering}p{\LENGTHH}|}
\multicolumn{2}{l}{
$(\MF{g},\MF{h})=(\MF{su}(2n,2n),\MF{sp}(2n,\BS{R}))$
}\\
\hline
$\Theta(\subset \varPsi)$ &
$\varPsi\setminus\{e_{i_{1}}-e_{i_{1}+1},\ldots,
e_{i_{k-1}}-e_{i_{k-1}+1},2e_{n}\}$
\\\hline
$\MF{h}_{\Theta}$
& 
$\displaystyle{\sum^{k}_{l=1}}\MF{sp}(i_{l}-i_{l-1},\BS{C})$
\\\hline
$\MF{h}_{\Theta}$
&
$0=i_{0}<i_{1}<\cdots<i_{k}=n$
\\\hline
\hline
$\Theta(\subset\varPsi)$&
$\varPsi\setminus\{e_{i_{1}}-e_{i_{1}+1},\ldots,e_{i_{k}}-e_{i_{k}+1}\}$
\\\hline
$\MF{h}_{\Theta}$
& 
$\displaystyle{\sum^{k}_{l=1}}\MF{sp}(i_{l}-i_{l-1},\BS{C})+\MF{sp}(2(n-i_{k}),\BS{R})$
\\\hline
Remarks
&
$0=i_{0}<i_{1}<\cdots<i_{k}< n$
\\
\hline
\end{tabular}
\end{center}
\end{table}

\begin{table}[htbp]
\footnotesize
\contcaption{(continued)}
\begin{center}
\begin{tabular}{|>{\PBS\centering}p{\LENGTHTHETA}|>{\PBS\centering}p{\LENGTHH}|}
\multicolumn{2}{l}{
$(\MF{g},\MF{h})=(\MF{so}(2n,2n),\MF{su}(n,n)+\MF{so}(2))$
}\\
\hline
$\Theta(\subset \varPsi)$ &
$\varPsi\setminus\{e_{i_{1}}-e_{i_{1}+1},\ldots,
e_{i_{k-1}}-e_{i_{k-1}+1},2e_{n}
\}$
\\\hline
$\MF{h}_{\Theta}$
& 
$\displaystyle{\sum^{k}_{l=1}\MF{sp}(i_{l}-i_{l-1},\BS{R})}$
\\\hline
Remarks
&
$0=i_{0}<i_{1}<\cdots<i_{k}=n$
\\
\hline
\hline
$\Theta(\subset \varPsi)$ &
$\varPsi\setminus\{e_{i_{1}}-e_{i_{1}+1},\ldots,e_{i_{k}}-e_{i_{k}+1}\}$
\\\hline
$\MF{h}_{\Theta}$
& 
$\displaystyle{\sum^{k}_{l=1}\MF{sp}(i_{l}-i_{l-1},\BS{R})+\MF{su}(n-i_{k},n-i_{k})+\MF{so}(2)}$
\\\hline
Remarks
&
$0=i_{0}<i_{1}<\cdots<i_{k}<n$
\\
\hline
\end{tabular}
\end{center}
\begin{center}
\begin{tabular}{|>{\PBS\centering}p{\LENGTHTHETA}|>{\PBS\centering}p{\LENGTHH}|}
\multicolumn{2}{l}{
$(\MF{g},\MF{h})=(\MF{so}^{*}(4n),\MF{so}^{*}(2n)+\MF{so}^{*}(2n))$
}\\
\hline
$\Theta(\subset \varPsi)$ &
$\varPsi\setminus\{e_{i_{1}}-e_{i_{1}+1},\ldots,
e_{i_{k-1}}-e_{i_{k-1}+1},2e_{n}
\}$
\\\hline
$\MF{h}_{\Theta}$
& 
$\displaystyle{\sum^{k}_{l=1}\MF{so}^{*}(2(i_{l}-i_{l-1}))}$
\\\hline
Remarks
&
$0=i_{0}<i_{1}<\cdots<i_{k}=n$
\\
\hline
\hline
$\Theta(\subset \varPsi)$ &
$\varPsi\setminus\{e_{i_{1}}-e_{i_{1}+1},\ldots,e_{i_{k}}-e_{i_{k}+1}\}$
\\\hline
$\MF{h}_{\Theta}$
& 
$\displaystyle{\sum^{k}_{l=1}\MF{so}^{*}(2(i_{l}-i_{l-1}))+\MF{so}^{*}(2(n-i_{k}))^{2}}$
\\\hline
Remarks
&
$0=i_{0}<i_{1}<\cdots<i_{k}<n$
\\
\hline
\end{tabular}
\end{center}
\begin{center}
 \begin{tabular}{|>{\PBS\centering}p{\LENGTHTHETA}|>{\PBS\centering}p{\LENGTHH}|}
\multicolumn{2}{l}{
$(\MF{g},\MF{h})=(\MF{sp}(n,n),\MF{su}(n,n)+\MF{so}(2))$
}\\
\hline
$\Theta(\subset \varPsi)$ &
$\varPsi\setminus\{e_{i_{1}}-e_{i_{1}+1},\ldots, e_{i_{k-1}}-e_{i_{k-1}+1},2e_{n}\}$
\\\hline
$\MF{h}_{\Theta}$
& 
$\displaystyle{\sum^{k}_{l=1}\MF{so}^{*}(2(i_{l}-i_{l-1}))}$
\\\hline
Remarks
&
$0=i_{0}<i_{1}<\cdots<i_{k}=n$
\\
\hline
\hline
$\Theta(\subset \varPsi)$ &
$\varPsi\setminus\{e_{i_{1}}-e_{i_{1}+1},\ldots,e_{i_{k}}-e_{i_{k}+1}\}$
\\\hline
$\MF{h}_{\Theta}$
&
$\displaystyle{\sum^{k}_{l=1}}\MF{so}^{*}(2(i_{l}-i_{l-1}))+\MF{su}(n-i_{k},n-i_{k})+\MF{so}(2)$
\\\hline
Remarks
&
$0=i_{0}<i_{1}<\cdots<i_{k}<n$
\\
\hline
\end{tabular}
\end{center}
\begin{center}
\begin{tabular}{|>{\PBS\centering}p{\LENGTHTHETA}|>{\PBS\centering}p{\LENGTHH}|}
\multicolumn{2}{l}{
$(\MF{g},\MF{h})=(\MF{sp}(2n,\BS{R}),\MF{sp}(n,\BS{R})+\MF{sp}(n,\BS{R}))$
}\\
\hline
$\Theta(\subset\varPsi)$ &
$\varPsi\setminus\{e_{i_{1}}-e_{i_{1}+1},\ldots,
e_{i_{k-1}}-e_{i_{k-1}+1},2e_{n}
\}$
\\\hline
$\MF{h}_{\Theta}$
& 
$\displaystyle{\sum^{k}_{l=1}\MF{sp}(i_{l}-i_{l-1},\BS{R})}$
\\\hline
Remarks
&
$0=i_{0}<i_{1}<\cdots<i_{k}=n$
\\
\hline
\hline
$\Theta(\subset \varPsi)$ &
$\varPsi\setminus\{e_{i_{1}}-e_{i_{1}+1},\ldots,e_{i_{k}}-e_{i_{k}+1}\}$
\\\hline
$\MF{h}_{\Theta}$
& 
$\displaystyle{\sum^{k}_{l=1}\MF{sp}(i_{l}-i_{l-1},\BS{R})+\MF{sp}(n-i_{k},\BS{R})^{2}}$
\\\hline
Remarks
&
$0=i_{0}<i_{1}<\cdots<i_{k}<n$
\\
\hline
\end{tabular}
\end{center}
\end{table}

\begin{table}[htbp]
\footnotesize
\contcaption{(continued)}
\begin{center}
\begin{tabular}{|>{\PBS\centering}p{\LENGTHTHETA}|>{\PBS\centering}p{\LENGTHH}|}
\multicolumn{2}{l}{
$(\MF{g},\MF{h})=(\MF{sp}(2n,\BS{R}),\MF{sp}(n,\BS{C})$
}\\
\hline
$\Theta(\subset \varPsi)$ &
$\varPsi\setminus\{e_{i_{1}}-e_{i_{1}+1},\ldots,
e_{i_{k-1}}-e_{i_{k-1}+1},2e_{n}\}$
\\\hline
$\MF{h}_{\Theta}$
& 
$\displaystyle{\sum^{k}_{l=1}}\MF{sp}(i_{l}-i_{l-1},\BS{R})$
\\\hline
Remarks
&
$0=i_{0}<i_{1}<\cdots<i_{k}= n$
\\
\hline
\hline
$\Theta(\subset \varPsi)$ &
$\varPsi\setminus\{e_{i_{1}}-e_{i_{1}+1},\ldots,e_{i_{k}}-e_{i_{k}+1}\}$
\\\hline
$\MF{h}_{\Theta}$
& 
$\displaystyle{\sum^{k}_{l=1}}\MF{sp}(i_{l}-i_{l-1},\BS{R})+\MF{sp}(n-i_{k},\BS{C})$
\\\hline
Remarks
&
$0=i_{0}<i_{1}<\cdots<i_{k}< n$
\\
\hline
\end{tabular}
\end{center}
\begin{center}
\begin{tabular}{|>{\PBS\centering}p{\LENGTHTHETA}|>{\PBS\centering}p{\LENGTHH}|}
\multicolumn{2}{l}{
$(\MF{g},\MF{h})=(\MF{sp}(n,n),\MF{su}^{*}(2n)+\BS{R})$
}\\
\hline
$\Theta(\subset \varPsi)$ &
$\varPsi\setminus\{e_{i_{1}}-e_{i_{1}+1},\ldots,
e_{i_{k-1}}-e_{i_{k-1}+1},2e_{n}\}$
\\\hline
$\MF{h}_{\Theta}$
& 
$\displaystyle{\sum^{k}_{l=1}}\MF{so}^{*}(2(i_{l}-i_{l-1}))$
\\\hline
Remarks
&
$0=i_{0}<i_{1}<\cdots<i_{k}= n$
\\
\hline
\hline
$\Theta(\subset \varPsi)$ &
$\varPsi\setminus\{e_{i_{1}}-e_{i_{1}+1},\ldots,e_{i_{k}}-e_{i_{k}+1}\}$
\\\hline
$\MF{h}_{\Theta}$
&
$\displaystyle{\sum^{k}_{l=1}}\MF{so}^{*}(2(i_{l}-i_{l-1}))+\MF{su}^{*}(2(n-i_{k}))+\BS{R}$
\\\hline
Remarks
&
$0=i_{0}<i_{1}<\cdots<i_{k}<n$
\\
\hline
\end{tabular}
\end{center}
\medskip
\begin{flushleft}
(IV) $\Delta=D_{r},\,\Delta^{a}=\Delta,\,\varPsi=\{e_{i}-e_{i+1}\mid 1 \leq i \leq r-1\}\cup\{e_{r-1}+e_{r}\}$
\end{flushleft}
\begin{center}
\begin{tabular}{|>{\PBS\centering}p{\LENGTHTHETA}|>{\PBS\centering}p{\LENGTHH}|}
\multicolumn{2}{l}{
$(\MF{g},\MF{h})=(\MF{so}(n,n)+\MF{so}(n,n),\MF{so}(n,n))$
}\\
\hline
$\Theta(\subset \varPsi)$ &
$\varPsi \setminus \{e_{i_{1}}-e_{i_{1}+1},\ldots,e_{i_{k-1}}-e_{i_{k-1}},e_{n-1}+e_{n}\}$
\\\hline
$\MF{h}_{\Theta}$& 
$\BS{R}^{k}+\displaystyle{\sum^{k}_{l=1}}\MF{sl}(i_{l}-i_{l-1},\BS{R})$
\\\hline
Remarks
&
$0=i_{0}<i_{1}<\cdots<i_{k}=n$
\\
\hline
\hline
$\Theta(\subset \varPsi)$&
$\varPsi \setminus \{e_{i_{1}}-e_{i_{1}+1},\ldots,e_{i_{k-1}}-e_{i_{k-1}},e_{i_{k}}-e_{i_{k}+1}\}$
\\\hline
$\MF{h}_{\Theta}$
&
$\BS{R}^{k}+\displaystyle{\sum^{k}_{l=1}}\MF{sl}(i_{l}-i_{l-1},\BS{R})+\MF{so}(n-i_{k},n-i_{k})$
\\\hline
Remarks
&
$0=i_{0}<i_{1}<\cdots<i_{k}< n$
\\
\hline
\end{tabular}
\end{center}
\begin{center}
\begin{tabular}{|>{\PBS\centering}p{\LENGTHTHETA}|>{\PBS\centering}p{\LENGTHH}|}
\multicolumn{2}{l}{
$(\MF{g},\MF{h})=(\MF{so}(2n,\BS{C}),\MF{so}(n,\BS{C})+\MF{so}(n,\BS{C}))$
}\\
\hline
$\Theta(\subset \varPsi)$ &
$\varPsi\setminus\{e_{i_{1}}-e_{i_{1}+1},\ldots,e_{i_{k-1}}-e_{i_{k-1}+1},e_{n-1}+e_{n}\}$
\\\hline
$\MF{h}_{\Theta}$& 
$\displaystyle{\sum^{k}_{l=1}}\MF{so}(i_{l}-i_{l-1},\BS{C})$
\\\hline
Remarks
&
$0=i_{0}<i_{1}<\cdots<i_{k}=n$
\\
\hline
\hline
$\Theta (\subset \varPsi)$ &
$\varPsi\setminus\{e_{i_{1}}-e_{i_{1}+1},\ldots,e_{i_{k}}-e_{i_{k}+1}\}$
\\\hline
$\MF{h}_{\Theta}$&
$\displaystyle{\sum^{k}_{l=1}}\MF{so}(i_{l}-i_{l-1},\BS{C})+\MF{so}(n-i_{k},\BS{C})^{2}$
\\\hline
Remarks
&
$0=i_{0}<i_{1}<\cdots<i_{k}<n$
\\
\hline
\end{tabular}
\end{center}
\end{table}

\begin{table}[htbp]
\footnotesize
\begin{flushleft}
(V) $\Delta=\Delta^{a}=(BC)_{r},\,\varPsi=\{e_{i}-e_{i+1}\mid 1 \leq i \leq r-1\}\cup\{e_{r}\}$ 
\end{flushleft}
\begin{center}
\begin{tabular}{|>{\PBS\centering}p{\LENGTHTHETA}|>{\PBS\centering}p{\LENGTHH}|}
\multicolumn{2}{l}{
$(\MF{g},\MF{h})=(\MF{su}(p,n-p)+\MF{su}(p,n-p),\MF{su}(p,n-p))\,(n > 2p)$
}\\
\hline
$\Theta(\subset \varPsi)$ &
$\varPsi\setminus\{e_{i_{1}}-e_{i_{1}+1},\ldots,
e_{i_{k-1}}-e_{i_{k-1}+1},e_{p}
\}$
\\\hline
$\MF{h}_{\Theta}$
&
$\BS{R}^{k}+\MF{so}(2)^{k}+\displaystyle{\sum^{k}_{l=1}}\MF{sl}(i_{l}-i_{l-1},\BS{C})$
\\\hline
Remarks
&
$0=i_{0}<i_{1}<\dots<i_{k}= p$\\
\hline
\hline
$\Theta(\subset \varPsi)$ &
$\varPsi\setminus\{e_{i_{1}}-e_{i_{1}+1},\ldots,e_{i_{k}}-e_{i_{k}+1}\}$
\\\hline
$\MF{h}_{\Theta}$&
$\BS{R}^{k}+\MF{so}(2)^{k}+\displaystyle{\sum^{k}_{l=1}}\MF{sl}(i_{l}-i_{l-1},\BS{C})+\MF{su}(n-2p)$
\\\hline
Remarks
&
$0=i_{0}<i_{1}<\dots<i_{k}<p$\\
\hline
\end{tabular}
\end{center}
\begin{center}
\begin{tabular}{|>{\PBS\centering}p{\LENGTHTHETA}|>{\PBS\centering}p{\LENGTHH}|}
\multicolumn{2}{l}{
$(\MF{g},\MF{h})=(\MF{sl}(n,\BS{C}),\MF{sl}(p,\BS{C})+\MF{sl}(n-p,\BS{C})+\BS{C})\,(n > 2p)$
}\\
\hline
$\Theta(\subset \varPsi)$ &
$\varPsi\setminus\{e_{i_{1}}-e_{i_{1}+1},\ldots,
e_{i_{k-1}}-e_{i_{k-1}+1},e_{p}
\}$
\\\hline
$\MF{h}_{\Theta}$&
$\BS{C}^{k}+\displaystyle{\sum^{k}_{l=1}}\MF{sl}(i_{l}-i_{l-1},\BS{C})+\MF{sl}(n-2p,\BS{C})$
\\\hline
Remarks
&
$0=i_{0}<i_{1}<\dots<i_{k}= p$\\
\hline
\hline
$\Theta(\subset \varPsi)$ &
$\varPsi\setminus\{e_{i_{1}}-e_{i_{1}+1},\ldots,e_{i_{k}}-e_{i_{k}+1}\}$
\\\hline
$\MF{h}_{\Theta}$&
$\BS{C}^{k}+\displaystyle{\sum^{k}_{l=1}}\MF{sl}(i_{l}-i_{l-1},\BS{C})+\MF{sl}(p-i_{k},\BS{C})+\MF{sl}(n-p-i_{k},\BS{C})+\BS{C}$
\\\hline
Remarks
&
$0=i_{0}<i_{1}<\dots<i_{k}<p$\\
\hline
\end{tabular}
\end{center}
\begin{center}
\begin{tabular}{|>{\PBS\centering}p{\LENGTHTHETA}|>{\PBS\centering}p{\LENGTHH}|}
\multicolumn{2}{l}{
$(\MF{g},\MF{h})=(\MF{so}^{*}(2(2n+1))+\MF{so}^{*}(2(2n+1)),\MF{so}^{*}(2(2n+1)))$
}\\
\hline
$\Theta(\subset \varPsi)$ &
$\varPsi\setminus\{e_{i_{1}}-e_{i_{1}+1},\ldots, e_{i_{k-1}}-e_{i_{k-1}+1},e_{n}\}$
\\\hline
$\MF{h}_{\Theta}$& 
$\BS{R}^{k}+\MF{so}(2)+\displaystyle{\sum^{k}_{l=1}}\MF{su}^{*}(2(i_{l}-i_{l-1}))$
\\\hline
Remarks
&
$0=i_{0}<i_{1}<\cdots<i_{k}=n$
\\
\hline
\hline
$\Theta(\subset \varPsi)$ &
$\varPsi\setminus\{e_{i_{1}}-e_{i_{1}+1},\ldots,e_{i_{k}}-e_{i_{k}+1}\}$
\\\hline
$\MF{h}_{\Theta}$&
$\BS{R}^{k}+\displaystyle{\sum^{k}_{l=1}}\MF{su}^{*}(2(i_{l}-i_{l-1}))+\MF{so}^{*}(2(2(n-i_{k})+1))$
\\\hline
Remarks
&
$0=i_{0}<i_{1}<\cdots<i_{k}<n$
\\
\hline
\end{tabular}
\end{center}
\begin{center}
\begin{tabular}{|>{\PBS\centering}p{\LENGTHTHETA}|>{\PBS\centering}p{\LENGTHH}|}
\multicolumn{2}{l}{
$(\MF{g},\MF{h})=(\MF{so}(2(2n+1),\BS{C}),\MF{sl}(2n+1,\BS{C})+\BS{C})$
}\\
\hline
$\Theta(\subset \varPsi)$ &
$\varPsi\setminus\{e_{i_{1}}-e_{i_{1}+1},\ldots, e_{i_{k-1}}-e_{i_{k-1}+1},e_{n}\}$
\\\hline
$\MF{h}_{\Theta}$
& 
$\BS{C}+\displaystyle{\sum^{k}_{l=1}}\MF{sp}(i_{l}-i_{l-1},\BS{C})$
\\\hline
Remarks
&
$0=i_{0}<i_{1}<\cdots<i_{k}=n$
\\
\hline
\hline
$\Theta(\subset \varPsi)$ &
$\varPsi\setminus\{e_{i_{1}}-e_{i_{1}+1},\ldots,e_{i_{k}}-e_{i_{k}+1}\}$
\\\hline
$\MF{h}_{\Theta}$
&
$\displaystyle{\sum^{k}_{l=1}}\MF{sp}(i_{l}-i_{l-1},\BS{C})+\MF{sl}(2(n-i_{k})+1,\BS{C})+\BS{C}$
\\\hline
Remarks
&
$0=i_{0}<i_{1}<\cdots<i_{k}<n$
\\
\hline
\end{tabular}
\end{center}
\end{table}

\begin{table}[htbp]
\footnotesize
\contcaption{(continued)}
\begin{center}
\begin{tabular}{|>{\PBS\centering}p{\LENGTHTHETA}|>{\PBS\centering}p{\LENGTHH}|}
\multicolumn{2}{l}{
$(\MF{g},\MF{h})=(\MF{sp}(p,n-p)+\MF{sp}(p,n-p),\MF{sp}(p,n-p))\,(n>2p)$
}\\
\hline
$\Theta(\subset \varPsi)$ &
$\varPsi\setminus\{e_{i_{1}}-e_{i_{1}+1},\ldots, e_{i_{k-1}}-e_{i_{k-1}+1},e_{p}\}$
\\\hline
$\MF{h}_{\Theta}$& 
$\BS{R}^{k}+\displaystyle{\sum^{k}_{l=1}}\MF{su}^{*}(2(i_{l}-i_{l-1}))+\MF{sp}(n-2p)$
\\\hline
Remarks
&
$0=i_{0}<i_{1}<\cdots<i_{k}= p$
\\
\hline
\hline
$\Theta(\subset \varPsi)$ &
$\varPsi\setminus\{e_{i_{1}}-e_{i_{1}+1},\ldots, e_{i_{k-1}}-e_{i_{k-1}+1},e_{p}\}$
\\\hline
$\MF{h}_{\Theta}$& 
$\BS{R}^{k}+\displaystyle{\sum^{k}_{l=1}}\MF{su}^{*}(2(i_{l}-i_{l-1}))+\MF{sp}(p-i_{k},n-p-i_{k})$
\\\hline
Remarks
&
$0=i_{0}<i_{1}<\cdots<i_{k}< p$
\\
\hline
\end{tabular}
\end{center}
\begin{center}
\begin{tabular}{|>{\PBS\centering}p{\LENGTHTHETA}|>{\PBS\centering}p{\LENGTHH}|}
\multicolumn{2}{l}{
$(\MF{g},\MF{h})=(\MF{sp}(n,\BS{C}),\MF{sp}(p,\BS{C})+\MF{sp}(n-p,\BS{C}))\,(n>2p)$
}\\
\hline
$\Theta(\subset \varPsi)$ &
$\varPsi\setminus\{e_{i_{1}}-e_{i_{1}+1},\ldots,
e_{i_{k-1}}-e_{i_{k-1}+1},2e_{n}\}$
\\\hline
$\MF{h}_{\Theta}$
& 
$\displaystyle{\sum^{k}_{l=1}}\MF{sp}(i_{l}-i_{l-1},\BS{C})+\MF{sp}(p-i_{k},\BS{C})+\MF{sp}(n-2p,\BS{C})$
\\\hline
Remarks
&
$0=i_{0}<i_{1}<\cdots<i_{k}= p$
\\
\hline
\hline
$\Theta(\subset \varPsi)$ &
$\varPsi\setminus\{e_{i_{1}}-e_{i_{1}+1},\ldots,e_{i_{k}}-e_{i_{k}+1}\}$
\\\hline
$\MF{h}_{\Theta}$
& 
$\displaystyle{\sum^{k}_{l=1}}\MF{sp}(i_{l}-i_{l-1},\BS{C})+\MF{sp}(p-i_{k},\BS{C})+\MF{sp}(n-p-i_{k},\BS{C})$
\\\hline
Remarks
&
$0=i_{0}<i_{1}<\cdots<i_{k}< p$
\\
\hline
\end{tabular}
\end{center}
\begin{center}
\begin{tabular}{|>{\PBS\centering}p{\LENGTHTHETA}|>{\PBS\centering}p{\LENGTHH}|}
\multicolumn{2}{l}{
$(\MF{g},\MF{h})=(\MF{su}(2p,2(n-p)),\MF{sp}(p,n-p))\,(n>2p)$
}\\
\hline
$\Theta(\subset \varPsi)$ &
$\varPsi\setminus\{e_{i_{1}}-e_{i_{1}+1},\ldots, e_{i_{k-1}}-e_{i_{k-1}+1},e_{p}\}$
\\\hline
$\MF{h}_{\Theta}$& 
$\displaystyle{\sum^{k}_{l=1}}\MF{sp}(i_{l}-i_{l-1},\BS{C})+\MF{sp}(n-2p)$
\\\hline
Remarks
&
$0=i_{0}<i_{1}<\cdots<i_{k}= p$
\\
\hline
\hline
$\Theta(\subset \varPsi)$ &
$\varPsi\setminus\{e_{i_{1}}-e_{i_{1}+1},\ldots,e_{i_{k}}-e_{i_{k}+1}\}$
\\\hline
$\MF{h}_{\Theta}$& 
$\displaystyle{\sum^{k}_{l=1}}\MF{sp}(i_{l}-i_{l-1},\BS{C})+\MF{sp}(p-i_{k},n-p-i_{k})$
\\\hline
Remarks
&
$0=i_{0}<i_{1}<\cdots<i_{k}< p$
\\
\hline
\end{tabular}
\end{center}
\begin{center}
\begin{tabular}{|>{\PBS\centering}p{\LENGTHTHETA}|>{\PBS\centering}p{\LENGTHH}|}
\multicolumn{2}{l}{
$(\MF{g},\MF{h})=(\MF{su}^{*}(2n),\MF{su}^{*}(2p)+\MF{su}^{*}(2(n-p))+\BS{R})\,(n>2p)$
}\\
\hline
$\Theta(\subset \varPsi)$ &
$\varPsi\setminus\{e_{i_{1}}-e_{i_{1}+1},\ldots, e_{i_{k-1}}-e_{i_{k-1}+1},e_{p}\}$
\\\hline
$\MF{h}_{\Theta}$
& 
$\BS{R}^{k}+\displaystyle{\sum^{k}_{l=1}}\MF{su}^{*}(2(i_{l}-i_{l-1}))+\MF{su}^{*}(2(n-2p))$
\\\hline
Remarks
&
$0=i_{0}<i_{1}<\cdots<i_{k}=p$
\\
\hline
\hline
$\Theta(\subset \varPsi)$&
$\varPsi\setminus\{e_{i_{1}}-e_{i_{1}+1},\ldots,e_{i_{k}}-e_{i_{k}+1}\}$
\\\hline
$\MF{h}_{\Theta}$
&
$\BS{R}^{k}+\displaystyle{\sum^{k}_{l=1}}\MF{su}^{*}(2(i_{l}-i_{l-1}))+\MF{su}^{*}(2(p-i_{k}))+\MF{su}^{*}(2(n-p-i_{k}))+\BS{R}$
\\\hline
Remarks
&
$0=i_{0}<i_{1}<\cdots<i_{k}<p$
\\
\hline
\end{tabular}
\end{center}
\end{table}

\begin{table}[htbp]
\footnotesize
\contcaption{(continued)}
\begin{center}
\begin{tabular}{|>{\PBS\centering}p{\LENGTHTHETA}|>{\PBS\centering}p{\LENGTHH}|}
\multicolumn{2}{l}{
$(\MF{g},\MF{h})=(\MF{su}^{*}(2(2n+1)),\MF{sl}(2n+1,\BS{C})+\MF{so}(2))$
}\\
\hline
$\Theta(\subset \varPsi)$ &
$\varPsi\setminus\{e_{i_{1}}-e_{i_{1}+1},\ldots, e_{i_{k-1}}-e_{i_{k-1}+1},e_{n}\}$
\\\hline
$\MF{h}_{\Theta}$
& 
$\BS{R}^{k}+\MF{so}(2)+\displaystyle{\sum^{k}_{l=1}}\MF{su}^{*}(2(i_{l}-i_{l-1}))$
\\\hline
Remarks
&
$0=i_{0}<i_{1}<\cdots<i_{k}=n$
\\
\hline
\hline
$\Theta(\subset \varPsi)$ &
$\varPsi\setminus\{e_{i_{1}}-e_{i_{1}+1},\ldots,e_{i_{k}}-e_{i_{k}+1}\}$
\\\hline
$\MF{h}_{\Theta}$
&
$\BS{R}^{k}+\MF{so}(2)+\displaystyle{\sum^{k}_{l=1}}\MF{su}^{*}(2(i_{l}-i_{l-1}))+\MF{sl}(2(n-i_{k})+1,\BS{C})+\MF{so}(2)$
\\\hline
Remarks
&
$0=i_{0}<i_{1}<\cdots<i_{k}<n$
\\
\hline
\end{tabular}
\end{center}
\begin{center}
\begin{tabular}{|>{\PBS\centering}p{\LENGTHTHETA}|>{\PBS\centering}p{\LENGTHH}|}
\multicolumn{2}{l}{
$(\MF{g},\MF{h})=(\MF{su}(2n+1,2n+1),\MF{sp}(2n+1,\BS{R}))$
}\\
\hline
$\Theta(\subset \varPsi)$ &
$\varPsi\setminus\{e_{i_{1}}-e_{i_{1}+1},\ldots,
e_{i_{k-1}}-e_{i_{k-1}+1},e_{n}
\}$
\\\hline
$\MF{h}_{\Theta}$
& 
$\displaystyle{\sum^{k}_{l=1}}\MF{sp}(i_{l}-i_{l-1},\BS{C})$
\\\hline
Remarks
&
$0=i_{0}<i_{1}<\cdots<i_{k}= n$
\\
\hline
\hline
$\Theta(\subset \varPsi)$ &
$\varPsi\setminus\{e_{i_{1}}-e_{i_{1}+1},\ldots,e_{i_{k}}-e_{i_{k}+1}\}$
\\\hline
$\MF{h}_{\Theta}$& 
$\displaystyle{\sum^{k}_{l=1}}\MF{sp}(i_{l}-i_{l-1},\BS{C})+\MF{sp}(2(n-i_{k})+1,\BS{R})$
\\\hline
Remarks
&
$0=i_{0}<i_{1}<\cdots<i_{k}< n$
\\
\hline
\end{tabular}
\end{center}
\begin{center}
\begin{tabular}{|>{\PBS\centering}p{\LENGTHTHETA}|>{\PBS\centering}p{\LENGTHH}|}
\multicolumn{2}{l}{
$(\MF{g},\MF{h})=(\MF{so}(2p,2(n-p)),\MF{su}(p,n-p)+\MF{so}(2))\,(n>2p)$
}\\
\hline
$\Theta(\subset \varPsi)$ &
$\varPsi\setminus\{e_{i_{1}}-e_{i_{1}+1},\ldots, e_{i_{k-1}}-e_{i_{k-1}+1},e_{p}\}$
\\\hline
$\MF{h}_{\Theta}$
& 
$\displaystyle{\sum^{k}_{l=1}}\MF{sp}(i_{l}-i_{l-1},\BS{R})$
\\\hline
Remarks
&
$0=i_{0}<i_{1}<\cdots<i_{k} = p$
\\
\hline
\hline
$\Theta(\subset \varPsi)$ &
$\varPsi\setminus\{e_{i_{1}}-e_{i_{1}+1},\ldots,e_{i_{k}}-e_{i_{k}+1}\}$
\\\hline
$\MF{h}_{\Theta}$& 
$\displaystyle{\sum^{k}_{l=1}}\MF{sp}(i_{l}-i_{l-1},\BS{R})+\MF{su}(p-i_{k},n-p-i_{k})+\MF{so}(2)$
\\\hline
Remarks
&
$0=i_{0}<i_{1}<\cdots<i_{k}< p$
\\
\hline
\end{tabular}
\end{center}
\begin{center}
\begin{tabular}{|>{\PBS\centering}p{\LENGTHTHETA}|>{\PBS\centering}p{\LENGTHH}|}
\multicolumn{2}{l}{
$(\MF{g},\MF{h})=(\MF{so}^{*}(2n),\MF{so}^{*}(2p)+\MF{so}^{*}(2(n-p)))\,(n>2p)$
}\\
\hline
$\Theta(\subset \varPsi)$ &
$\varPsi\setminus\{e_{i_{1}}-e_{i_{1}+1},\ldots,
e_{i_{k-1}}-e_{i_{k-1}+1},e_{p}
\}$
\\\hline
$\MF{h}_{\Theta}$
& 
$\displaystyle{\sum^{k}_{l=1}}\MF{so}^{*}(2(i_{l}-i_{l-1}))$
\\\hline
Remarks
&
$0=i_{0}<i_{1}<\cdots<i_{k}= p$
\\
\hline
\hline
$\Theta(\subset \varPsi)$ &
$\varPsi\setminus\{e_{i_{1}}-e_{i_{1}+1},\ldots,e_{i_{k}}-e_{i_{k}+1}\}$
\\\hline
$\MF{h}_{\Theta}$
& 
$\displaystyle{\sum^{k}_{l=1}}\MF{so}^{*}(2(i_{l}-i_{l-1}))+\MF{so}^{*}(2(p-i_{k}))+\MF{so}^{*}(2(n-p-i_{k}))$
\\\hline
Remarks
&
$0=i_{0}<i_{1}<\cdots<i_{k}< p$
\\
\hline
\end{tabular}
\end{center}
\end{table}

\begin{table}[htbp]
\footnotesize
\contcaption{(continued)}
\begin{center}
\begin{tabular}{|>{\PBS\centering}p{\LENGTHTHETA}|>{\PBS\centering}p{\LENGTHH}|}
\multicolumn{2}{l}{
$(\MF{g},\MF{h})=(\MF{sp}(p,n-p),\MF{su}(p,n-p)+\MF{so}(2))\,(n>2p)$
}\\
\hline
$\Theta(\subset \varPsi)$ &
$\varPsi\setminus\{e_{i_{1}}-e_{i_{1}+1},\ldots, e_{i_{k-1}}-e_{i_{k-1}+1},e_{p}\}$
\\\hline
$\MF{h}_{\Theta}$
& 
$\MF{u}(n-2p)+\displaystyle{\sum^{k}_{l=1}}\MF{so}^{*}(2(i_{l}-i_{l-1}))$
\\\hline
Remarks
&
$0=i_{0}<i_{1}<\cdots<i_{k}=p$
\\
\hline
\hline
$\Theta(\subset \varPsi)$ &
$\varPsi\setminus\{e_{i_{1}}-e_{i_{1}+1},\ldots,e_{i_{k}}-e_{i_{k}+1}\}$
\\\hline
$\MF{h}_{\Theta}$
&
$\displaystyle{\sum^{k}_{l=1}}\MF{so}(2(i_{l}-i_{l-1}))+\MF{su}(p-i_{k},n-p-i_{k})+\MF{so}(2)$
\\\hline
Remarks
&
$0=i_{0}<i_{1}<\cdots<i_{k}<p$
\\
\hline
\end{tabular}
\end{center}
\begin{center}
\begin{tabular}{|>{\PBS\centering}p{\LENGTHTHETA}|>{\PBS\centering}p{\LENGTHH}|}
\multicolumn{2}{l}{
$(\MF{g},\MF{h})=(\MF{sp}(n,\BS{R}),\MF{sp}(p,\BS{R})+\MF{sp}(n-p,\BS{R}))\,(n>2p)$
}\\
\hline
$\Theta(\subset \varPsi)$ &
$\varPsi\setminus\{e_{i_{1}}-e_{i_{1}+1},\ldots, e_{i_{k-1}}-e_{i_{k-1}+1},e_{p}\}$
\\\hline
$\MF{h}_{\Theta}$
& 
$\displaystyle{\sum^{k}_{l=1}}\MF{sp}(i_{l}-i_{l-1},\BS{R})+\MF{sp}(n-2p,\BS{R})$
\\\hline
Remarks
&
$0=i_{0}<i_{1}<\cdots<i_{k}= p$
\\
\hline
\hline
$\Theta(\subset \varPsi)$ &
$\varPsi\setminus\{e_{i_{1}}-e_{i_{1}+1},\ldots,e_{i_{k}}-e_{i_{k}+1}\}$
\\\hline
$\MF{h}_{\Theta}$
& 
$\displaystyle{\sum^{k}_{l=1}}\MF{sp}(i_{l}-i_{l-1},\BS{R})+\MF{sp}(p-i_{k},\BS{R})+\MF{sp}(n-p-i_{k},\BS{R})$
\\\hline
Remarks
&
$0=i_{0}<i_{1}<\cdots<i_{k}< p$
\\
\hline
\end{tabular}
\end{center}
\medskip
\begin{flushleft}
(VI) $\Delta=(BC)_{r},\,\Delta^{a}=B_{r},\,\varPsi=\{e_{i}-e_{i+1}\mid 1 \leq i \leq r-1\}\cup\{e_{r}\}$
\end{flushleft}
\begin{center}
\begin{tabular}{|>{\PBS\centering}p{\LENGTHTHETA}|>{\PBS\centering}p{\LENGTHH}|}
\multicolumn{2}{l}{
$(\MF{g},\MF{h})=(\MF{sl}(2n+1,\BS{C}),\MF{sl}(2n+1,\BS{R}))$
}\\
\hline
$\Theta(\subset \varPsi)$&
$\varPsi\setminus\{e_{i_{1}}-e_{i_{1}+1},\ldots,
e_{i_{k-1}}-e_{i_{k-1}+1},e_{n}
\}$
\\\hline
$\MF{h}_{\Theta}$
&
$\displaystyle{\BS{R}^{k}+\MF{so}(2)^{k}+\sum^{k}_{l=1}}\MF{sl}(i_{l}-i_{l-1},\BS{C})$
\\\hline
Remarks
&$0=i_{0}<i_{1}<\dots<i_{k}= n$\\
\hline
\hline
$\Theta(\subset \varPsi)$ &
$\varPsi\setminus\{e_{i_{1}}-e_{i_{1}+1},\ldots,e_{i_{k}}-e_{i_{k}+1}\}$
\\\hline
$\MF{h}_{\Theta}$&
$\BS{R}^{k}+\MF{so}(2)^{k}+\displaystyle{\sum^{k}_{l=1}}\MF{sl}(i_{l}-i_{l-1},\BS{C})+\MF{sl}(2(n-i_{k})+1,\BS{R})$
\\\hline
Remarks
&$0=i_{0}<i_{1}<\dots<i_{k}<n$\\
\hline
\end{tabular}
\end{center}
\begin{center}
\begin{tabular}{|>{\PBS\centering}p{\LENGTHTHETA}|>{\PBS\centering}p{\LENGTHH}|}
\multicolumn{2}{l}{
$(\MF{g},\MF{h})=(\MF{su}(p,n-p),\MF{so}(p,n-p))$
}\\
\hline
$\Theta(\subset \varPsi)$&
$\varPsi\setminus\{e_{i_{1}}-e_{i_{1}+1},\ldots, e_{i_{k-1}}-e_{i_{k-1}+1},e_{p}\}$
\\\hline
$\MF{h}_{\Theta}$& 
$\displaystyle{\sum_{l=1}^{k}\MF{so}(i_{l}-i_{l-1},\BS{C})+\MF{so}(n-2p)}$
\\\hline
Remarks
&
$0=i_{0}<i_{1}<\cdots<i_{k} = p$
\\
\hline
\hline
$\Theta(\subset \varPsi)$ &
$\varPsi\setminus\{e_{i_{1}}-e_{i_{1}+1},\ldots,e_{i_{k}}-e_{i_{k}+1}\}$
\\\hline
$\MF{h}_{\Theta}$& 
$\displaystyle{\sum_{l=1}^{k}\MF{so}(i_{l}-i_{l-1},\BS{C})+\MF{so}(p-i_{k},n-p-i_{k})}$
\\\hline
Remarks &
$0=i_{0}<i_{1}<\cdots<i_{k} < p$
\\
\hline
\end{tabular}
\end{center}
\end{table}

\begin{table}[htbp]
\footnotesize
\contcaption{(continued)}
\begin{center}
\begin{tabular}{|>{\PBS\centering}p{\LENGTHTHETA}|>{\PBS\centering}p{\LENGTHH}|}
\multicolumn{2}{l}{
$(\MF{g},\MF{h})=(\MF{sl}(n,\BS{R}),\MF{sl}(p,\BS{R})+\MF{sl}(n-p,\BS{R})+\BS{R})\,(n>2p)$
}\\
\hline
$\Theta(\subset \varPsi)$ &
$\varPsi\setminus\{e_{i_{1}}-e_{i_{1}+1},\ldots, e_{i_{k-1}}-e_{i_{k-1}+1},e_{p}\}$
\\\hline
$\MF{h}_{\Theta}$& 
$\BS{R}^{k}+\displaystyle{\sum^{k}_{l=1}}\MF{sl}(i_{l}-i_{l-1},\BS{R})+\MF{sl}(n-2p,\BS{R})$
\\\hline
Remarks&
$0=i_{0}<i_{1}<\cdots<i_{k}=p$
\\
\hline
\hline
$\Theta(\subset \varPsi)$ &
$\varPsi\setminus\{e_{i_{1}}-e_{i_{1}+1},\ldots,e_{i_{k}}-e_{i_{k}+1}\}$
\\\hline
$\MF{h}_{\Theta}$&
$\BS{R}^{k}+\displaystyle{\sum^{k}_{l=1}}\MF{sl}(i_{l}-i_{l-1},\BS{R})+\MF{sl}(p-i_{k},\BS{R})+\MF{sl}(n-p-i_{k},\BS{R})+\BS{R}$
\\\hline
Remarks &
$0=i_{0}<i_{1}<\cdots<i_{k}<p$
\\
\hline
\end{tabular}
\end{center}
\begin{center}
\begin{tabular}{|>{\PBS\centering}p{\LENGTHTHETA}|>{\PBS\centering}p{\LENGTHH}|}
\multicolumn{2}{l}{
$(\MF{g},\MF{h})=(\MF{so}^{*}(2(2n+1)),\MF{so}(2n+1,\BS{C}))$
}\\
\hline
$\Theta(\subset \varPsi)$ &
$\varPsi\setminus\{e_{i_{1}}-e_{i_{1}+1},\ldots, e_{i_{k-1}}-e_{i_{k-1}+1},e_{n}\}$
\\\hline
$\MF{h}_{\Theta}$
& 
$\displaystyle{\sum^{k}_{l=1}}\MF{so}^{*}(2(i_{l}-i_{l-1}))$
\\\hline
Remarks
&
$0=i_{0}<i_{1}<\cdots<i_{k}= n$
\\
\hline
\hline
$\Theta(\subset \varPsi)$ &
$\varPsi\setminus\{e_{i_{1}}-e_{i_{1}+1},\ldots,e_{i_{k}}-e_{i_{k}+1}\}$
\\\hline
$\MF{h}_{\Theta}$ & 
$\displaystyle{\sum^{k}_{l=1}}\MF{so}^{*}(2(i_{l}-i_{l-1}))+\MF{so}(2(n-i_{k})+1,\BS{C})$
\\\hline
Remarks
&
$0=i_{0}<i_{1}<\cdots<i_{k}< n$
\\
\hline
\end{tabular}
\end{center}
\begin{center}
\begin{tabular}{|>{\PBS\centering}p{\LENGTHTHETA}|>{\PBS\centering}p{\LENGTHH}|}
\multicolumn{2}{l}{
$(\MF{g},\MF{h})=(\MF{so}(2n+1,2n+1),\MF{sl}(2n+1,\BS{R})+\BS{R})$
}\\
\hline
$\Theta(\subset \varPsi)$&
$\varPsi\setminus\{e_{i_{1}}-e_{i_{1}+1},\ldots, e_{i_{k-1}}-e_{i_{k-1}+1},e_{n}\}$
\\\hline
$\MF{h}_{\Theta}$
& 
$\BS{R}+\displaystyle{\sum^{k}_{l=1}}\MF{sp}(i_{l}-i_{l-1},\BS{R})$
\\\hline
Remarks
&
$0=i_{0}<i_{1}<\cdots<i_{k}=n$
\\
\hline
\hline
$\Theta(\subset \varPsi)$&
$\varPsi\setminus\{e_{i_{1}}-e_{i_{1}+1},\ldots,e_{i_{k}}-e_{i_{k}+1}\}$
\\\hline
$\MF{h}_{\Theta}$&
$\displaystyle{\sum^{k}_{l=1}}\MF{sp}(i_{l}-i_{l-1},\BS{R})+\MF{sl}(2(n-i_{k})+1,\BS{R})+\BS{R}$
\\\hline
Remarks &
$0=i_{0}<i_{1}<\cdots<i_{k}<n$
\\
\hline
\end{tabular}
\end{center}
\end{table}

\begin{Ex}\label{ex.local2}
Let $G/H$ be a semisimple pseudo-Riemannian symmetric space.
Suppose that the restricted root system $\Delta$
with respect to a vector-type maximal split abelian subspace
satisfies $(\Delta, \Delta^{a})=(C_{r}, D_{r})$.
By using the argument as in Subsection \ref{subsec.CD}
and Theorem \ref{thm.main},
we have $\MC{L}_{h}(G/H)=\{[\MF{h}_{\Theta}]\mid \Theta \subset \varPsi\}
\cup\{[\MF{h}_{\Theta}]\mid \Theta \subset s_{2e_{r}}\cdot\varPsi\}$,
where $\varPsi=\{e_{i}-e_{i+1}\mid 1\leq i\leq r-1\}\cup\{2e_{r}\}$.
In Table \ref{table.localtable.2},
we list up the set of all the possible local orbit types
of the hyperbolic orbits for the s-representations
associated with
all classical-type semisimple pseudo-Riemannian symmetric spaces
satisfying $(\Delta,\Delta^{a})=(C_{r}, D_{r})$.
\end{Ex}

\begin{table}[htbp]
\footnotesize
\caption{Local orbit types}\label{table.localtable.2}
\begin{flushleft}
$\Delta=C_{r},\,\Delta^{a}=D_{r},\,\varPsi=\{e_{i}-e_{i+1}\mid 1 \leq i \leq r-1\}\cup \{2e_{r}\}$
\end{flushleft}
\begin{center}
\begin{tabular}{|>{\PBS\centering}p{\LENGTHTHETA}|>{\PBS\centering}p{\LENGTHH}|}
\multicolumn{2}{l}{
$(\MF{g},\MF{h})=(\MF{sl}(2n,\BS{C}),\MF{sl}(2n,\BS{R}))$
}\\
\hline
$\Theta(\subset \varPsi)$ &
$\varPsi\setminus\{e_{i_{1}}-e_{i_{1}+1},\ldots,
e_{i_{k-1}}-e_{i_{k-1}+1},2e_{n}
\}$
\\\hline
$\MF{h}_{\Theta}$
&
$\BS{R}^{k-1}+\MF{so}(2)^{k}+\displaystyle{\sum^{k}_{l=1}}\MF{sl}(i_{l}-i_{l-1},\BS{C})$\\\hline
Remarks
&$0=i_{0}<i_{1}<\dots<i_{k}= n$\\
\hline
\hline
$\Theta(\subset \varPsi)$ &
$\varPsi\setminus\{e_{i_{1}}-e_{i_{1}+1},\ldots,e_{i_{k}}-e_{i_{k}+1}\}$
\\\hline
$\MF{h}_{\Theta}$ &
$\BS{R}^{k}+\MF{so}(2)^{k}+\displaystyle{\sum^{k}_{l=1}}\MF{sl}(i_{l}-i_{l-1},\BS{C})+\MF{sl}(2(n-i_{k}),\BS{R})$
\\\hline
Remarks
&$0=i_{0}<i_{1}<\dots<i_{k}<n$\\
\hline
\hline
$\Theta(\subset s_{2e_{n}}\varPsi)$ &
$s_{2e_{n}}\left(\varPsi\setminus\{e_{i_{1}}-e_{i_{1}+1},\ldots,
e_{i_{k-1}}-e_{i_{k-1}+1},2e_{n}
\}\right)$
\\\hline
$\MF{h}_{\Theta}$
&
$\BS{R}^{k-1}+\MF{so}(2)^{k}+\displaystyle{\sum^{k}_{l=1}}\MF{sl}(i_{l}-i_{l-1},\BS{C})$\\\hline
Remarks&
$0=i_{0}<i_{1}<\dots<i_{k}= n$\\
\hline
\hline
$\Theta(\subset s_{2e_{n}}\varPsi)$ &
$s_{2e_{n}}\left(\varPsi\setminus\{e_{i_{1}}-e_{i_{1}+1},\ldots,e_{i_{k}}-e_{i_{k}+1}\}\right)$
\\\hline
$\MF{h}_{\Theta}$
&
$\BS{R}^{k}+\MF{so}(2)^{k}+\displaystyle{\sum^{k}_{l=1}}\MF{sl}(i_{l}-i_{l-1},\BS{C})+\MF{sl}(2(n-i_{k}),\BS{R})$
\\\hline
Remarks
&$0=i_{0}<i_{1}<\dots<i_{k}<n$\\
\hline
\end{tabular}
\end{center}
\begin{center}
\begin{tabular}{|>{\PBS\centering}p{\LENGTHTHETA}|>{\PBS\centering}p{\LENGTHH}|}
\multicolumn{2}{l}{
$(\MF{g},\MF{h})=(\MF{su}(n,n),\MF{so}(n,n))$
}\\
\hline
$\Theta(\subset \varPsi)$ &
$\varPsi\setminus\{e_{i_{1}}-e_{i_{1}+1},\ldots, e_{i_{k-1}}-e_{i_{k-1}+1},2e_{n}\}$
\\\hline
$\MF{h}_{\Theta}$
& 
$\displaystyle{\sum_{l=1}^{k}\MF{so}(i_{l}-i_{l-1},\BS{C})}$
\\\hline
Remarks
&
$0=i_{0}<i_{1}<\cdots<i_{k} = n$
\\
\hline
\hline
$\Theta(\subset \varPsi)$&
$\varPsi\setminus\{e_{i_{1}}-e_{i_{1}+1},\ldots,e_{i_{k}}-e_{i_{k}+1}\}$
\\\hline
$\MF{h}_{\Theta}$
& 
$\displaystyle{\sum_{l=1}^{k}\MF{so}(i_{l}-i_{l-1},\BS{C})+\MF{so}(n-i_{k},n-i_{k})}$
\\\hline
Remarks
&
$0=i_{0}<i_{1}<\cdots<i_{k} < n$
\\
\hline
\hline
$\Theta(\subset s_{2e_{n}}\varPsi)$ &
$s_{2e_{n}}\left(\varPsi\setminus\{e_{i_{1}}-e_{i_{1}+1},\ldots,
e_{i_{k-1}}-e_{i_{k-1}+1},2e_{n}
\}\right)$
\\\hline
$\MF{h}_{\Theta}$
&
$\displaystyle{\sum_{l=1}^{k}\MF{so}(i_{l}-i_{l-1},\BS{C})}$
\\\hline
Remarks
&$0=i_{0}<i_{1}<\dots<i_{k}= n$\\
\hline
\hline
$\Theta(\subset s_{2e_{n}}\varPsi)$ &
$s_{2e_{n}}\left(\varPsi\setminus\{e_{i_{1}}-e_{i_{1}+1},\ldots,e_{i_{k}}-e_{i_{k}+1}\}\right)$
\\\hline
$\MF{h}_{\Theta}$
&
$\displaystyle{\sum_{l=1}^{k}\MF{so}(i_{l}-i_{l-1},\BS{C})+\MF{so}(n-i_{k},n-i_{k})}$
\\\hline
Remarks
&$0=i_{0}<i_{1}<\dots<i_{k}<n$\\
\hline
\end{tabular}
\end{center}
\end{table}

\begin{table}[htbp]
\footnotesize
\contcaption{(continued)}
\begin{center}
\begin{tabular}{|>{\PBS\centering}p{\LENGTHTHETA}|>{\PBS\centering}p{\LENGTHH}|}
\multicolumn{2}{l}{
$(\MF{g},\MF{h})=(\MF{sl}(2n,\BS{R}),\MF{sl}(n,\BS{R})+\MF{sl}(n,\BS{R})+\BS{R})$
}\\
\hline
$\Theta(\subset \varPsi)$ &
$\varPsi\setminus\{e_{i_{1}}-e_{i_{1}+1},\ldots,
e_{i_{k-1}}-e_{i_{k-1}+1},2e_{n}
\}$
\\\hline
$\MF{h}_{\Theta}$& 
$\BS{R}^{k-1}+\displaystyle{\sum^{k}_{l=1}}\MF{sl}(i_{l}-i_{l-1},\BS{R})$
\\\hline
Remarks&
$0=i_{0}<i_{1}<\cdots<i_{k}=n$
\\
\hline
\hline
$\Theta(\subset \varPsi)$ &
$\varPsi\setminus\{e_{i_{1}}-e_{i_{1}+1},\ldots,e_{i_{k}}-e_{i_{k}+1}\}$
\\\hline
$\MF{h}_{\Theta}$&
$\BS{R}^{k}+\displaystyle{\sum^{k}_{l=1}}\MF{sl}(i_{l}-i_{l-1},\BS{R})+\MF{sl}(n-i_{k},\BS{R})^{2}+\BS{R}$
\\\hline
Remarks
&
$0=i_{0}<i_{1}<\cdots<i_{k}<n$
\\
\hline
\hline
$\Theta(\subset s_{2e_{n}}\varPsi)$ &
$s_{2e_{n}}\left(\varPsi\setminus\{e_{i_{1}}-e_{i_{1}+1},\ldots,
e_{i_{k-1}}-e_{i_{k-1}+1},2e_{n}
\}\right)$
\\\hline
$\MF{h}_{\Theta}$
& 
$\BS{R}^{k-1}+\displaystyle{\sum^{k}_{l=1}}\MF{sl}(i_{l}-i_{l-1},\BS{R})$
\\\hline
Remarks
&
$0=i_{0}<i_{1}<\cdots<i_{k}=n$
\\
\hline
\hline
$\Theta(\subset s_{2e_{n}}\varPsi)$ &
$s_{2e_{n}}\left(\varPsi\setminus\{e_{i_{1}}-e_{i_{1}+1},\ldots,e_{i_{k}}-e_{i_{k}+1}\}\right)$
\\\hline
$\MF{h}_{\Theta}$
&
$\BS{R}^{k}+\displaystyle{\sum^{k}_{l=1}}\MF{sl}(i_{l}-i_{l-1},\BS{R})+\MF{sl}(n-i_{k},\BS{R})^{2}+\BS{R}$
\\\hline
Remarks
&
$0=i_{0}<i_{1}<\cdots<i_{k}<n$
\\
\hline
\end{tabular}
\end{center}
\begin{center}
\begin{tabular}{|>{\PBS\centering}p{\LENGTHTHETA}|>{\PBS\centering}p{\LENGTHH}|}
\multicolumn{2}{l}{
$(\MF{g},\MF{h})=(\MF{so}^{*}(4n),\MF{so}(2n,\BS{C}))$
}\\
\hline
$\Theta(\subset \varPsi)$ &
$\varPsi\setminus\{e_{i_{1}}-e_{i_{1}+1},\ldots,e_{i_{k-1}}-e_{i_{k-1}+1},2e_{n}\}$
\\\hline
$\MF{h}_{\Theta}$& 
$\displaystyle{\sum^{k}_{l=1}}\MF{so}^{*}(2(i_{l}-i_{l-1}))$
\\\hline
Remarks
&
$0=i_{0}<i_{1}<\cdots<i_{k}= n$
\\
\hline
\hline
$\Theta(\subset \varPsi)$ &
$\varPsi\setminus\{e_{i_{1}}-e_{i_{1}+1},\ldots,e_{i_{k}}-e_{i_{k}+1}\}$
\\\hline
$\MF{h}_{\Theta}$& 
$\displaystyle{\sum^{k}_{l=1}}\MF{so}^{*}(2(i_{l}-i_{l-1}))+\MF{so}(2(n-i_{k}),\BS{C})$
\\\hline
Remarks&
$0=i_{0}<i_{1}<\cdots<i_{k}= n$
\\
\hline
\hline
$\Theta(\subset s_{2e_{n}}\varPsi)$ &
$s_{2e_{n}}\left(\varPsi\setminus\{e_{i_{1}}-e_{i_{1}+1},\ldots,e_{i_{k-1}}-e_{i_{k-1}+1},2e_{n}\}\right)$
\\\hline
$\MF{h}_{\Theta}$& 
$\displaystyle{\sum^{k}_{l=1}}\MF{so}^{*}(2(i_{l}-i_{l-1}))$
\\\hline
Remarks
&
$0=i_{0}<i_{1}<\cdots<i_{k}= n$
\\
\hline
\hline
$\Theta(\subset s_{2e_{n}}\varPsi)$ &
$s_{2e_{n}}\left(\varPsi\setminus\{e_{i_{1}}-e_{i_{1}+1},\ldots,e_{i_{k}}-e_{i_{k}+1}\}\right)$
\\\hline
$\MF{h}_{\Theta}$
& 
$\displaystyle{\sum^{k}_{l=1}}\MF{so}^{*}(2(i_{l}-i_{l-1}))+\MF{so}(2(n-i_{k}),\BS{C})$
\\\hline
Remarks
&
$0=i_{0}<i_{1}<\cdots<i_{k}= n$
\\
\hline
\end{tabular}
\end{center}
\vspace{10mm}
\end{table}

\begin{table}[htbp]
\footnotesize
\contcaption{(continued)}
\begin{center}
\begin{tabular}{|>{\PBS\centering}p{\LENGTHTHETA}|>{\PBS\centering}p{\LENGTHH}|}
\multicolumn{2}{l}{
$(\MF{g},\MF{h})=(\MF{so}(2n,2n),\MF{sl}(2n,\BS{R})+\BS{R})$
}\\
\hline
$\Theta(\subset \varPsi)$ &
$\varPsi\setminus\{e_{i_{1}}-e_{i_{1}+1},\ldots,
e_{i_{k-1}}-e_{i_{k-1}+1},2e_{n}\}$
\\\hline
$\MF{h}_{\Theta}$
& 
$\displaystyle{\sum^{k}_{l=1}}\MF{sp}(i_{l}-i_{l-1},\BS{R})$
\\\hline
Remarks
&
$0=i_{0}<i_{1}<\cdots<i_{k}=n$
\\
\hline
\hline
$\Theta(\subset \varPsi)$ &
$\varPsi\setminus\{e_{i_{1}}-e_{i_{1}+1},\ldots,e_{i_{k}}-e_{i_{k}+1}\}$
\\\hline
$\MF{h}_{\Theta}$
&
$\displaystyle{\sum^{k}_{l=1}}\MF{sp}(i_{l}-i_{l-1},\BS{R})+\MF{sl}(2(n-i_{k}),\BS{R})+\BS{R}$
\\\hline
Remarks
&
$0=i_{0}<i_{1}<\cdots<i_{k}<n$
\\
\hline
\hline
$\Theta(\subset s_{2e_{n}}\varPsi)$ &
$s_{2e_{n}}\left(\varPsi\setminus\{e_{i_{1}}-e_{i_{1}+1},\ldots,
e_{i_{k-1}}-e_{i_{k-1}+1},2e_{n}\}\right)$
\\\hline
$\MF{h}_{\Theta}$
& 
$\displaystyle{\sum^{k}_{l=1}}\MF{sp}(i_{l}-i_{l-1},\BS{R})$
\\\hline
Remarks
&
$0=i_{0}<i_{1}<\cdots<i_{k}=n$
\\\hline
\hline
$\Theta(\subset s_{2e_{n}}\varPsi)$ &
$s_{2e_{n}}\left(\varPsi\setminus\{e_{i_{1}}-e_{i_{1}+1},\ldots,e_{i_{k}}-e_{i_{k}+1}\}\right)$
\\\hline
$\MF{h}_{\Theta}$
&
$\displaystyle{\sum^{k}_{l=1}}\MF{sp}(i_{l}-i_{l-1},\BS{R})+\MF{sl}(2(n-i_{k}),\BS{R})+\BS{R}$
\\\hline
Remarks
&
$0=i_{0}<i_{1}<\cdots<i_{k}<n$
\\
\hline
\end{tabular}
\end{center}
\end{table}


\section{Local orbit types of the elliptic orbits}\label{Sec.ellip}

Let $(\MF{g},\MF{h})$ be a semisimple
symmetric pair and $\sigma$ be an involution of $\MF{g}$
with $\MF{h} = \OPE{Ker}(\sigma - \OPE{id})$.
Set $\MF{g}^{c}:=\MF{h}+\sqrt{-1}\MF{q} (\subset \MF{g}^{\BS{C}})$.
The symmetric pair $(\MF{g}^{c},\MF{h})$ is called 
the $c$-dual pair of $(\MF{g},\MF{h})$.
In this section, we discuss the relation between the isotropy subalgebras of elliptic orbits for the s-representation associated with $(\MF{g}, \MF{h})$ and
those of hyperbolic orbits for the s-representation associated with $(\MF{g}^{c},\MF{h})$.
Then we have the following fact.

\begin{Lem}\label{lem.tro1}
For any elliptic element $X \in \MF{q}$,
the centralizer of $X$ in $\MF{h}(\subset \MF{g})$ coincides with that of $\sqrt{-1}X$ in $\MF{h}(\subset \MF{g}^{c})$.
\end{Lem}

\noindent
From Lemma \ref{lem.tro1} we can determine the local orbit types of elliptic orbits for the s-representation of $(\MF{g},\MF{h})$ by investigating those of hyperbolic orbits for the s-representation of $(\MF{g}^{c},\MF{h})$.
For example, in the case of $(\MF{g},\MF{h})=(\MF{sl}(n,\BS{R})+\MF{sl}(n,\BS{R}),\MF{sl}(n,\BS{R}))$,
we have $(\MF{g}^{c},\MF{h})=(\MF{sl}(n,\BS{C}),\MF{sl}(n,\BS{R}))$.
Then it follows from the above argument that the elliptic principal isotropy subalgebra for the s-representation associated with $(\MF{sl}(n,\BS{R})+\MF{sl}(n,\BS{R}),\MF{sl}(n,\BS{R}))$ coincides with $\BS{R}^{[(n-1)/2]}+\MF{so}(2)^{[n/2]}$ for any $n \in \BS{N}$ (see, Table \ref{table.hprin}).
In the case of $n=4$,
this result was shown by Boumuki (PROPOSITION 5.1 in \cite{MR2370009}).
He actually determined all the isotropy subalgebras of elliptic orbits
for the s-representation associated with $(\MF{sl}(4,\BS{R})+\MF{sl}(4,\BS{R}),\MF{sl}(4,\BS{R}))$.
Our method of determining the isotropy subalgebras depends on restricted root system theory for semisimple symmetric pairs,
and is different from Boumuki's method, which depends on root system theory for semisimple complex Lie algebras and that for compact Lie algebras.

\medskip

\noindent
{\bf ACKNOWLEDGEMENTS.} The author would like to express his sincere gratitude
to Professor Naoyuki Koike for valuable discussions and valuable comments.

\vspace{\baselineskip}

\begin{flushleft}
{\scriptsize
Kurando BABA\\
Department of General Education,
Fukushima National College of Technology,\\
30 Nagao, Kamiarakawa, Taira, Iwaki, Fukushima 970-8034, Japan\\
E-mail: baba@fukushima-nct.ac.jp
}
\end{flushleft}

\end{document}